\documentclass{amsart}
\usepackage{amsmath,amssymb,amsthm,amscd,latexsym}
\usepackage[dvipdfmx]{graphicx}
\usepackage[all,cmtip]{xy}
\usepackage{wrapfig}
\usepackage{color}
\usepackage{enumerate}
\usepackage{cases}
\usepackage[top=30mm,bottom=30mm,left=30mm,right=30mm]{geometry}
\usepackage{graphics}
\usepackage{tikz}
\usepackage{pxpgfmark}
\usepackage{multirow}
\usepackage{extarrows}
\usepackage{comment}

\usetikzlibrary{patterns}
\usetikzlibrary{intersections, calc,decorations.markings}
    
\numberwithin{equation}{section}
\theoremstyle{definition}
\newtheorem{example}{Example}[section]
\newtheorem{defn}[example]{Definition}

\newtheorem{remark}[example]{Remark}
\newtheorem{defprop}[example]{Definition-Proposition}
\newtheorem{propdef}[example]{Proposition-Definition}

\theoremstyle{plain}
\newtheorem{lem}[example]{Lemma}
\newtheorem{thm}[example]{Theorem}
\newtheorem{prop}[example]{Proposition}
\newtheorem{cor}[example]{Corollary}
\newtheorem{conjecture}[example]{Conjecture}

\newlength\squareheight
\DeclareMathOperator{\bAL}{\mathbb{AL}}
\DeclareMathOperator{\bM}{\mathbb{M}}

\DeclareMathOperator{\bQ}{\mathbb{Q}}
\DeclareMathOperator{\bR}{\mathbb{R}}

\DeclareMathOperator{\bZ}{\mathbb{Z}}

\DeclareMathOperator{\cA}{\mathcal{A}}
\DeclareMathOperator{\cC}{\mathcal{C}}
\DeclareMathOperator{\cF}{\mathcal{F}}
\DeclareMathOperator{\cG}{\mathcal{G}}
\DeclareMathOperator{\cM}{\mathcal{M}}

\DeclareMathOperator{\cS}{\mathcal{S}}
\DeclareMathOperator{\cT}{\mathcal{T}}

\DeclareMathOperator{\be}{\mathbf{e}}

\DeclareMathOperator{\bv}{\mathbf{v}}

\DeclareMathOperator{\cl}{\mathsf{c}}

\DeclareMathOperator{\sC}{\mathsf{C}}
\DeclareMathOperator{\sM}{\mathsf{M}}

\DeclareMathOperator{\sX}{\mathsf{X}}

\renewcommand{\mod}{{\rm mod}}

\DeclareMathOperator{\add}{{\rm add}}
\DeclareMathOperator{\Cok}{{\rm Cok}}
\DeclareMathOperator{\End}{{\rm End}}
\DeclareMathOperator{\Fac}{{\rm Fac}}
\DeclareMathOperator{\Hom}{{\rm Hom}}
\DeclareMathOperator{\Int}{{\rm Int}}

\DeclareMathOperator{\proj}{{\rm proj}}

\DeclareMathOperator{\trigid}{{\rm \tau-rigid}}
\DeclareMathOperator{\itrigid}{{\rm i\tau-rigid}}
\DeclareMathOperator{\sttilt}{{\rm s\tau-tilt}}
\DeclareMathOperator{\ctilt}{{\rm c-tilt}}
\DeclareMathOperator{\rigid}{{\rm rigid}}
\DeclareMathOperator{\irigid}{{\rm irigid}}
\DeclareMathOperator{\rctilt}{{\rm c-tilt^{\Sigma}}}
\DeclareMathOperator{\rrigid}{{\rm rigid^{\Sigma}}}
\DeclareMathOperator{\ririgid}{{\rm irigid^{\Sigma}}}

\DeclareMathOperator{\udim}{\underline{\rm dim}}

\title[Dimension vectors and $f$-vectors from triangulated surfaces]
{Dimension vectors of $\tau$-rigid modules and $f$-vectors of cluster monomials from triangulated surfaces}
\author{Toshiya Yurikusa}
\address{Laboratoire de Math\'{e}matiques de Versailles, UVSQ, CNRS, Universit\'{e} Paris-Saclay, 78035 Versailles, France and Mathematical Institute, Tohoku University, Sendai 980-8578, Japan}
\email{toshiya.yurikusa.d8@tohoku.ac.jp}

\begin{document}

\begin{abstract}
For the cluster algebra $\mathcal{A}$ associated with a triangulated surface, we give a characterization of the triangulated surface such that different non-initial cluster monomials in $\mathcal{A}$ have different $f$-vectors. Similarly, for the associated Jacobian algebra $J$, we give a characterization of the triangulated surface such that different $\tau$-rigid $J$-modules have different dimension vectors. Moreover, we also show that different basic support $\tau$-tilting $J$-modules have different dimension vectors. Our main ingredient is a notion of intersection numbers defined by Qiu and Zhou. As an application, we show that the denominator conjecture holds for $\mathcal{A}$ if the marked surface is a closed surface with exactly one puncture, or the given tagged triangulation has neither loops nor tagged arcs connecting punctures.
\end{abstract}
\maketitle

\section{Introduction}\label{sec:intro}

Cluster algebras \cite{FZ02} are commutative algebras with generators called cluster variables. The certain tuples of cluster variables are called clusters and they have combinatorial structures called mutations. Their original motivation was to study total positivity of semisimple Lie groups and canonical bases of quantum groups.  In recent years, cluster algebras have interacted with various subjects in mathematics, for example, representation theory of quivers, Poisson geometry, integrable systems, and so on.

In a cluster algebra with principal coefficients, by Laurent phenomenon (Proposition \ref{prop:Laurent phenomenon}), every nonzero element $x$ is expressed by a Laurent polynomial of the initial cluster variables $(x_1,\ldots,x_n)$ and coefficients $(y_1,\ldots,y_n)$
\[
x=\frac{F(x_1, \ldots,x_n,y_1,\ldots,y_n)}{x_1^{d_1} \cdots x_n^{d_n}},
\]
where $d_i\in\bZ$ and $F(x_1, \ldots, x_n,y_1,\ldots,y_n)\in\bZ[x_1,\ldots, x_n,y_1,\ldots,y_n]$ is not divisible by any $x_i$. We call $d(x):=(d_i)_{1 \le i \le n}$ the denominator vector of $x$. For the maximal degree $f_i$ of $y_i$ in the $F$-polynomial $F(1,\ldots,1,y_1,\ldots,y_n)$ of $x$, we call $f(x):=(f_i)_{1 \le i \le n}$ the $f$-vector of $x$.

A monomial in cluster variables belonging to the same cluster is called a cluster monomial. The following conjecture is known as the denominator conjecture in cluster algebra theory.

\begin{conjecture}[{\cite[Conjecture 4.17]{FZ03}}]\label{conj:denominator}
For any cluster monomials $x$ and $x'$ in a cluster algebra, if $d(x)=d(x')$, then $x=x'$.
\end{conjecture}

Conjecture \ref{conj:denominator} was proved for cluster algebras of rank $2$ \cite{SZ04}, of finite type \cite{FG22,FG24b}, acyclic cluster algebras with respect to an acyclic initial seed \cite{RS20} (see also \cite{CK06,CK08,FZ07}), and cluster algebras associated with triangulated surfaces with respect to certain initial seed \cite{FG24a} (see the remark after Theorem 1.7). Recently, Fei \cite{F} gave its counterexample (Example \ref{ex:counterexample}). It will give rise to a natural question: For which cluster algebras, does Conjecture \ref{conj:denominator} hold?

One of our motivations is to study the question and an analogue of Conjecture \ref{conj:denominator} for $f$-vectors. Since all $f$-vectors of initial cluster variables are zero, we restrict to non-initial cluster monomials, that are cluster monomials in non-initial cluster variables. In an acyclic skew-symmetric cluster algebra with respect to an acyclic initial seed, it immediately follows from \cite{CK06,FK10} that different non-initial cluster monomials have different $f$-vectors. However, it does not hold in general. We study it for cluster algebras associated with triangulated surfaces that were developed in \cite{FG06, FG09, FoST08, FT18, GSV05}. One of our results gives a characterization of the triangulated surface such that different non-initial cluster monomials in the associated cluster algebra have different $f$-vectors (Theorem \ref{thm:unique f-vector}). As an application, we show that Conjecture \ref{conj:denominator} holds if the marked surface is a closed surface with exactly one puncture, or the given tagged triangulation has neither loops nor tagged arcs connecting punctures (Theorem \ref{thm:conj}).

Our other motivation comes from $\tau$-tilting theory \cite{AIR14}. In representation theory of finite dimensional algebras $\Lambda$, one of important problems is to describe the isomorhism classes of $\Lambda$-modules with the same dimension vector. From this point of view, restricting to $\tau$-rigid modules, we consider when different $\tau$-rigid $\Lambda$-modules have different dimension vectors. Via a categorification of cluster algebras, this problem is closely related to the above problem for $f$-vectors (see Section \ref{sec:tau}). As above, we consider this problem for certain Jacobian algebras $J$ associated with triangulated surfaces. We give a characterization of the triangulated surface such that different $\tau$-rigid $J$-modules have different dimension vectors (Theorem \ref{thm:unique rigid}), and we also show that different basic support $\tau$-tilting $J$-modules have different dimension vectors (Theorem \ref{thm:unique sttilt}).

Remark that there are some weak or related results for the above problems as follows:
\begin{itemize}
\item Different cluster variables in a finite type or affine type cluster algebra have different denominator vectors. Moreover, different indecomposable $\tau$-rigid modules over the associated Jacobian algebra have different dimension vectors \cite{FG19,GP12,R11}.
\item Different non-initial cluster variables in the cluster algebra associated with a triangulated surface have different $f$-vectors if and only if the given tagged triangulation does not have two tagged arcs connecting two (possibly same) common punctures such that the underlying curves are different \cite{GY20}.
\item Different $\tau$-rigid modules over a finite dimensional gentle algebra $\Lambda$ have different dimension vectors if and only if the quiver of $\Lambda$ does not admit an oriented cycle of even length with full relations \cite{FG22}.
\end{itemize}
In particular, our results for intersection numbers and $f$-vectors in Subsections \ref{subsec:result cluster} and \ref{subsec:result tau} can be seen as a natural generalization of ones in \cite{GY20}.

This paper is organized as follows: In the rest of this section, we give results of this paper. Since our main ingredient is a notion of intersection numbers defined in \cite{QZ17}, we first state our results in terms of intersection numbers. After that, we restate them in terms of $f$-vectors in cluster algebra theory and dimension vectors in $\tau$-tilting theory. In Section \ref{sec:cS}, we are devoted to studying on triangulated surfaces and prove theorems of the next subsection. Theorem \ref{thm:segment} plays an important role to prove them, and its proof will be given in Section \ref{sec:modif}. To prove Theorem \ref{thm:segment}, we introduce and study a notion of modifications for multi-sets of certain curves. In Section \ref{sec:cluster}, we recall cluster algebras associated with triangulated surfaces and prove our results in Subsection \ref{subsec:result cluster} for $f$-vectors. In Section \ref{sec:tau}, we recall $\tau$-tilting theory and cluster tilting theory. Via a categorification of cluster algebras, we prove our results in Subsection \ref{subsec:result tau} for dimension vectors.

\medskip\noindent{\bf Notations}.
In this paper, we assume that all sets and multi-sets are finite. A set, of course, means a non-multi-set. For a multi-set
\[
S=\{s_1,\ldots,s_1,s_2,\dots,s_{n-1},s_n,\ldots,s_n\},
\]
we denote by $m_S(s)$ the multiplicity of an element $s$, where $m_S(s)=0$ if $s\notin S$, and we represent it as
\[
S=\{s_1^{m_S(s_1)},\ldots,s_n^{m_S(s_n)}\}.
\]
Moreover, for $(r_{s_1},\ldots,r_{s_n})\in\bR^n$, we define a sum over the multi-set $S$ as follows:
\[
\sum_{s\in S}r_s:=\sum_{i=1}^nm_S(s_i)r_{s_i}.
\]
For multi-sets $S_1,\ldots,S_m$, their sum (or disjoint union) is a multi-set such that the multiplicity of each element $s$ is given by $m_{S_1}(s)+\cdots+m_{S_m}(s)$. We denote it by
\[
\bigsqcup_{i=1}^mS_i.
\]

\subsection{Our results on triangulated surfaces}\label{subsec:result}

We refer to Section \ref{sec:cS} for the details on this subsection. Let $\cS$ be a marked surface. A \emph{punctured loop} is a loop whose both ends are tagged in the same way such that it cuts out a monogon with exactly one puncture (see Figure \ref{fig:punctured loop}). Note that punctured loops are not tagged arcs. To a pair of conjugate arcs, we associate a punctured loop as in Figure \ref{fig:punctured loop}.

%
%
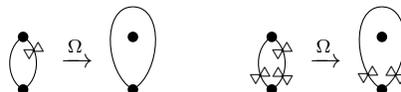
\begin{figure}[ht]
\[
\begin{tikzpicture}[baseline=2mm]
\coordinate(u)at(90:1);\coordinate(p)at(0,0.6);\coordinate(d)at(0,-0.1);
\draw(d)to[out=150,in=-150](p);\draw(d)to[out=30,in=-30]node[pos=0.7]{\rotatebox{30}{\footnotesize $\bowtie$}}(p);
\fill(p)circle(0.07);\fill(d)circle(0.07);
\end{tikzpicture}
\xrightarrow{\Omega}
\begin{tikzpicture}[baseline=2mm]
\coordinate(u)at(90:1);\coordinate(p)at(0,0.6);\coordinate(d)at(0,-0.1);
\draw(d)to[out=150,in=180](u);\draw(d)to[out=30,in=0](u);
\fill(p)circle(0.07);\fill(d)circle(0.07);
\end{tikzpicture}
\hspace{10mm}
\begin{tikzpicture}[baseline=2mm]
\coordinate(u)at(90:1);\coordinate(p)at(0,0.6);\coordinate(d)at(0,-0.1);
\draw(d)to[out=150,in=-150]node[pos=0.3]{\rotatebox{30}{\footnotesize $\bowtie$}}(p);
\draw(d)to[out=30,in=-30]node[pos=0.3]{\rotatebox{150}{\footnotesize $\bowtie$}}node[pos=0.7]{\rotatebox{30}{\footnotesize $\bowtie$}}(p);
\fill(p)circle(0.07);\fill(d)circle(0.07);
\end{tikzpicture}
\xrightarrow{\Omega}
\begin{tikzpicture}[baseline=2mm]
\coordinate(u)at(90:1);\coordinate(p)at(0,0.6);\coordinate(d)at(0,-0.1);
\draw(d)to[out=150,in=180]node[pos=0.2]{\rotatebox{30}{\footnotesize $\bowtie$}}(u);
\draw(d)to[out=30,in=0]node[pos=0.2]{\rotatebox{150}{\footnotesize $\bowtie$}}(u);
\fill(p)circle(0.07);\fill(d)circle(0.07);
\end{tikzpicture}
\]
\caption{Punctured loops associated with pairs of conjugate arcs}
\label{fig:punctured loop}
\end{figure}

We denote by $\bM_{\cS}$ the set of all multi-sets of pairwise compatible tagged arcs in $\cS$. For $U\in\bM_{\cS}$, we denote by $\Omega(U)$ the multi-set obtained from $U$ by replacing a maximal set of disjoint pairs of conjugate arcs with the corresponding punctured loops (Definition-Proposition \ref{defprop:Omega}).

Let $T$ be a tagged triangulation of $\cS$. We assume that all tags in $\Omega(T)$ are plain (see Subsection \ref{subsec:Int}). For $U\in\bM_{\cS}$ with $U\cap T=\emptyset$, we consider a decomposition
\[
\Omega(U)=U_1\sqcup U_2,
\]
where $U_2$ consists of all $2$-notched curves in $\Omega(U)$ whose underlying plain curves are in $\Omega(T)$. For each puncture $p$, we denote by $n(U_2,p)$ the number of notched tags in $U_2$ incident to $p$.

On the other hand, we consider the intersection vector of $U\in\bM_{\cS}$ with respect to $T$
\[
\Int_T(U):=\left(\Int(t,U)\right)_{t\in T}:=\left(\sum_{u\in U}\Int(t,u)\right)_{t\in T}\in\bZ_{\ge 0}^T,
\]
where $\Int(t,u)$ is an intersection number defined in \cite{QZ17} (see Definition \ref{def:Int}). We are ready to state the main result in this paper.

\begin{thm}\label{thm:main}
Let $T$ be a tagged triangulation of $\cS$, and $U,V\in\bM_{\cS}$ with $U\cap T=V\cap T=\emptyset$. If $\Int_T(U)=\Int_T(V)$, then $U_1=V_1$ and $n(U_2,p)=n(V_2,p)$ for all punctures $p$.
\end{thm}

In Theorem \ref{thm:main}, $U$ and $V$ do not coincide in general. We give a sufficient condition of $U$ and $V$ such that they coincide, and a necessary and sufficient condition of $T$ such that they always coincide. Let $G_T$ be a (multi-)graph whose vertices are punctures in $\cS$ incident to $\Omega(T)$, and whose edges are tagged arcs and punctured loops in $\Omega(T)$ connecting punctures.

\begin{thm}\label{thm:unique tag tri}
Let $T$, $U$, and $V$ be tagged triangulations of $\cS$. If $\Int_T(U)=\Int_T(V)$, then $U=V$.
\end{thm}

\begin{thm}\label{thm:unique Int}
Let $T$ be a tagged triangulation of $\cS$. Then the following are equivalent:
\begin{itemize}
\item[(1)] For any $U,V\in\bM_{\cS}$ with $U\cap T=V\cap T=\emptyset$, if $\Int_T(U)=\Int_T(V)$, then $U=V$.
\item[(2)] Each connected component of $G_T$ contains at most one cycle of odd length and no cycles of even length.
\end{itemize}
\end{thm}

Moreover, we can give a complete list of marked surfaces that have tagged triangulations satisfying the equivalent properties in Theorem \ref{thm:unique Int}.

\begin{thm}\label{thm:list cS}
Let $\cS'$ be a marked surface.
\begin{itemize}
\item[(1)] There is at least one tagged triangulation of $\cS'$ that satisfies the equivalent properties in Theorem \ref{thm:unique Int} if and only if the boundary of $\cS'$ is not empty.
\item[(2)] All tagged triangulations of $\cS'$ satisfy the equivalent properties in Theorem \ref{thm:unique Int} if and only if $\cS'$ is one of the following:
\begin{itemize}
\item A polygon with at most two punctures.
\item An annulus with at most one puncture.
\item A marked surface with no punctures.
\end{itemize}
\end{itemize}
\end{thm}

\subsection{Our results on cluster algebra theory}\label{subsec:result cluster}

We refer to Section \ref{sec:cluster} for the details on this subsection. To a tagged triangulation $T$ of $\cS$, one can associate a quiver $Q_T$ and a cluster algebra $\cA(Q_T)$. In $\cA(Q_T)$, the $f$-vector of a cluster monomial coincides with the intersection vector of some multi-set in $\bM_{\cS}$ with respect to $T$. Moreover, if $\cS$ is a closed surface with exactly one puncture, then different non-initial cluster monomials in $\cA(Q_{T'})$ have different $f$-vectors for any tagged triangulation $T'$ of $\cS$ (Proposition \ref{prop:1-punc f}). As a consequence, we can obtain analogues of the theorems in Subsection \ref{subsec:result} for $f$-vectors of non-initial cluster monomials. In particular, Theorem \ref{thm:unique Int} induces a desired characterization as follows.

\begin{thm}\label{thm:unique f-vector}
Let $T$ be a tagged triangulation of $\cS$. Then the following are equivalent:
\begin{itemize}
\item[(1)] For any non-initial cluster monomials $x$ and $x'$ in $\cA(Q_T)$, if $f(x)=f(x')$, then $x=x'$.
\item[(2)] The marked surface $\cS$ is a closed surface with exactly one puncture, or each connected component of $G(T)$ contains at most one cycle of odd length and no cycles of even length.
\end{itemize}
\end{thm}

As an application, we give a sufficient condition that Conjecture \ref{conj:denominator} holds.

\begin{thm}\label{thm:conj}
Let $T$ be a tagged triangulation of $\cS$. Then Conjecture \ref{conj:denominator} holds for $\cA(Q_T)$ if $\cS$ is a closed surface with exactly one puncture, or $T$ has neither loops nor tagged arcs connecting punctures.
\end{thm}

Remark that Fu and Geng independently showed that Conjecture \ref{conj:denominator} holds for $\cA(Q_T)$ if $T$ has no loops and each puncture is enclosed by a punctured loop in $\Omega(T)$ \cite[Theorem 5.1]{FG24a}, or $\cS$ is a polygon with exactly one puncture \cite[Theorem 4.2]{FG24b}. Their proof uses a notion of ``intersection numbers'' defined in \cite{FoST08} while our proof uses a notion of ``intersection numbers'' defined in \cite{QZ17}. As a consequence of Theorem \ref{thm:conj}, we can obtain \cite[Theorem 5.1]{FG24a} and \cite[Theorem 4.2]{FG24b}.

\subsection{Our results on $\tau$-tilting theory}\label{subsec:result tau}

We refer to Section \ref{sec:tau} for the details on this subsection. Let $T$ be a tagged triangulation of $\cS$. If $\cS$ is not a closed surface with exactly one puncture, then we take any non-degenerate potential of $Q_T$. If $\cS$ is a closed surface with exactly one puncture, then we take a non-degenerate potential $W_T^{\lambda,n}$ of $Q_T$ defined in \cite{GLM22,Lab09} (see \eqref{eq:potential}). In both cases, the associated Jacobian algebra $J_T$ is finite dimensional.

Via a categorification of cluster algebras, the dimension vector of a $\tau$-rigid $J_T$-module coincide with the intersection vector of some multi-set in $\bM_{\cS}$ with respect to $T$ unless $\cS$ is a closed surface with exactly one puncture. In which case, we introduce a notion of $n$-intersection numbers and vectors in Subsection \ref{subsec:1-punc}. Then we give the same results as Theorems \ref{thm:main} and \ref{thm:unique tag tri} in terms of $n$-intersection vectors (Theorem \ref{thm:unique Intn} and Corollary \ref{cor:unique tag tri}). Moreover, we show that the dimension vector of a $\tau$-rigid $J_T$-module coincide with the $n$-intersection vector of some multi-set in $\bM_{\cS}$ with respect to $T$ (Proposition \ref{prop:Intn=dim}). As a consequence, we can obtain analogues of the theorems in Subsection \ref{subsec:result} for dimension vectors of $\tau$-rigid modules as below.

For a $\tau$-rigid $J_T$-module $M$, we consider a decomposition $M=M_1\oplus M_2$, where $M_2$ is the maximal projective module whose each indecomposable direct summand is an indecomposable projective module at a vertex of $Q_T$ corresponding to a tagged arc in $T$ connecting punctures. 

\begin{thm}\label{thm:unique rigid0}
Let $T$ be a tagged triangulation of $\cS$, and $M$ and $N$ be $\tau$-rigid $J_T$-modules. If $\udim M=\udim N$, then $M_1=N_1$ and $\udim M_2=\udim N_2$.
\end{thm}

\begin{thm}\label{thm:unique sttilt}
Let $T$ be a tagged triangulation of $\cS$, and $M$ and $N$ be basic support $\tau$-tilting $J_T$-modules. If $\udim M=\udim N$, then $M=N$.
\end{thm}

Remark that in general there are two non-isomorphic basic $\tau$-rigid $J_T$-modules with the same dimension vector. Moreover, there are two non-isomorphic non-basic support $\tau$-tilting $J_T$-modules with the same dimension vector (Example \ref{ex:rep}).

\begin{thm}\label{thm:unique rigid}
Let $T$ be a tagged triangulation of $\cS$. Then the following are equivalent:
\begin{itemize}
\item[(1)] For any $\tau$-rigid $J_T$-modules $M$ and $N$, if $\udim M=\udim N$, then $M=N$.
\item[(2)] Each connected component of $G_T$ contains at most one cycle of odd length and no cycles of even length.
\end{itemize}
\end{thm}

\section{Triangulated surfaces}\label{sec:cS}

In this section, we are devoted to studying on triangulated surfaces \cite{FoST08,FT18}. In particular, we prove Theorems \ref{thm:main} and \ref{thm:unique tag tri} in Subsection \ref{subsec:pf thm:main}, and Theorems \ref{thm:unique Int} and \ref{thm:list cS} in Subsection \ref{subsec:graph}. Theorem \ref{thm:segment} plays an important role to prove them, and its proof will be given in Section \ref{sec:modif}.

\subsection{Intersection numbers}\label{subsec:Int}

Let $\cS$ be a connected compact oriented Riemann surface with (possibly empty) boundary $\partial\cS$, and $\cM$ be a non-empty finite set of marked points in $\cS$ with at least one marked point on each connected component of $\partial\cS$. We call the pair $(\cS,\cM)$ a \emph{marked surface}. Throughout this paper, we fix a marked surface $(\cS,\cM)$, denoted by $\cS$ for short. A marked point in the interior of $\cS$ is called a \emph{puncture}. For technical reasons, we assume that $\cS$ is not a monogon with at most one puncture, a digon without punctures, a triangle without punctures, and a sphere with at most three punctures (see \cite{FoST08} for the details). A curve in $\cS$ is considered up to isotopy relative to $\cM$.

A \emph{tagged arc} in $\cS$ is a curve in $\cS$ whose endpoints are in $\cM$ and each end is tagged in one of two ways, \emph{plain} or \emph{notched}, such that the following conditions are satisfied:
\begin{itemize}
\item It does not intersect itself except at its endpoints.
\item It is disjoint from $\cM$ and $\partial\cS$ except at its endpoints.
\item It does not cut out a monogon with at most one puncture or a digon without punctures.
\item Its ends incident to $\partial\cS$ are tagged plain.
\item Both ends of a loop are tagged in the same way,
\end{itemize}
where a loop is a curve whose endpoints coincide. In the figures, we represent tags as follows:
\[
\begin{tikzpicture}[baseline=-1mm]
 \coordinate(0)at(0,0) node[left]{plain};
 \coordinate(1)at(1,0); \fill(1)circle(0.07);
 \draw(0)to(1);
\end{tikzpicture}
\hspace{7mm}
\begin{tikzpicture}[baseline=-1mm]
 \coordinate(0)at(0,0) node[left]{notched};
 \coordinate(1)at(1,0); \fill(1)circle(0.07);
 \draw(0)to node[pos=0.8]{\rotatebox{90}{\footnotesize $\bowtie$}}(1);
\end{tikzpicture}\ .
\]
We also consider certain curves, that are not tagged arcs, as in Figure \ref{fig:punctured loop}.

\begin{defn}
A \emph{punctured loop} is a loop whose both ends are tagged in the same way such that it cuts out a monogon with exactly one puncture.
\end{defn}

We call a tagged arc (resp., punctured loop)
\begin{itemize}
\item a \emph{plain arc} (resp., \emph{plain punctured loop}) if both its tags are plain;
\item a \emph{$1$-notched arc} if its tags are different;
\item a \emph{$2$-notched arc} (resp., \emph{$2$-notched punctured loop}) if both its tags are notched.
\end{itemize}

For short, plain (resp., $2$-notched) arcs and plain (resp., $2$-notched) punctured loops in $\cS$ are collectively called \emph{plain} (resp., \emph{$2$-notched}) \emph{curves}.

A \emph{pair of conjugate arcs} is a pair of tagged arcs whose underlying curves coincide and exactly one of their tags is different from the others. To a pair $P$ of conjugate arcs, we associate a punctured loop $\Omega(P)$ as follows: Assume that $P$ connects marked points $p$ and $q$, and tags in $P$ at $p$ are different. Then $\Omega(P)$ is the punctured loop with the same tags at $q$ as $P$ that cuts out a monogon with exactly one puncture $p$ (see Figure \ref{fig:punctured loop}).

Throughout this paper, when we consider intersections of curves, we assume that they intersect transversally in a minimum number of points in $\cS\setminus\cM$. We denote by $\bAL_{\cS}$ the set of all tagged arcs and punctured loops in $\cS$. We extend the notion of intersection numbers of tagged arcs in \cite[Definition 3.3]{QZ17} to elements of $\bAL_{\cS}$.

\begin{defn}[{\cite[Definition 3.3]{QZ17}}]\label{def:Int}
Let $\gamma,\delta\in\bAL_{\cS}$. The \emph{intersection number $\Int(\gamma,\delta)$ of $\gamma$ and $\delta$} is defined by $A_{\gamma,\delta}+B_{\gamma,\delta}+C_{\gamma,\delta}$, where
\begin{itemize}
 \item $A_{\gamma,\delta}$ is the number of intersection points of $\gamma$ and $\delta$ in $\cS\setminus\cM$;
 \item $B_{\gamma,\delta}$ is the number of pairs of an end of $\gamma$ and an end of $\delta$ such that they are incident to a common puncture and their tags are different;
 \item $C_{\gamma,\delta}=0$ unless $\gamma$ and $\delta$ form a pair of conjugate arcs, in which case $C_{\gamma,\delta}=-1$.
\end{itemize}
\end{defn}

Note that the intersection number in Definition \ref{def:Int} is symmetric and different from the ``intersection number'' $(\gamma | \delta)$ in \cite[Definition 8.4]{FoST08} (see \cite{Y24} about the difference between them).

For $\gamma,\delta\in\bAL_{\cS}$, we say that they are \emph{compatible} if $\Int(\gamma,\delta)=0$. In particular, for a pair $\{\gamma,\gamma'\}$ of conjugate arcs, $\gamma$ and $\gamma'$ are compatible since $A_{\gamma,\gamma'}=0$, $B_{\gamma,\gamma'}=1$, and $C_{\gamma,\gamma'}=-1$. For multi-sets $U$ and $V$ of elements of $\bAL_{\cS}$, we also say that they are \emph{compatible} if $\Int(\gamma,\delta)=0$ for all $\gamma\in U$ and $\delta\in V$. We show that the above $\Omega$ preserves intersection numbers.

\begin{prop}\label{prop:Int conjugate}
Let $\{\gamma,\gamma'\}$ be a pair of conjugate arcs in $\cS$ and $\delta\in\bAL_{\cS}$. Then
\[
\Int(\delta,\Omega(\{\gamma,\gamma'\}))=\Int(\delta,\gamma)+\Int(\delta,\gamma').
\]
\end{prop}

\begin{proof}
First, we assume that the underlying curves of $\gamma$ and $\delta$ coincide. If $\delta$ is equal to either $\gamma$ or $\gamma'$, then it is easy to see that both sides of the desired equality are zero. Assume that $\delta$ is equal to neither $\gamma$ nor $\gamma'$, that is, either $\{\gamma,\delta\}$ or $\{\gamma',\delta\}$ is a pair of conjugate arcs. Then $\Int(\delta,\gamma)=0$ and $\Int(\delta,\gamma')=2$ in the former case; $\Int(\delta,\gamma)=2$ and $\Int(\delta,\gamma')=0$ in the latter case; $\Int(\delta,\Omega(\{\gamma,\gamma'\}))=2$ in the both cases. Therefore, the desired equality holds.

Next, we assume that the underlying curves of $\gamma$ and $\delta$ do not coincide. Let $p$ and $q$ be marked points connected by $\gamma$ such that $q$ is the endpoint of the punctured loop $\varepsilon=\Omega(\{\gamma,\gamma'\})$. Definition \ref{def:Int} means that
\begin{itemize}
\item$A_{\delta,\gamma}=A_{\delta,\gamma'}$;
\item $A_{\delta,\varepsilon}=2A_{\delta,\gamma}+\#\{\text{endpoints of $\delta$ at $p$}\}$;
\item $B_{\delta,\gamma}=\#\{\text{tags of $\delta$ at $p$ different from one of $\gamma$}\}+\#\{\text{tags of $\delta$ at $q$ different from one of $\gamma$}\}$;
\item $B_{\delta,\gamma'}=\#\{\text{tags of $\delta$ at $p$ different from one of $\gamma'$}\}+\#\{\text{tags of $\delta$ at $q$ different from one of $\gamma'$}\}$;
\item $B_{\delta,\varepsilon}=2\#\{\text{tags of $\delta$ at $q$ different from one of $\varepsilon$}\}$;
\item $C_{\delta,\gamma}=C_{\delta,\gamma'}=C_{\delta,\varepsilon}=0$.
\end{itemize}
Therefore,
\begin{align*}
\Int(\delta,\varepsilon)
&=2A_{\delta,\gamma}+\#\{\text{endpoints of $\delta$ at $p$}\}+2\#\{\text{tags of $\delta$ at $q$ different from one of $\varepsilon$}\}\\
&=A_{\delta,\gamma}+A_{\delta,\gamma'}+B_{\delta,\gamma}+B_{\delta,\gamma'}\\
&=\Int(\delta,\gamma)+\Int(\delta,\gamma'),
\end{align*}
where the second equality follows from the facts that tags of $\gamma$ and $\gamma'$ at $p$ are different, and tags of $\gamma$, $\gamma'$, and $\varepsilon$ at $q$ are the same.
\end{proof}

Next, we focus on multi-sets of elements of $\bAL_{\cS}$. For a set $T$ and a multi-set $U$ of elements of $\bAL_{\cS}$, the \emph{intersection vector of $U$ with respect to $T$} is the non-negative vector
\[
\Int_T(U):=\left(\Int(t,U)\right)_{t\in T}:=\Biggl(\sum_{u\in U}\Int(t,u)\Biggr)_{t\in T}\in\bZ_{\ge 0}^T.
\]
We also denote $\Int_T(\{\gamma\})$ by $\Int_T(\gamma)$. We extend the above $\Omega$ to elements of
\[
\bM_{\cS}:=\{\text{multi-sets of pairwise compatible tagged arcs in $\cS$}\}.
\]

\begin{defprop}\label{defprop:Omega}
For $U\in\bM_{\cS}$, we define $\Omega(U)$ as the multi-set obtained from $U$ by replacing the sub-multi-set $\{\gamma^m,(\gamma')^m\}$ with $\{\Omega(\{\gamma,\gamma'\})^m\}$ for each pair $\{\gamma,\gamma'\}$ of conjugate arcs, where $m=\min\{m_U(\gamma),m_U(\gamma')\}$. It induces a bijection
\[
\Omega:\bM_{\cS}
\rightarrow
\Biggl\{{\begin{gathered}\text{multi-sets of pairwise compatible elements}\\\text{of $\bAL_{\cS}$ without pairs of conjugate arcs}\end{gathered}}\Biggr\}.
\]
\end{defprop}

\begin{proof}
If $\{\gamma,\gamma'\}$ and $\{\gamma,\gamma''\}$ are pairs of conjugate arcs in $U$,  then $\gamma'=\gamma''$ since $\gamma'$ and $\gamma''$ are  compatible. Thus the map $\Omega$ is well-defined. On the other hand, by Proposition \ref{prop:Int conjugate}, an element of $\bAL_{\cS}$ is compatible with both $\gamma$ and $\gamma'$ if and only if it is compatible with $\Omega(\{\gamma,\gamma'\})$. Therefore, $\Omega(U)$ is a multi-set of pairwise compatible elements of $\bAL_{\cS}$ without pairs of conjugate arcs. Since the inverse map can be given by replacing all punctured loops with the corresponding pairs of conjugate arcs, $\Omega$ is a bijection.
\end{proof}

Finally, we consider certain sets in $\bM_{\cS}$. A \emph{tagged triangulation} of $\cS$ is a maximal set of pairwise compatible tagged arcs in $\cS$. For a tagged triangulation $T$ of $\cS$ and a tagged arc $\gamma$ in $T$, there is a unique tagged arc $\gamma'\notin T$ such that $\mu_{\gamma}T:=(T\setminus\{\gamma\})\cup\{\gamma'\}$ is a tagged triangulation of $\cS$. Here, $\mu_{\gamma}T$ is called the \emph{flip of $T$ at $\gamma$}.

For $\gamma\in\bAL_{\cS}$ and a puncture $p$ in $\cS$, we define $\gamma^{(p)}$ as the element of $\bAL_{\cS}$ obtained from $\gamma$ by changing all tags at $p$. Note that $\gamma^{(p)}=\gamma$ if $\gamma$ is not incident to $p$. It is easy to see that $\Int(\gamma^{(p)},\delta^{(p)})=\Int(\gamma,\delta)$ for $\delta\in\bAL_{\cS}$. Therefore, when we consider intersection vectors with respect to a tagged triangulation $T$ of $\cS$, by changing tags, we can assume that $T$ satisfies the following condition:
\begin{equation*}\label{diamond}
\tag{$\Diamond$} \begin{array}{l}\text{The tagged triangulation $T$ consists of plain arcs and $1$-notched arcs whose}\\
\text{each $1$-notched arc is contained in a pair of conjugate arcs.}\end{array}
\end{equation*}
In particular, if $T$ satisfies \eqref{diamond}, then all tags in $\Omega(T)$ are plain.

\begin{example}\label{ex:T Int}
Let $\cS$ be a monogon with three punctures. The following set $T$ of tagged arcs in $\cS$ is a tagged triangulation satisfying \eqref{diamond} and $\Omega(T)$ only consists of plain curves:
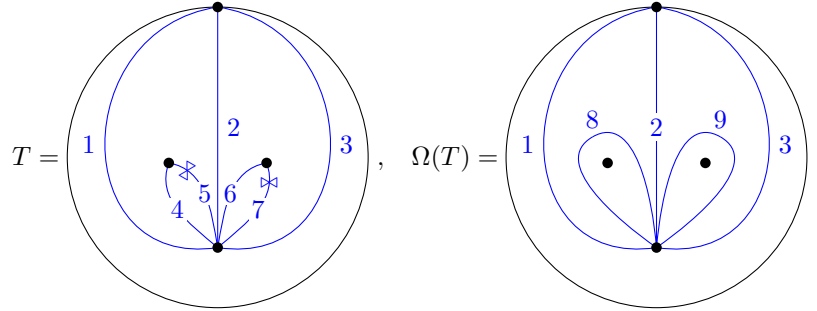
\[
T=
\begin{tikzpicture}[baseline=-1mm]
\coordinate(d)at(0,-1.2);\coordinate(u)at(90:2);
\coordinate(l)at($(d)+(120:1.3)$);\coordinate(r)at($(d)+(60:1.3)$);
\draw(0,0)circle(2);
\draw[blue](u)..controls(-2,1.7)and(-2,-1.5)..node[left]{$1$}(d);
\draw[blue](u)--node[right]{$2$}(d);
\draw[blue](u)..controls(2,1.7)and(2,-1.5)..node[right]{$3$}(d);
\draw[blue](d)to[out=20,in=130,relative]node[fill=white,inner sep=1]{$4$}(l);
\draw[blue](d)to[out=-20,in=-130,relative]node[fill=white,inner sep=1]{$5$}node[pos=0.8]{\scriptsize\rotatebox{60}{$\bowtie$}}(l);
\draw[blue](d)to[out=20,in=130,relative]node[fill=white,inner sep=1]{$6$}(r);
\draw[blue](d)to[out=-20,in=-130,relative]node[fill=white,inner sep=1]{$7$}node[pos=0.8]{\scriptsize\rotatebox{0}{$\bowtie$}}(r);
\fill(l)circle(0.07);\fill(r)circle(0.07);\fill(d)circle(0.07);\fill(u)circle(0.07);
\end{tikzpicture}
\ ,\hspace{3mm}
\Omega(T)=
\begin{tikzpicture}[baseline=-1mm]
\coordinate(d)at(0,-1.2);\coordinate(u)at(90:2);
\coordinate(l)at($(d)+(120:1.3)$);\coordinate(r)at($(d)+(60:1.3)$);
\draw(0,0)circle(2);
\draw[blue](u)..controls(-2,1.7)and(-2,-1.5)..node[left]{$1$}(d);
\draw[blue](u)--node[fill=white,inner sep=2]{$2$}(d);
\draw[blue](u)..controls(2,1.7)and(2,-1.5)..node[right]{$3$}(d);
\draw[blue](d)to[out=25,in=90,relative]($(d)+(120:1.7)$);
\draw[blue](d)to[out=-25,in=-90,relative]($(d)+(120:1.7)$)node[above]{$8$};
\draw[blue](d)to[out=25,in=90,relative]($(d)+(60:1.7)$);
\draw[blue](d)to[out=-25,in=-90,relative]($(d)+(60:1.7)$)node[above]{$9$};
\fill(l)circle(0.07);\fill(r)circle(0.07);\fill(d)circle(0.07);\fill(u)circle(0.07);
\end{tikzpicture}\ ,
\]
where $8=\Omega(\{4,5\})$ and $9=\Omega(\{6,7\})$. We take multi-sets $U=\{\alpha,\beta,\gamma^3,\gamma'\}\in\bM_{\cS}$ and $\Omega(U)=\{\alpha,\beta,\gamma^2,\Omega(\{\gamma,\gamma'\})\}$ whose each curve is given as follows:
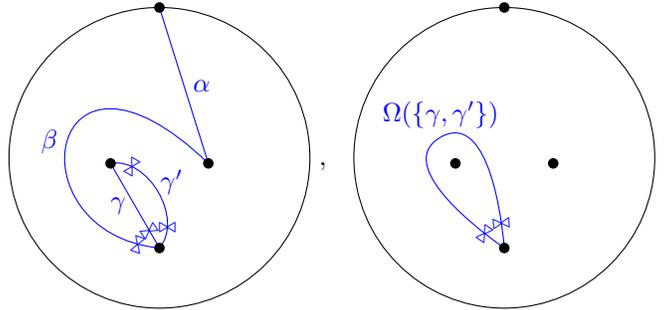
\[
\begin{tikzpicture}[baseline=-1mm]
\coordinate(d)at(0,-1.2);\coordinate(u)at(90:2);
\coordinate(l)at($(d)+(120:1.3)$);\coordinate(r)at($(d)+(60:1.3)$);
\draw(0,0)circle(2);
\draw[blue](u)--node[right]{$\alpha$}(r);
\draw[blue](d)..controls(-2,-1)and(-1.5,2)..node[left]{$\beta$}node[pos=0.05]{\scriptsize\rotatebox{70}{$\bowtie$}}(r);
\draw[blue](d)to[out=-70,in=-110,relative]node[right,pos=0.6]{$\gamma'$}node[pos=0.2]{\scriptsize\rotatebox{0}{$\bowtie$}}node[pos=0.8]{\scriptsize\rotatebox{60}{$\bowtie$}}(l);
\draw[blue](d)--node[left,pos=0.5]{$\gamma$}node[pos=0.2]{\scriptsize\rotatebox{35}{$\bowtie$}}(l);
\fill(l)circle(0.07);\fill(r)circle(0.07);\fill(d)circle(0.07);\fill(u)circle(0.07);
\end{tikzpicture}
\ ,\hspace{3mm}
\begin{tikzpicture}[baseline=-1mm]
\coordinate(0)at(0,0);\coordinate(d)at(0,-1.2);\coordinate(u)at(90:2);
\coordinate(l)at($(d)+(120:1.3)$);\coordinate(r)at($(d)+(60:1.3)$);
\draw(0,0)circle(2);
\draw[blue](d)to[out=25,in=90,relative]node[pos=0.15]{\scriptsize\rotatebox{55}{$\bowtie$}}($(d)+(120:1.7)$);
\draw[blue](d)to[out=-25,in=-90,relative]node[pos=0.15]{\scriptsize\rotatebox{5}{$\bowtie$}}($(d)+(120:1.7)$)node[above]{$\Omega(\{\gamma,\gamma'\})$};
\fill(l)circle(0.07);\fill(r)circle(0.07);\fill(d)circle(0.07);\fill(u)circle(0.07);
\end{tikzpicture}\ .
\]
Then $\Int_T(U)=\Int_T(\Omega(U))=(5,6,5,3,7,5,7)$ is given by
\begin{gather*}
\Int_T(\alpha)=(0,0,0,0,0,0,1), \Int_T(\beta)=(1,2,1,1,1,1,2),\\
\Int_T(\gamma)+\Int_T(\gamma')=(1,1,1,0,2,1,1)+(1,1,1,2,0,1,1)=(2,2,2,2,2,2,2)=\Int_T(\Omega(\{\gamma,\gamma'\})).
\end{gather*}
\end{example}

\subsection{Puzzle pieces, their edges and segments}\label{subsec:segment}

For a tagged triangulation $T$ of $\cS$, $\Omega(T)$ decomposes $\cS$ into triangles and monogons (see \cite[Remark 4.2]{FoST08}), called \emph{triangle pieces} and \emph{monogon pieces}, respectively. We also call them \emph{puzzle pieces}. Remark that puzzle pieces of $T$ are defined in \cite[Remark 4.2]{FoST08}, and they appear in Table \ref{table:Qtri} In a puzzle piece, we define certain curves, called (\emph{loop} or \emph{non-loop}) \emph{edges} and \emph{segments}, as in Table \ref{table:segment}. We often identify each edge in puzzle pieces of $\Omega(T)$ with the corresponding tagged arc or punctured loop in $T$ or $\Omega(T)$.

\renewcommand{\arraystretch}{1.3}
{\begin{table}[ht]
\begin{tabular}{l||c|c|c}
\multirow{2}{*}{Puzzle piece}&\multicolumn{2}{|c|}{Edges}&\multirow{2}{*}{Segments}\\\cline{2-3}
&Non-loop edges&Loop edge&\\\hline\hline
Triangle piece
\begin{tikzpicture}[baseline=1mm]
\coordinate(l)at(-150:1);\coordinate(r)at(-30:1);\coordinate(u)at(90:1);\draw(u)--(l)--(r)--(u);
\node at(90:0.6){$a$};\node at(0,1.1){};\node at(0,-0.6){};
\end{tikzpicture}
&
\begin{tikzpicture}[baseline=1mm]
\coordinate(l)at(-150:1);\coordinate(r)at(-30:1);\coordinate(u)at(90:1);\draw(u)--(l)--(r)--(u);
\draw[blue](l)to[out=20,in=160,relative]node[above]{$e_a$}(r);
\end{tikzpicture}
&
&
\begin{tikzpicture}[baseline=1mm]
\coordinate(l)at(-150:1);\coordinate(r)at(-30:1);\coordinate(u)at(90:1);\draw(u)--(l)--(r)--(u);
\draw[blue]($(u)!0.5!(l)$)--node[below]{$h_a$}($(u)!0.5!(r)$);
\end{tikzpicture}\ 
\begin{tikzpicture}[baseline=1mm]
\coordinate(l)at(-150:1);\coordinate(r)at(-30:1);\coordinate(u)at(90:1);\draw(u)--(l)--(r)--(u);
\draw[blue](u)--node[right,pos=0.7]{$v_a$}($(l)!0.5!(r)$);
\end{tikzpicture}\ 
\begin{tikzpicture}[baseline=1mm]
\coordinate(l)at(-150:1);\coordinate(r)at(-30:1);\coordinate(u)at(90:1);\draw(u)--(l)--(r)--(u);
\draw[blue](0,0)--($(u)!0.5!(l)$);\draw[blue](0,0)--($(u)!0.5!(r)$);\draw[blue](0,0)--node[right,pos=0.3]{$y$}($(l)!0.5!(r)$);
\end{tikzpicture}
\\\hline
Monogon piece\hspace{1mm}
\begin{tikzpicture}[baseline=1mm]
\coordinate(u)at(90:1);\coordinate(p)at(0,0.5);\coordinate(d)at(-90:0.5);
\draw(d)to[out=150,in=180](u);\draw(d)to[out=30,in=0](u);
\fill(p)circle(0.07);\node at(0,-0.2){$a$};\node at(0,1.1){};\node at(0,-0.6){};
\end{tikzpicture}
&
\begin{tikzpicture}[baseline=1mm]
\coordinate(u)at(90:1);\coordinate(p)at(0,0.5);\coordinate(d)at(-90:0.5);
\draw(d)to[out=150,in=180](u);\draw(d)to[out=30,in=0](u);
\draw[blue](d)--(p)node[above]{$f_a$};
\fill(p)circle(0.07);
\end{tikzpicture}\ 
\begin{tikzpicture}[baseline=1mm]
\coordinate(u)at(90:1);\coordinate(p)at(0,0.5);\coordinate(d)at(-90:0.5);
\draw(d)to[out=150,in=180](u);\draw(d)to[out=30,in=0](u);
\draw[blue](d)--node[pos=0.8]{\footnotesize$\bowtie$}(p);
\fill(p)circle(0.07);\node at(0,0.75){$f_a^{\bowtie}$};
\end{tikzpicture}
&
\begin{tikzpicture}[baseline=1mm]
\coordinate(u)at(90:1);\coordinate(p)at(0,0.5);\coordinate(d)at(-90:0.5);
\draw(d)to[out=150,in=180](u);\draw(d)to[out=30,in=0](u);
\draw[blue](d)to[out=130,in=180](90:0.8);\draw[blue](d)to[out=50,in=0](90:0.8);\node[blue]at(0,0){$e_a$};
\fill(p)circle(0.07);
\end{tikzpicture}
&
\begin{tikzpicture}[baseline=1mm]
\coordinate(u)at(90:1);\coordinate(p)at(0,0.5);\coordinate(d)at(-90:0.5);
\draw(d)to[out=150,in=180]coordinate[pos=0.3](l)(u);\draw(d)to[out=30,in=0]coordinate[pos=0.3](r)(u);
\draw[blue](l)--node[above]{$h_a$}(r);
\fill(p)circle(0.07);
\end{tikzpicture}\ 
\begin{tikzpicture}[baseline=1mm]
\coordinate(u)at(90:1);\coordinate(p)at(0,0.5);\coordinate(d)at(-90:0.5);
\draw(d)to[out=150,in=180](u);\draw(d)to[out=30,in=0](u);
\draw[blue](p)--node[below,pos=0.3]{$i_a$}(0.4,0.5);
\fill(p)circle(0.07);
\end{tikzpicture}\ 
\begin{tikzpicture}[baseline=1mm]
\coordinate(u)at(90:1);\coordinate(p)at(0,0.5);\coordinate(d)at(-90:0.5);
\draw(d)to[out=150,in=180](u);\draw(d)to[out=30,in=0](u);
\draw[blue](p)--node[pos=0.4]{\rotatebox{90}{\footnotesize $\bowtie$}}(0.4,0.5);
\fill(p)circle(0.07);\node at(0.1,0.1){$i_a^{\bowtie}$};
\end{tikzpicture}
\end{tabular}\vspace{3mm}
\caption{Edges and segments in each puzzle piece with an angle $a$}
\label{table:segment}
\end{table}}

We can naturally extend the notions of intersection numbers and compatibility of tagged arcs to edges and segments in a puzzle piece as in the following example. We also define that two of them in different puzzle pieces are compatible.

\begin{example}
Let $S$ be a multi-set of edges and segments in a puzzle piece with an angle $a$. In the case of a triangle piece,
\[
\Int(h_a,S)=m_S(e_b)+m_S(e_c)+m_S(v_a),
\]
where $b$ and $c$ are the other angles in the triangle piece. In the case of a monogon piece,
\[
\Int(f_a,S)=m_S(h_a)+m_S(i_a^{\bowtie})\mbox{\ and\ }\Int(f_a^{\bowtie},S)=m_S(h_a)+m_S(i_a).
\]
Moreover, we also have the equalities (cf. Proposition \ref{prop:Int conjugate})
\[
\Int(e_a,S)=2m_S(h_a)+m_S(i_a)+m_S(i_a^{\bowtie})=\Int(f_a,S)+\Int(f_a^{\bowtie},S).
\]
\end{example}

We show that a multi-set $S$ of pairwise compatible segments in a puzzle piece is uniquely determined by their intersection numbers with non-loop edges. Note that $m_S(y)\le 1$ since $\Int(y,y)=1$.

\begin{prop}\label{prop:Int triangle}
Let $\triangle$ be a triangle piece with angles $1$, $2$, and $3$. For multi-sets $S$ and $S'$ of pairwise compatible segments in $\triangle$, if $\Int(e_i,S)=\Int(e_i,S')$ for all $i$, then $S=S'$.
\end{prop}

\begin{proof}
For short, we denote $\Int(e_i,S)$ by $a_i$. By symmetry, we can assume that $a_1\ge a_2, a_3$. Then it follows from the compatibility of $S$ that $S$ is one of the following (see Figure \ref{fig:compatible segments}):
\begin{itemize}
\item[(1)] A multi-set consisting of $h_2$, $h_3$, and $v_1$ with $m_S(v_1)>0$.
\item[(2)] A multi-set consisting of $h_1$, $h_2$, $h_3$, and $y$ with $m_S(y)\le 1$.
\end{itemize}
By comparing the cases, if $a_1>a_2+a_3$, then $S$ is (1), and it must be the multi-set
\[
\{h_2^{a_3},h_3^{a_2},v_1^{a_1-a_2-a_3}\}.
\]
If $a_1\le a_2+a_3$ and $a_1+a_2+a_3$ is even, then $S$ is (2) with $m_S(y)=0$, and it must be the multi-set
\[
\{h_1^{\frac{a_2+a_3-a_1}{2}},h_2^{\frac{a_3+a_1-a_2}{2}},h_3^{\frac{a_1+a_2-a_3}{2}}\}.
\]
If $a_1\le a_2+a_3$ and $a_1+a_2+a_3$ is odd, then $S$ is (2) with $m_S(y)=1$, and it must be the multi-set
\[
\{h_1^{\frac{a_2+a_3-a_1-1}{2}},h_2^{\frac{a_3+a_1-a_2-1}{2}},h_3^{\frac{a_1+a_2-a_3-1}{2}},y\}.
\]
Therefore, $S$ is uniquely determined by $a_1$, $a_2$, and $a_3$, thus the assertion holds.
\end{proof}

%
%
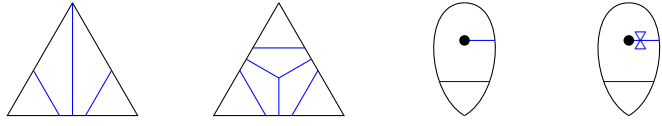
\begin{figure}[ht]
\[
\begin{tikzpicture}
\coordinate(l)at(-150:1);\coordinate(r)at(-30:1);\coordinate(u)at(90:1);\draw(u)--(l)--(r)--(u);
\draw[blue](u)--($(l)!0.5!(r)$);\draw[blue]($(u)!0.6!(l)$)--($(l)!0.4!(r)$);\draw[blue]($(u)!0.6!(r)$)--($(l)!0.6!(r)$);
\end{tikzpicture}
\hspace{10mm}
\begin{tikzpicture}
\coordinate(l)at(-150:1);\coordinate(r)at(-30:1);\coordinate(u)at(90:1);\draw(u)--(l)--(r)--(u);
\draw[blue](0,0)--($(u)!0.5!(l)$);\draw[blue](0,0)--($(u)!0.5!(r)$);\draw[blue](0,0)--($(l)!0.5!(r)$);
\draw[blue]($(u)!0.6!(l)$)--($(l)!0.4!(r)$);\draw[blue]($(u)!0.6!(r)$)--($(l)!0.6!(r)$);\draw[blue]($(u)!0.4!(l)$)--($(u)!0.4!(r)$);
\end{tikzpicture}
\hspace{10mm}
\begin{tikzpicture}
\coordinate(u)at(90:1);\coordinate(p)at(0,0.5);\coordinate(d)at(-90:0.5);
\draw(d)to[out=150,in=180]coordinate[pos=0.3](l)(u);\draw(d)to[out=30,in=0]coordinate[pos=0.3](r)(u);
\draw[blue](l)--(r);
\draw[blue](p)--(0.4,0.5);
\fill(p)circle(0.07);
\end{tikzpicture}
\hspace{10mm}
\begin{tikzpicture}
\coordinate(u)at(90:1);\coordinate(p)at(0,0.5);\coordinate(d)at(-90:0.5);
\draw(d)to[out=150,in=180]coordinate[pos=0.3](l)(u);\draw(d)to[out=30,in=0]coordinate[pos=0.3](r)(u);
\draw[blue](l)--(r);
\draw[blue](p)--node[pos=0.4]{\rotatebox{90}{\footnotesize $\bowtie$}}(0.4,0.5);
\fill(p)circle(0.07);
\end{tikzpicture}
\]
\caption{The maximal sets of pairwise compatible segments in each puzzle piece}
\label{fig:compatible segments}
\end{figure}

\begin{prop}\label{prop:Int monogon}
Let $\triangle$ be a monogon piece with an angle $a$. For multi-sets $S$ and $S'$ of pairwise compatible segments in $\triangle$, if $\Int(f_a,S)=\Int(f_a,S')$ and $\Int(f_a^{\bowtie},S)=\Int(f_a^{\bowtie},S')$, then $S=S'$.
\end{prop}

\begin{proof}
It follows from the compatibility of $S$ that $S$ is a multi-set consisting of $h_a$ and either $i_a$ or $i_a^{\bowtie}$ (see Figure \ref{fig:compatible segments}). By comparing the cases, if $\Int(f_a,S)<\Int(f_a^{\bowtie},S)$, then $S$ must be the multi-set
\[
\{h_a^{\Int(f_a,S)},i_a^{\Int(f_a^{\bowtie},S)-\Int(f_a,S)}\}.
\]
If $\Int(f_a,S)\ge\Int(f_a^{\bowtie},S)$, then $S$ must be the multi-set
\[
\{h_a^{\Int(f_a^{\bowtie},S)},(i_a^{\bowtie})^{\Int(f_a,S)-\Int(f_a^{\bowtie},S)}\}.
\]
Therefore, $S$ is uniquely determined by $\Int(f_a,S)$ and $\Int(f_a^{\bowtie},S)$, thus the assertion holds.
\end{proof}

\begin{example}\label{ex:segment}
In the setting of Example \ref{ex:T Int}, the intersection vector $\Int_T(U)=(5,6,5,3,7,5,7)$ induces the following segments in each puzzle piece (see the proofs of Propositions \ref{prop:Int triangle} and \ref{prop:Int monogon}):
\[
\begin{tikzpicture}[baseline=-5mm]
\coordinate(l)at(-150:1);\coordinate(r)at(-30:1);\coordinate(u)at(90:1);
\draw(u)--node[above left]{$1$}(l)--node[below]{$8$}(r)--node[above right]{$2$}(u);
\draw[blue](0,0)--($(u)!0.5!(l)$);\draw[blue](0,0)--($(u)!0.5!(r)$);\draw[blue](0,0)--($(l)!0.5!(r)$);
\draw[blue]($(u)!0.6!(l)$)--($(l)!0.4!(r)$);\draw[blue]($(u)!0.65!(l)$)--($(l)!0.35!(r)$);\draw[blue]($(u)!0.7!(l)$)--($(l)!0.3!(r)$);\draw[blue]($(u)!0.75!(l)$)--($(l)!0.25!(r)$);
\draw[blue]($(u)!0.6!(r)$)--($(l)!0.6!(r)$);\draw[blue]($(u)!0.65!(r)$)--($(l)!0.65!(r)$);\draw[blue]($(u)!0.7!(r)$)--($(l)!0.7!(r)$);\draw[blue]($(u)!0.75!(r)$)--($(l)!0.75!(r)$);\draw[blue]($(u)!0.8!(r)$)--($(l)!0.8!(r)$);
\end{tikzpicture}
\ ,\hspace{7mm}
\begin{tikzpicture}[baseline=-5mm]
\coordinate(l)at(-150:1);\coordinate(r)at(-30:1);\coordinate(u)at(90:1);
\draw(u)--node[above left]{$2$}(l)--node[below]{$9$}(r)--node[above right]{$3$}(u);
\draw[blue](u)--($(l)!0.5!(r)$);
\draw[blue]($(u)!0.6!(l)$)--($(l)!0.4!(r)$);\draw[blue]($(u)!0.65!(l)$)--($(l)!0.35!(r)$);\draw[blue]($(u)!0.7!(l)$)--($(l)!0.3!(r)$);\draw[blue]($(u)!0.75!(l)$)--($(l)!0.25!(r)$);\draw[blue]($(u)!0.8!(l)$)--($(l)!0.2!(r)$);\draw[blue]($(u)!0.85!(l)$)--($(l)!0.15!(r)$);
\draw[blue]($(u)!0.6!(r)$)--($(l)!0.6!(r)$);\draw[blue]($(u)!0.65!(r)$)--($(l)!0.65!(r)$);\draw[blue]($(u)!0.7!(r)$)--($(l)!0.7!(r)$);\draw[blue]($(u)!0.75!(r)$)--($(l)!0.75!(r)$);\draw[blue]($(u)!0.8!(r)$)--($(l)!0.8!(r)$);
\end{tikzpicture}
\ ,\hspace{7mm}
\begin{tikzpicture}[baseline=-5mm]
\coordinate(l)at(-150:1);\coordinate(r)at(-30:1);\coordinate(u)at(90:1);
\draw(u)--node[above left]{$1$}(l)--node[below]{}(r)--node[above right]{$3$}(u);
\draw[blue]($(u)!0.4!(l)$)--($(u)!0.4!(r)$);\draw[blue]($(u)!0.35!(l)$)--($(u)!0.35!(r)$);\draw[blue]($(u)!0.3!(l)$)--($(u)!0.3!(r)$);\draw[blue]($(u)!0.25!(l)$)--($(u)!0.25!(r)$);\draw[blue]($(u)!0.2!(l)$)--($(u)!0.2!(r)$);
\end{tikzpicture}
\ ,\hspace{7mm}
\begin{tikzpicture}[baseline=-3mm]
\coordinate(u)at(90:1);\coordinate(p)at(0,0.5);\coordinate(d)at(-90:0.5);
\draw(d)to[out=150,in=180]coordinate[pos=0.3](l)coordinate[pos=0.25](l1)coordinate[pos=0.35](l2)coordinate[pos=0.5](l3)coordinate[pos=0.65](l4)(u);
\draw(d)to[out=30,in=0]coordinate[pos=0.3](r)coordinate[pos=0.25](r1)coordinate[pos=0.35](r2)coordinate[pos=0.5](r3)coordinate[pos=0.65](r4)(u)node[above]{$8$};
\draw[blue](l)--(r)(l1)--(r1)(l2)--(r2);
\draw[blue](l3)--(p)--(r3)(l4)--(p)--(r4);
\fill(p)circle(0.07);
\end{tikzpicture}
\ ,\hspace{7mm}
\begin{tikzpicture}[baseline=-3mm]
\coordinate(u)at(90:1);\coordinate(p)at(0,0.5);\coordinate(d)at(-90:0.5);
\draw(d)to[out=150,in=180]coordinate[pos=0.3](l)coordinate[pos=0.25](l1)coordinate[pos=0.35](l2)coordinate[pos=0.4](l3)coordinate[pos=0.45](l4)(u);
\draw(d)to[out=30,in=0]coordinate[pos=0.3](r)coordinate[pos=0.25](r1)coordinate[pos=0.35](r2)coordinate[pos=0.4](r3)coordinate[pos=0.45](r4)(u)node[above]{$9$};
\draw[blue](l)--(r)(l1)--(r1)(l2)--(r2)(l3)--(r3)(l4)--(r4);
\draw[blue](-0.4,0.5)--(p)--(0.4,0.5);
\fill(p)circle(0.07);
\end{tikzpicture}\ .
\]
In particular, these segments are pairwise compatible.
\end{example}

\subsection{Proof of Theorems \ref{thm:main} and \ref{thm:unique tag tri}}\label{subsec:pf thm:main}

In this subsection, we prove Theorems \ref{thm:main} and \ref{thm:unique tag tri}. For that, we prepare some notations and lemmas. Fix a tagged triangulation $T$ of $\cS$ satisfying \eqref{diamond}, in particular, all tags in $\Omega(T)$ are plain. Let $U\in\bM_{\cS}$ with $U\cap T=\emptyset$. We consider a decomposition
\[
\Omega(U)=U_1\sqcup U_2,
\]
where $U_2$ consists of all $2$-notched curves in $\Omega(U)$ whose underlying plain curves are in $\Omega(T)$. For each puncture $p$ and $V\in\bM_{\cS}\cup\Omega(\bM_{\cS})$, we denote by $n(V,p)$ the number of notched tags in $V$ incident to $p$, and by $c_p$ the simple closed curve enclosing exactly one puncture $p$. We set
\[
U^{\circ}:=U_0\sqcup\{c_p^{n(\Omega(U),p)}\mid \mbox{$p$ is a puncture incident to $\Omega(T)$}\},
\]
where $U_0$ is the multi-set of tagged arcs and punctured loops obtained from $U_1$ by changing all tags at punctures incident to $\Omega(T)$ to plain. Note that a puncture is not incident to $\Omega(T)$ if and only if it is enclosed by a punctured loop in $\Omega(T)$.

\begin{lem}\label{lem:recover0}
Let $U, V\in\bM_{\cS}$ with $U\cap T=V\cap T=\emptyset$. If $U^{\circ}=V^{\circ}$, then $U_1=V_1$.
\end{lem}

\begin{proof}
Since $U_0$ contains no closed curves, it is obtained from $U^{\circ}$ by removing all closed curves. Moreover, $U_1$ is obtained from $U_0$ by changing all tags at punctures enclosed by closed curves in $U^{\circ}$ to notched. Therefore, the assertion holds.
\end{proof}

For each puzzle piece $\triangle$ of $\Omega(T)$ and a plain curve or simple closed curve $\gamma$, the intersection $\gamma\cap\triangle$ is either an edge or a multi-set of segments in $\triangle$, where we define that $\gamma\cap\triangle=\{\gamma\}$ if $\gamma$ is contained in $\Omega(T)$ and is also an edge of $\triangle$. We define the intersection $U^{\circ}\cap\triangle$ as the multi-set of edges and segments in $\triangle$
\[
U^{\circ}\cap\triangle:=\bigsqcup_{\gamma\in U^{\circ}}(\gamma\cap\triangle).
\]
Note that if $\gamma\in U^{\circ}$ is a common edge of adjacent puzzle pieces $\triangle$ and $\triangle'$, then both $U^{\circ}\cap\triangle$ and $U^{\circ}\cap\triangle'$ contain $\gamma$, that is,
\[
m_{U^{\circ}\cap\triangle}(\gamma)=m_{U^{\circ}\cap\triangle'}(\gamma)=m_{U^{\circ}}(\gamma).
\]

\begin{lem}\label{lem:no loop edge}
Let $U\in\bM_{\cS}$ with $U\cap T=\emptyset$. Then $U^{\circ}$ contains no punctured loops in $\Omega(T)$. In particular, $U^{\circ}\cap\triangle$ contains no loop edges for all monogon pieces $\triangle$ of $\Omega(T)$.
\end{lem}

\begin{proof}
Assume that $U^{\circ}$ contains a punctured loop $\gamma$ in $\Omega(T)$ with endpoint $p$. Then $\Omega(U)$ must contain either $\gamma$ or $\gamma^{(p)}$. In the former case, it contradicts $U\cap T=\emptyset$. In the latter case, $\gamma^{(p)}$ is in $U_2$, but not in $U_1$. Thus it contradicts $\gamma\in U_0$.
\end{proof}

By Lemma \ref{lem:no loop edge}, $U^{\circ}\cap\triangle$ consists of non-loop edges and segments in $\triangle$. We consider the multi-set of non-loop edges and segments
\[
S_U:=\bigsqcup_{\triangle} (U^{\circ}\cap\triangle),
\]
where $\triangle$ runs over all puzzle pieces of $\Omega(T)$.

\begin{lem}\label{lem:recover1}
Let $U, V\in\bM_{\cS}$ with $U\cap T=V\cap T=\emptyset$. If $S_U=S_V$, then $U^{\circ}=V^{\circ}$, in particular, $n(\Omega(U),p)=n(\Omega(V),p)$ for all punctures $p$ incident to $\Omega(T)$.
\end{lem}

\begin{proof}
The assertion holds since $U^{\circ}$ is the disjoint union of the multi-set $\{\gamma^{\frac{1}{2}m_{S_U}(\gamma)}\mid\gamma\in\Omega(T)\}$ and the multi-set of curves obtained from $S_U$ by gluing segments simultaneously when we glue puzzle pieces of $\Omega(T)$ (see Example \ref{ex:glue} and Subsection \ref{subsec:glue}).
\end{proof}

For a multi-set $S$ of edges and segments, and a puzzle piece $\triangle$ of $\Omega(T)$, we also denote by $S\cap\triangle$ the maximal sub-multi-set of $S$ consisting of edges and segments in $\triangle$. The following theorem is a key result in this paper.

\begin{thm}\label{thm:segment}
Let $U\in\bM_{\cS}$ with $U\cap T=\emptyset$. Then there is a multi-set $\Phi(S_U)$ of pairwise compatible segments such that $\Int(\gamma,\Phi(S_U)\cap\triangle)=\Int(\gamma,U)$ for each puzzle piece $\triangle$ and $\gamma\in T\cup\Omega(T)$ that is also an edge in $\triangle$. Moreover, if there is $V\in\bM_{\cS}$ such that $V\cap T=\emptyset$ and $\Phi(S_U)=\Phi(S_V)$, then $S_U=S_V$.
\end{thm}

We will prove Theorem \ref{thm:segment} in Subsection \ref{subsec:proof}. We are ready to prove Theorems \ref{thm:main} and \ref{thm:unique tag tri}.

\begin{proof}[Proof of Theorem \ref{thm:main}]
The assumption $\Int_T(U)=\Int_T(V)$ and Theorem \ref{thm:segment} induce $\Int(\gamma,\Phi(S_U)\cap\triangle)=\Int(\gamma,\Phi(S_V)\cap\triangle)$ for all puzzle pieces $\triangle$ of $\Omega(T)$ and all non-loop edges $\gamma$ in $\triangle$. Then $\Phi(S_U)\cap\triangle=\Phi(S_V)\cap\triangle$ by Propositions \ref{prop:Int triangle} and \ref{prop:Int monogon}. By Theorem \ref{thm:segment} again, we obtain that $S_U=S_V$. Therefore, it follows from Lemmas \ref{lem:recover0} and \ref{lem:recover1} that $U_1=V_1$ and $n(\Omega(U),p)=n(\Omega(V),p)$ for all punctures $p$. In particular,
\[
n(U_2,p)=n(\Omega(U),p)-n(U_1,p)=n(\Omega(V),p)-n(V_1,p)=n(V_2,p).\qedhere
\]
\end{proof}

\begin{proof}[Proof of Theorem \ref{thm:unique tag tri}]
Assume that there is $\gamma\in U\setminus V$. Since $V$ is a tagged triangulation, $\gamma$ is not compatible with $V$. By Proposition \ref{prop:Int conjugate}, it is also not compatible with $\Omega(V)$. By Proposition \ref{prop:Int conjugate} again, there is an element of $\Omega(U)$ that is not compatible with $\Omega(V)$. On the other hand, we know that $U_1=V_1$ by Theorem \ref{thm:main}. Moreover, $U_2$ is compatible with $V_2$ by the compatibility of $T$. Thus $\Omega(U)$ must be compatible with $\Omega(V)$, a contradiction. Therefore, $U$ is contained in $V$. Similarly, $V$ is contained in $U$, that is, $U=V$.
\end{proof}

\begin{example}\label{ex:glue}
In the setting of Example \ref{ex:T Int}, there are decompositions
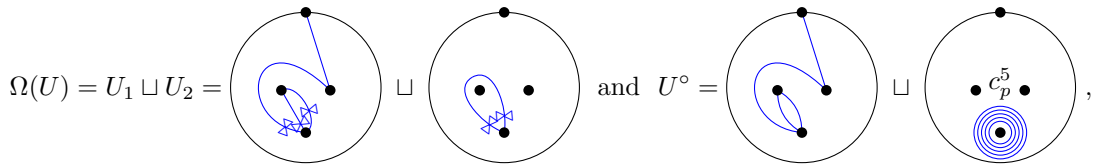
\[
\Omega(U)=U_1\sqcup U_2=
\begin{tikzpicture}[baseline=-1mm,scale=0.5]
\coordinate(d)at(0,-1.2);\coordinate(u)at(90:2);
\coordinate(l)at($(d)+(120:1.3)$);\coordinate(r)at($(d)+(60:1.3)$);
\draw(0,0)circle(2);
\draw[blue](u)--(r);
\draw[blue](d)..controls(-2,-1)and(-1.5,2)..node[pos=0.1]{\scriptsize\rotatebox{70}{$\bowtie$}}(r);
\draw[blue](d)--node[pos=0.2]{\scriptsize\rotatebox{35}{$\bowtie$}}(l);
\draw[blue](d)to[out=-90,in=-90,relative]node[pos=0.4]{\scriptsize\rotatebox{10}{$\bowtie$}}(l);
\fill(l)circle(0.14);\fill(r)circle(0.14);\fill(d)circle(0.14);\fill(u)circle(0.14);
\end{tikzpicture}
\ \sqcup\ 
\begin{tikzpicture}[baseline=-1mm,scale=0.5]
\coordinate(0)at(0,0);\coordinate(d)at(0,-1.2);\coordinate(u)at(90:2);
\coordinate(l)at($(d)+(120:1.3)$);\coordinate(r)at($(d)+(60:1.3)$);
\draw(0,0)circle(2);
\draw[blue](d)to[out=40,in=90,relative]node[pos=0.2]{\scriptsize\rotatebox{55}{$\bowtie$}}($(d)+(120:1.7)$);
\draw[blue](d)to[out=-40,in=-90,relative]node[pos=0.2]{\scriptsize\rotatebox{5}{$\bowtie$}}($(d)+(120:1.7)$);
\fill(l)circle(0.14);\fill(r)circle(0.14);\fill(d)circle(0.14);\fill(u)circle(0.14);
\end{tikzpicture}
\ \text{ and }\ 
U^{\circ}=
\begin{tikzpicture}[baseline=-1mm,scale=0.5]
\coordinate(d)at(0,-1.2);\coordinate(u)at(90:2);
\coordinate(l)at($(d)+(120:1.3)$);\coordinate(r)at($(d)+(60:1.3)$);
\draw(0,0)circle(2);
\draw[blue](u)--(r);
\draw[blue](d)..controls(-2,-1)and(-1.5,2)..(r);
\draw[blue](d)to[out=30,in=150,relative](l)(d)to[out=-30,in=-150,relative](l);
\fill(l)circle(0.14);\fill(r)circle(0.14);\fill(d)circle(0.14);\fill(u)circle(0.14);
\end{tikzpicture}
\ \sqcup\ 
\begin{tikzpicture}[baseline=-1mm,scale=0.5]
\coordinate(0)at(0,0);\coordinate(d)at(0,-1.2);\coordinate(u)at(90:2);
\coordinate(l)at($(d)+(120:1.3)$);\coordinate(r)at($(d)+(60:1.3)$);
\draw(0,0)circle(2);\node at(0,0.1){$c_p^5$};
\draw[blue](d)circle(0.3)circle(0.4)circle(0.5)circle(0.6)circle(0.7);
\fill(l)circle(0.14);\fill(r)circle(0.14);\fill(d)circle(0.14);\fill(u)circle(0.14);
\end{tikzpicture}\ ,
\]
where $n(\Omega(U),p)=5$ and $n(U_2,p)=2$ for the bottom puncture $p$. For each puzzle piece $\triangle$ of $\Omega(T)$, the multi-set $U^{\circ}\cap\triangle$ of non-loop edges and segments in $\triangle$ is given as follows (c.f. Example \ref{ex:segment}):
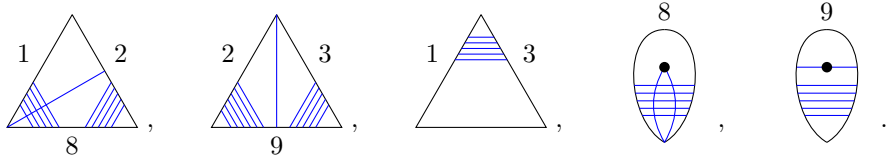
\[
\begin{tikzpicture}[baseline=-5mm]
\coordinate(l)at(-150:1);\coordinate(r)at(-30:1);\coordinate(u)at(90:1);
\draw(u)--node[above left]{$1$}(l)--node[below]{$8$}(r)--node[above right]{$2$}(u);
\draw[blue](l)--($(u)!0.5!(r)$);
\draw[blue]($(u)!0.6!(l)$)--($(l)!0.4!(r)$);\draw[blue]($(u)!0.65!(l)$)--($(l)!0.35!(r)$);\draw[blue]($(u)!0.7!(l)$)--($(l)!0.3!(r)$);\draw[blue]($(u)!0.75!(l)$)--($(l)!0.25!(r)$);\draw[blue]($(u)!0.8!(l)$)--($(l)!0.2!(r)$);
\draw[blue]($(u)!0.6!(r)$)--($(l)!0.6!(r)$);\draw[blue]($(u)!0.65!(r)$)--($(l)!0.65!(r)$);\draw[blue]($(u)!0.7!(r)$)--($(l)!0.7!(r)$);\draw[blue]($(u)!0.75!(r)$)--($(l)!0.75!(r)$);\draw[blue]($(u)!0.8!(r)$)--($(l)!0.8!(r)$);
\end{tikzpicture}
\ ,\hspace{7mm}
\begin{tikzpicture}[baseline=-5mm]
\coordinate(l)at(-150:1);\coordinate(r)at(-30:1);\coordinate(u)at(90:1);
\draw(u)--node[above left]{$2$}(l)--node[below]{$9$}(r)--node[above right]{$3$}(u);
\draw[blue](u)--($(l)!0.5!(r)$);
\draw[blue]($(u)!0.6!(l)$)--($(l)!0.4!(r)$);\draw[blue]($(u)!0.65!(l)$)--($(l)!0.35!(r)$);\draw[blue]($(u)!0.7!(l)$)--($(l)!0.3!(r)$);\draw[blue]($(u)!0.75!(l)$)--($(l)!0.25!(r)$);\draw[blue]($(u)!0.8!(l)$)--($(l)!0.2!(r)$);\draw[blue]($(u)!0.85!(l)$)--($(l)!0.15!(r)$);
\draw[blue]($(u)!0.6!(r)$)--($(l)!0.6!(r)$);\draw[blue]($(u)!0.65!(r)$)--($(l)!0.65!(r)$);\draw[blue]($(u)!0.7!(r)$)--($(l)!0.7!(r)$);\draw[blue]($(u)!0.75!(r)$)--($(l)!0.75!(r)$);\draw[blue]($(u)!0.8!(r)$)--($(l)!0.8!(r)$);
\end{tikzpicture}
\ ,\hspace{7mm}
\begin{tikzpicture}[baseline=-5mm]
\coordinate(l)at(-150:1);\coordinate(r)at(-30:1);\coordinate(u)at(90:1);
\draw(u)--node[above left]{$1$}(l)--node[below]{}(r)--node[above right]{$3$}(u);
\draw[blue]($(u)!0.4!(l)$)--($(u)!0.4!(r)$);\draw[blue]($(u)!0.35!(l)$)--($(u)!0.35!(r)$);\draw[blue]($(u)!0.3!(l)$)--($(u)!0.3!(r)$);\draw[blue]($(u)!0.25!(l)$)--($(u)!0.25!(r)$);\draw[blue]($(u)!0.2!(l)$)--($(u)!0.2!(r)$);
\end{tikzpicture}
\ ,\hspace{7mm}
\begin{tikzpicture}[baseline=-3mm]
\coordinate(u)at(90:1);\coordinate(p)at(0,0.5);\coordinate(d)at(-90:0.5);
\draw(d)to[out=150,in=180]coordinate[pos=0.3](l)coordinate[pos=0.25](l1)coordinate[pos=0.35](l2)coordinate[pos=0.4](l3)coordinate[pos=0.45](l4)(u);
\draw(d)to[out=30,in=0]coordinate[pos=0.3](r)coordinate[pos=0.25](r1)coordinate[pos=0.35](r2)coordinate[pos=0.4](r3)coordinate[pos=0.45](r4)(u)node[above]{$8$};
\draw[blue](l)--(r)(l1)--(r1)(l2)--(r2)(l3)--(r3)(l4)--(r4);
\draw[blue](d)to[out=30,in=150,relative](p)(d)to[out=-30,in=-150,relative](p);
\fill(p)circle(0.07);
\end{tikzpicture}
\ ,\hspace{7mm}
\begin{tikzpicture}[baseline=-3mm]
\coordinate(u)at(90:1);\coordinate(p)at(0,0.5);\coordinate(d)at(-90:0.5);
\draw(d)to[out=150,in=180]coordinate[pos=0.3](l)coordinate[pos=0.25](l1)coordinate[pos=0.35](l2)coordinate[pos=0.4](l3)coordinate[pos=0.45](l4)(u);
\draw(d)to[out=30,in=0]coordinate[pos=0.3](r)coordinate[pos=0.25](r1)coordinate[pos=0.35](r2)coordinate[pos=0.4](r3)coordinate[pos=0.45](r4)(u)node[above]{$9$};
\draw[blue](l)--(r)(l1)--(r1)(l2)--(r2)(l3)--(r3)(l4)--(r4);
\draw[blue](-0.4,0.5)--(p)--(0.4,0.5);
\fill(p)circle(0.07);
\end{tikzpicture}\ .
\]
The multi-set of curves obtained from $S_U$ by gluing segments simultaneously when we glue puzzle pieces of $\Omega(T)$ is given by
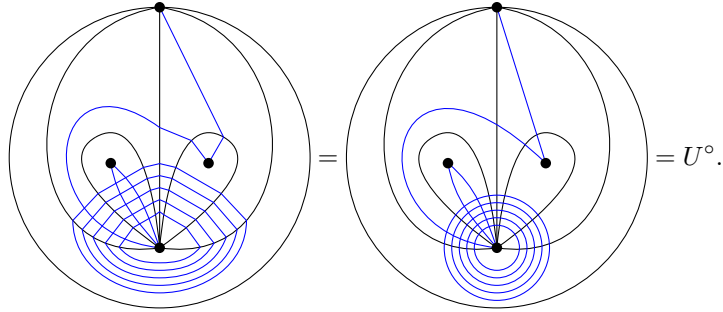
\[
\begin{tikzpicture}[baseline=-1mm]
\coordinate(d)at(0,-1.2);\coordinate(u)at(90:2);
\coordinate(l)at($(d)+(120:1.3)$);\coordinate(r)at($(d)+(60:1.3)$);
\draw(0,0)circle(2);
\draw(u)..controls(-2,1.7)and(-2,-1.5)..coordinate[pos=0.74](15)coordinate[pos=0.78](14)coordinate[pos=0.82](13)coordinate[pos=0.86](12)coordinate[pos=0.9](11)(d);
\draw(u)--coordinate[pos=0.5](26)coordinate[pos=0.65](25)coordinate[pos=0.7](24)coordinate[pos=0.75](23)coordinate[pos=0.8](22)coordinate[pos=0.85](21)(d);
\draw(u)..controls(2,1.7)and(2,-1.5)..coordinate[pos=0.74](35)coordinate[pos=0.78](34)coordinate[pos=0.82](33)coordinate[pos=0.86](32)coordinate[pos=0.9](31)(d);
\draw(d)to[out=25,in=90,relative]coordinate[pos=0.2](8l1)coordinate[pos=0.27](8l2)coordinate[pos=0.34](8l3)coordinate[pos=0.41](8l4)coordinate[pos=0.48](8l5)($(d)+(120:1.7)$);
\draw(d)to[out=-25,in=-90,relative]coordinate[pos=0.2](8r1)coordinate[pos=0.27](8r2)coordinate[pos=0.34](8r3)coordinate[pos=0.41](8r4)coordinate[pos=0.48](8r5)($(d)+(120:1.7)$);
\draw(d)to[out=25,in=90,relative]coordinate[pos=0.2](9l1)coordinate[pos=0.27](9l2)coordinate[pos=0.34](9l3)coordinate[pos=0.41](9l4)coordinate[pos=0.48](9l5)coordinate[pos=0.7](9l6)($(d)+(60:1.7)$);
\draw(d)to[out=-25,in=-90,relative]coordinate[pos=0.2](9r1)coordinate[pos=0.27](9r2)coordinate[pos=0.34](9r3)coordinate[pos=0.41](9r4)coordinate[pos=0.48](9r5)($(d)+(60:1.7)$);
\draw[blue](d)to[out=10,in=160,relative](l)(d)to[out=-10,in=-160,relative](l);
\draw[blue](11)--(8l1)--(8r1)--(21)--(9l1)--(9r1)--(31)to[out=-130,in=0](0,-1.4)to[out=180,in=-50](11);
\draw[blue](12)--(8l2)--(8r2)--(22)--(9l2)--(9r2)--(32)to[out=-120,in=0](0,-1.5)to[out=180,in=-60](12);
\draw[blue](13)--(8l3)--(8r3)--(23)--(9l3)--(9r3)--(33)to[out=-110,in=0](0,-1.6)to[out=180,in=-70](13);
\draw[blue](14)--(8l4)--(8r4)--(24)--(9l4)--(9r4)--(34)to[out=-100,in=0](0,-1.7)to[out=180,in=-80](14);
\draw[blue](15)--(8l5)--(8r5)--(25)--(9l5)--(9r5)--(35)to[out=-90,in=0](0,-1.8)to[out=180,in=-90](15);
\draw[blue](u)--($(d)+(60:1.7)$)--(r);
\draw[blue](r)--(9l6)--(26)..controls(-1.5,1.5)and(-1.8,-1)..(d);
\fill(l)circle(0.07);\fill(r)circle(0.07);\fill(d)circle(0.07);\fill(u)circle(0.07);
\end{tikzpicture}
=
\begin{tikzpicture}[baseline=-1mm]
\coordinate(d)at(0,-1.2);\coordinate(u)at(90:2);
\coordinate(l)at($(d)+(120:1.3)$);\coordinate(r)at($(d)+(60:1.3)$);
\draw(0,0)circle(2);
\draw(u)..controls(-2,1.7)and(-2,-1.5)..(d);\draw(u)--(d);\draw(u)..controls(2,1.7)and(2,-1.5)..(d);
\draw(d)to[out=25,in=90,relative]($(d)+(120:1.7)$);\draw(d)to[out=-25,in=-90,relative]($(d)+(120:1.7)$);
\draw(d)to[out=25,in=90,relative]($(d)+(60:1.7)$);\draw(d)to[out=-25,in=-90,relative]($(d)+(60:1.7)$);
\draw[blue](d)circle(0.3)circle(0.4)circle(0.5)circle(0.6)circle(0.7);
\draw[blue](u)--(r);
\draw[blue](d)..controls(-2,-1)and(-1.5,2)..(r);
\draw[blue](d)to[out=10,in=160,relative](l)(d)to[out=-10,in=-160,relative](l);
\fill(l)circle(0.07);\fill(r)circle(0.07);\fill(d)circle(0.07);\fill(u)circle(0.07);
\end{tikzpicture}=U^{\circ}.
\]
On the other hand, the desired multi-set $\Phi(S_U)$ in Theorem \ref{thm:segment} is the multi-set of all segments in Example \ref{ex:segment} and obtained from $S_U$ by the following local modifications defined in the next section:
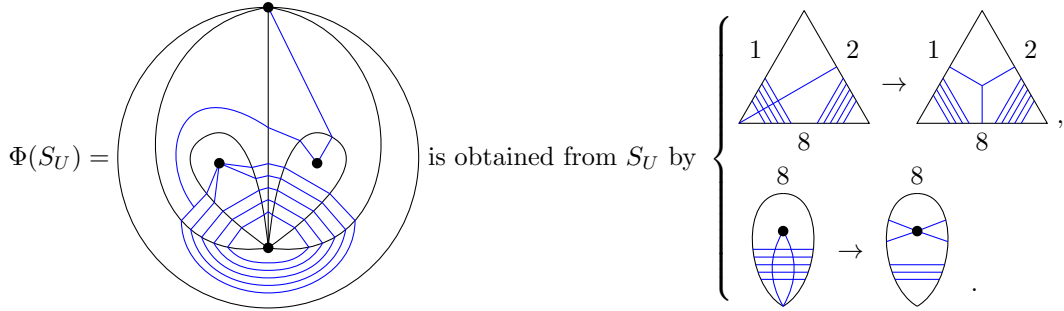
\[
\Phi(S_U)=
\begin{tikzpicture}[baseline=-1mm]
\coordinate(d)at(0,-1.2);\coordinate(u)at(90:2);
\coordinate(l)at($(d)+(120:1.3)$);\coordinate(r)at($(d)+(60:1.3)$);
\draw(0,0)circle(2);
\draw(u)..controls(-2,1.7)and(-2,-1.5)..coordinate[pos=0.74](15)coordinate[pos=0.78](14)coordinate[pos=0.82](13)coordinate[pos=0.86](12)coordinate[pos=0.9](11)(d);
\draw(u)--coordinate[pos=0.5](26)coordinate[pos=0.65](25)coordinate[pos=0.7](24)coordinate[pos=0.75](23)coordinate[pos=0.8](22)coordinate[pos=0.85](21)(d);
\draw(u)..controls(2,1.7)and(2,-1.5)..coordinate[pos=0.74](35)coordinate[pos=0.78](34)coordinate[pos=0.82](33)coordinate[pos=0.86](32)coordinate[pos=0.9](31)(d);
\draw(d)to[out=25,in=90,relative]coordinate[pos=0.2](8l1)coordinate[pos=0.27](8l2)coordinate[pos=0.34](8l3)coordinate[pos=0.41](8l4)coordinate[pos=0.48](8l5)($(d)+(120:1.7)$);
\draw(d)to[out=-25,in=-90,relative]coordinate[pos=0.2](8r1)coordinate[pos=0.27](8r2)coordinate[pos=0.34](8r3)coordinate[pos=0.41](8r4)coordinate[pos=0.48](8r5)($(d)+(120:1.7)$);
\draw(d)to[out=25,in=90,relative]coordinate[pos=0.2](9l1)coordinate[pos=0.27](9l2)coordinate[pos=0.34](9l3)coordinate[pos=0.41](9l4)coordinate[pos=0.48](9l5)coordinate[pos=0.7](9l6)($(d)+(60:1.7)$);
\draw(d)to[out=-25,in=-90,relative]coordinate[pos=0.2](9r1)coordinate[pos=0.27](9r2)coordinate[pos=0.34](9r3)coordinate[pos=0.41](9r4)coordinate[pos=0.48](9r5)($(d)+(60:1.7)$);
\draw[blue](11)--(8l1)--(8r1)--(21)--(9l1)--(9r1)--(31)to[out=-130,in=0](0,-1.4)to[out=180,in=-50](11);
\draw[blue](12)--(8l2)--(8r2)--(22)--(9l2)--(9r2)--(32)to[out=-120,in=0](0,-1.5)to[out=180,in=-60](12);
\draw[blue](13)--(8l3)--(8r3)--(23)--(9l3)--(9r3)--(33)to[out=-110,in=0](0,-1.6)to[out=180,in=-70](13);
\draw[blue](14)--(8l4)--(l)--(8r4)--(24)--(9l4)--(9r4)--(34)to[out=-100,in=0](0,-1.7)to[out=180,in=-80](14);
\draw[blue](15)--(8l5)--(l)--(8r5)--(25)--(9l5)--(9r5)--(35)to[out=-90,in=0](0,-1.8)to[out=180,in=-90](15);
\draw[blue](u)--($(d)+(60:1.7)$)--(r);
\draw[blue](r)--(9l6)--(26)..controls(-1.2,1.2)and(-1.5,0)..($(15)!0.5!(8l5)$);
\fill(l)circle(0.07);\fill(r)circle(0.07);\fill(d)circle(0.07);\fill(u)circle(0.07);
\end{tikzpicture}
\ \text{is obtained from $S_U$ by}\ 
\begin{cases}
\begin{tikzpicture}[baseline=-5mm]
\coordinate(l)at(-150:1);\coordinate(r)at(-30:1);\coordinate(u)at(90:1);
\draw(u)--node[above left]{$1$}(l)--node[below]{$8$}(r)--node[above right]{$2$}(u);
\draw[blue](l)--($(u)!0.5!(r)$);
\draw[blue]($(u)!0.6!(l)$)--($(l)!0.4!(r)$);\draw[blue]($(u)!0.65!(l)$)--($(l)!0.35!(r)$);\draw[blue]($(u)!0.7!(l)$)--($(l)!0.3!(r)$);\draw[blue]($(u)!0.75!(l)$)--($(l)!0.25!(r)$);\draw[blue]($(u)!0.8!(l)$)--($(l)!0.2!(r)$);
\draw[blue]($(u)!0.6!(r)$)--($(l)!0.6!(r)$);\draw[blue]($(u)!0.65!(r)$)--($(l)!0.65!(r)$);\draw[blue]($(u)!0.7!(r)$)--($(l)!0.7!(r)$);\draw[blue]($(u)!0.75!(r)$)--($(l)!0.75!(r)$);\draw[blue]($(u)!0.8!(r)$)--($(l)!0.8!(r)$);
\node at(1.2,0){$\rightarrow$};
\end{tikzpicture}
\begin{tikzpicture}[baseline=-5mm]
\coordinate(l)at(-150:1);\coordinate(r)at(-30:1);\coordinate(u)at(90:1);
\draw(u)--node[above left]{$1$}(l)--node[below]{$8$}(r)--node[above right]{$2$}(u);
\draw[blue](0,0)--($(u)!0.5!(l)$);\draw[blue](0,0)--($(u)!0.5!(r)$);\draw[blue](0,0)--($(l)!0.5!(r)$);
\draw[blue]($(u)!0.6!(l)$)--($(l)!0.4!(r)$);\draw[blue]($(u)!0.65!(l)$)--($(l)!0.35!(r)$);\draw[blue]($(u)!0.7!(l)$)--($(l)!0.3!(r)$);\draw[blue]($(u)!0.75!(l)$)--($(l)!0.25!(r)$);
\draw[blue]($(u)!0.6!(r)$)--($(l)!0.6!(r)$);\draw[blue]($(u)!0.65!(r)$)--($(l)!0.65!(r)$);\draw[blue]($(u)!0.7!(r)$)--($(l)!0.7!(r)$);\draw[blue]($(u)!0.75!(r)$)--($(l)!0.75!(r)$);\draw[blue]($(u)!0.8!(r)$)--($(l)!0.8!(r)$);
\end{tikzpicture}\ ,
\\
\begin{tikzpicture}[baseline=-3mm]
\coordinate(u)at(90:1);\coordinate(p)at(0,0.5);\coordinate(d)at(-90:0.5);
\draw(d)to[out=150,in=180]coordinate[pos=0.3](l)coordinate[pos=0.25](l1)coordinate[pos=0.35](l2)coordinate[pos=0.4](l3)coordinate[pos=0.45](l4)(u);
\draw(d)to[out=30,in=0]coordinate[pos=0.3](r)coordinate[pos=0.25](r1)coordinate[pos=0.35](r2)coordinate[pos=0.4](r3)coordinate[pos=0.45](r4)(u)node[above]{$8$};
\draw[blue](l)--(r)(l1)--(r1)(l2)--(r2)(l3)--(r3)(l4)--(r4);
\draw[blue](d)to[out=30,in=150,relative](p)(d)to[out=-30,in=-150,relative](p);
\fill(p)circle(0.07);
\node at(0.9,0.2){$\rightarrow$};
\end{tikzpicture}
\begin{tikzpicture}[baseline=-3mm]
\coordinate(u)at(90:1);\coordinate(p)at(0,0.5);\coordinate(d)at(-90:0.5);
\draw(d)to[out=150,in=180]coordinate[pos=0.3](l)coordinate[pos=0.25](l1)coordinate[pos=0.35](l2)coordinate[pos=0.5](l3)coordinate[pos=0.65](l4)(u);
\draw(d)to[out=30,in=0]coordinate[pos=0.3](r)coordinate[pos=0.25](r1)coordinate[pos=0.35](r2)coordinate[pos=0.5](r3)coordinate[pos=0.65](r4)(u)node[above]{$8$};
\draw[blue](l)--(r)(l1)--(r1)(l2)--(r2);
\draw[blue](l3)--(p)--(r3)(l4)--(p)--(r4);
\fill(p)circle(0.07);
\end{tikzpicture}\ .
\end{cases}
\]
\end{example}

\subsection{Graphs associated with tagged triangulations}\label{subsec:graph}

In the rest of this section, we prove Theorems \ref{thm:unique Int} and \ref{thm:list cS}. First, we briefly recall some notations in graph theory.

A (multi-)\emph{graph} is a pair $(V,E)$ consisting of a set $V$ of vertices and a set $E$ of edges each of which is an unordered pair of vertices, called its endpoints. We say that the graph is \emph{empty} if $V=E=\emptyset$. The \emph{degree} of $v\in V$ is the number of edges in $E$ incident to $v$, where a loop incident to $v$ contributes two to the degree of $v$.

A \emph{walk} is a sequence $e_1\cdots e_l$ of edges such that there are vertices $p_i$ and $p_{i+1}$ that are endpoints of $e_i$ for all $i$. It is called a \emph{path} if $p_i\neq p_j$ for distinct $i,j\in\{1,\ldots,l\}$. In addition, if $p_{l+1}=p_1$, then it is called a \emph{cycle}. Here, we say that they have \emph{odd length} (resp., \emph{even length}) if $l$ is odd (resp., even).

Let $T$ be a tagged triangulation of $\cS$ satisfying \eqref{diamond}. In particular, $\Omega(T)$ only consists of plain curves. To $T$, we associate a (possibly empty) graph as follows: We denote by $V_T$ the set of all punctures in $\cS$ incident to $\Omega(T)$, and by $E_T$ the set of all plain curves in $\Omega(T)$ connecting punctures. Then the pair $G_T:=(V_T,E_T)$ can naturally be considered as a graph. 

We also extend a rotation of tagged arcs defined in \cite{BQ15} (see Figure \ref{fig:rho}) to elements of $\bAL_{\cS}$.

\begin{defn}[\cite{BQ15}]\label{def:rotation}
For $\gamma\in\bAL_{\cS}$, its \emph{tagged rotation $\rho(\gamma)$} is defined as follows:
\begin{itemize}
 \item If $\gamma$ has an endpoint $o$ on a component $C$ of $\partial\cS$, then $\rho(\gamma)$ is obtained from $\gamma$ by moving $o$ to the next marked point on $C$ in the counterclockwise direction.
 \item If $\gamma$ has an endpoint at a puncture $p$, then $\rho(\gamma)$ is obtained from $\gamma$ by changing its tags at $p$.
\end{itemize}
\end{defn}

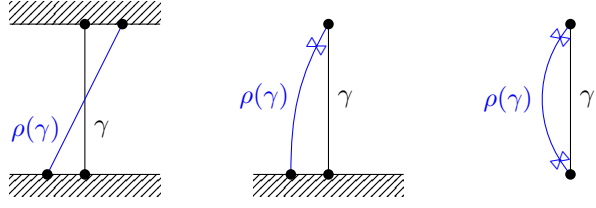
\begin{figure}[ht]
\begin{tikzpicture}[baseline=0mm]
\coordinate(u)at(0,1);\coordinate(d)at(0,-1);\coordinate(ru)at(0.5,1);\coordinate(ld)at(-0.5,-1);
\fill[pattern=north east lines]($(u)+(-1,0)$)rectangle($(u)+(1,0.3)$);\fill[pattern=north east lines]($(d)+(-1,0)$)rectangle($(d)+(1,-0.3)$);
\draw(-1,1)--(1,1) (-1,-1)--(1,-1);
\draw(u)--node[right,pos=0.7]{$\gamma$}(d);\draw[blue](ru)--node[left,pos=0.7]{$\rho(\gamma)$}(ld);
\fill(u)circle(0.07);\fill(d)circle(0.07);\fill(ru)circle(0.07);\fill(ld)circle(0.07);
\end{tikzpicture}
\hspace{7mm}
\begin{tikzpicture}[baseline=0mm]
\coordinate(u)at(0,1);\coordinate(d)at(0,-1);\coordinate(ld)at(-0.5,-1);
\fill[pattern=north east lines]($(d)+(-1,0)$)rectangle($(d)+(1,-0.3)$);
\draw(-1,-1)--(1,-1);
\draw(u)--node[right]{$\gamma$}(d);
\draw[blue](u)to[out=-120,in=90]node[left]{$\rho(\gamma)$}node[pos=0.15]{\rotatebox{-30}{\footnotesize $\bowtie$}}(ld);
\fill(u)circle(0.07);\fill(d)circle(0.07);\fill(ld)circle(0.07);
\end{tikzpicture}
\hspace{7mm}
\begin{tikzpicture}[baseline=0mm]
\coordinate(u)at(0,1);\coordinate(d)at(0,-1);
\draw(u)--node[right]{$\gamma$}(d);
\draw[blue](u)to[out=-130,in=130]node[left,pos=0.5]{$\rho(\gamma)$}node[pos=0.1]{\rotatebox{-40}{\footnotesize $\bowtie$}}node[pos=0.9]{\rotatebox{40}{\footnotesize $\bowtie$}}(d);
\fill(u)circle(0.07); \fill(d)circle(0.07);
\end{tikzpicture}
\caption{Plain curves $\gamma$ and their tagged rotations $\rho(\gamma)$}
\label{fig:rho}
\end{figure}

Note that for $\gamma\in E_T$, $\rho(\gamma)$ is the $2$-notched curve whose underlying plain curve is $\gamma$. Keeping the notations in the previous subsection, we give the following equivalence.

\begin{prop}\label{prop:equivalence graph}
Let $T$ be a tagged triangulation of $\cS$. Then the following are equivalent:
\begin{itemize}
\item[(1)] For any $U, V\in\bM_{\cS}$ with $U\cap T=V\cap T=\emptyset$, if $n(U_2,p)=n(V_2,p)$ for all punctures $p$, then $U_2=V_2$.
\item[(2)] Each connected component of $G_T$ contains at most one cycle of odd length and no cycles of even length.
\end{itemize}
\end{prop}

\begin{proof}
Since $U_2=\{\rho(\gamma)^{m_{U_2}(\rho(\gamma))}\mid \gamma\in E_T\}$, we only need to prove the desired equivalence under the assumption that $G_T$ is connected. Thus we assume it throughout this proof. We also note that $T$ satisfies (1) if and only if each $m_{U_2}(\rho(\gamma))$ is uniquely determined by $n(U_2,p)$ for all $p\in V_T$.

If there is a vertex $p\in V_T$ incident to exactly one edge $\gamma\in E_T$, then $m_{U_2}(\rho(\gamma))=n(U_2,p)$, and we remove $p$ and $\gamma$ from $G_T$. Repeating this process, we obtain a graph without vertices with degree one. If $G_T$ contains no cycles, then the resulting graph consists of at most one vertex and no edges. In which case, for each $\gamma\in E_T$, $m_{U_2}(\rho(\gamma))$ is described as a linear combination of $n(U_2,p)$ for all $p\in V_T$. Therefore, if $G_T$ contains no cycles, then $T$ satisfies (1).

Assume that $G_T$ satisfies (2) and contains cycles, that is, it contains exactly one cycle of odd length and no cycles of even length. Let $\gamma_1\cdots \gamma_{2l-1}$ be the unique cycle in $G_T$ for $l\in\bZ_{>0}$. Repeating the above process, for each $\gamma\in E_T\setminus\{\gamma_1,\ldots,\gamma_{2l-1}\}$, $m_{U_2}(\rho(\gamma))$ is described as a linear combination of $n(U_2,p)$ for all $p\in V_T$. Thus we assume that $G_T$ only consists of the cycle. Let $p_i$ and $p_{i+1}$ be the endpoints of $\gamma_i$ for all $i$. In particular, $p_{2l}=p_1$. Then since $n(U_2,p_i)=m_{U_2}(\rho(\gamma_{i-1}))+m_{U_2}(\rho(\gamma_i))$ for all $i$, where $\gamma_0:=\gamma_{2l-1}$, we have the equalities
\begin{align*}
n(U_2,p_1)+\sum_{i=1}^{l-1}n(U_2,p_{2i})
&=m_{U_2}(\rho(\gamma_0))+m_{U_2}(\rho(\gamma_1))+\sum_{i=1}^{2l-2}m_{U_2}(\rho(\gamma_i))\\
&=2m_{U_2}(\rho(\gamma_1))+\sum_{i=2}^{2l-1}m_{U_2}(\rho(\gamma_i))\\
&=2m_{U_2}(\rho(\gamma_1))+\sum_{i=1}^{l-1}n(U_2,p_{2i+1}).
\end{align*}
Therefore, $m_{U_2}(\rho(\gamma_1))$ is described as a linear combination of $n(U_2,p_i)$ for all $i$. Similarly, for each $k$, $m_{U_2}(\rho(\gamma_k))$ is also described as a linear combination of $n(U_2,p_i)$ for all $i$. Thus $T$ satisfies (1). Therefore, $T$ satisfies (1) if it satisfies (2).

Assume that $T$ does not satisfy (2). Then $G_T$ must contain at least one of the following: (a) A cycle of even length. (b) Two cycles of odd length. We show that $T$ does not satisfy (1) in each case. For that, we will give $U,V\in\bM_{\cS}$ such that $U_2\neq V_2$ and $n(U_2,p)=n(V_2,p)$ for all punctures $p$ in each case (see Figures \ref{fig:even}, \ref{fig:odd even}, and \ref{fig:odd odd}).

(a) Assume that $G_T$ contains a cycle $\gamma_1\cdots \gamma_{2l}$ for $l\in\bZ_{>0}$. Let $p_i$ and $p_{i+1}$ be the endpoints of $\gamma_i$ for all $i$. Take two multi-sets
\[
\Omega(U)=U_2=\{\rho(\gamma_{2i-1})\mid 1\le i\le l\}\ \text{ and }\ \Omega(V)=V_2=\{\rho(\gamma_{2i})\mid 1\le i\le l\}.
\]
Then $n(U_2,p)=n(V_2,p)$ for all punctures $p$ (see Figure \ref{fig:even}). Thus $T$ does not satisfy (1).

(b) Assume that $G_T$ contains two different cycles $\gamma_1\cdots \gamma_{2l-1}$ and $\delta_1\cdots \delta_{2m-1}$  for $l,m\in\bZ_{>0}$. Let $p_i$ and $p_{i+1}$ be the endpoints of $\gamma_i$, and $q_i$ and $q_{i+1}$ be the endpoints of $\delta_i$ for all $i$. The cycles contain one of the following: (b1) at least two common vertices; (b2) at most one common vertex.

(b1) Without loss of generality, we can assume that $p_1=q_1$ and $\gamma_1\neq \delta_1$. Then there exists
\[
k:=\min\left\{i\in\{2,\ldots,2l-1\}\mid\text{$p_i=q_j$ for some $j\in\{2,\ldots,2m-1\}$}\right\}.
\]
Let $h\in\{2,\ldots,2m-1\}$ with $q_h=p_k$. By the minimality of $k$, the walks $c_1:=\gamma_1\cdots \gamma_{k-1}\delta_{h-1}\delta_{h-2}\cdots \delta_1$ and $c_2:=\gamma_1\cdots \gamma_{k-1}\delta_h\delta_{h+1}\cdots \delta_{2m-1}$ are cycles. If both $k$ and $h$ are either odd or even, then $c_1$ has even length. If one of $k$ and $h$ is odd and the other is even, then $c_2$ has even length. Thus (b1) reduces to (a), and $T$ does not satisfy (1).

(b2) Since $G_T$ is connected, there is a path connecting the cycles. Without loss of generality, we can assume that there is a path $\varepsilon_1\cdots\varepsilon_n$ with endpoints $r_i$ and $r_{i+1}$ of $\varepsilon_i$ for all $i$ such that $r_1=p_1$ and $r_{n+1}=q_1$, where $n\in\bZ_{\ge 0}$ and $n=0$ means that the cycles have exactly one common vertex $p_1=q_1$. When $n$ is even, we take two multi-sets (see Figure \ref{fig:odd even})
\begin{align*}
\Omega(U)&=U_2=\left\{\rho(\gamma_{2i-1})\mid 1\le i\le l\right\}\sqcup\left\{\rho(\delta_{2i})\mid 1\le i\le m-1\right\}\sqcup\left\{\rho(\varepsilon_{2i})^2\mid 1\le i\le \frac{n}{2}\right\},\\
\Omega(V)&=V_2=\left\{\rho(\gamma_{2i})\mid 1\le i\le l-1\right\}\sqcup\left\{\rho(\delta_{2i-1})\mid 1\le i\le m\right\}\sqcup\left\{\rho(\varepsilon_{2i-1})^2\mid 1\le i\le \frac{n}{2}\right\};
\end{align*}
when $n$ is odd, we take two multi-sets (see Figure \ref{fig:odd odd})
\begin{align*}
\Omega(U)&=U_2=\left\{\rho(\gamma_{2i-1})\mid 1\le i\le l\right\}\sqcup\left\{\rho(\delta_{2i-1})\mid 1\le i\le m\right\}\sqcup\left\{\rho(\varepsilon_{2i})^2\mid 1\le i\le \frac{n-1}{2}\right\},\\
\Omega(V)&=V_2=\left\{\rho(\gamma_{2i})\mid 1\le i\le l-1\right\}\sqcup\left\{\rho(\delta_{2i})\mid 1\le i\le m-1\right\}\sqcup\left\{\rho(\varepsilon_{2i-1})^2\mid 1\le i\le \frac{n+1}{2}\right\}.
\end{align*}
In both cases, $n(U_2,p)=n(V_2,p)$ for all punctures $p$. Thus $T$ does not satisfy (1).
\end{proof}

%
%
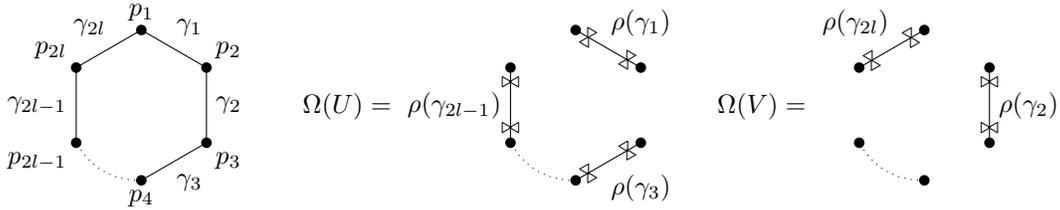
\begin{figure}[ht]
\[
\begin{tikzpicture}[baseline=-1mm]
\coordinate(1)at(90:1);\coordinate(2)at(30:1);\coordinate(3)at(-30:1);\coordinate(4)at(-90:1);
\coordinate(2l)at(150:1);\coordinate(2l-1)at(210:1);
\draw(1)--node[above right,pos=0.4]{$\gamma_1$}(2)--node[right]{$\gamma_2$}(3)--node[below right,pos=0.6]{$\gamma_3$}(4);
\draw(1)--node[above left,pos=0.4]{$\gamma_{2l}$}(2l)--node[left]{$\gamma_{2l-1}$}(2l-1);
\draw[dotted](4)to[out=180,in=-60](2l-1);
\fill(1)circle(0.07)node[above]{$p_1$};\fill(2)circle(0.07)node[above right]{$p_2$};\fill(3)circle(0.07)node[below right]{$p_3$};
\fill(4)circle(0.07)node[below]{$p_4$};\fill(2l)circle(0.07)node[above left]{$p_{2l}$};\fill(2l-1)circle(0.07)node[below left]{$p_{2l-1}$};
\end{tikzpicture}
\hspace{7mm}
\Omega(U)=
\begin{tikzpicture}[baseline=-1mm]
\coordinate(1)at(90:1);\coordinate(2)at(30:1);\coordinate(3)at(-30:1);\coordinate(4)at(-90:1);
\coordinate(2l)at(150:1);\coordinate(2l-1)at(210:1);
\draw(1)--node[above right,pos=0.4]{$\rho(\gamma_1)$}node[pos=0.2]{\rotatebox{60}{\footnotesize$\bowtie$}}node[pos=0.8]{\rotatebox{60}{\footnotesize$\bowtie$}}(2);
\draw(3)--node[below right,pos=0.6]{$\rho(\gamma_3)$}node[pos=0.2]{\rotatebox{120}{\footnotesize$\bowtie$}}node[pos=0.8]{\rotatebox{120}{\footnotesize$\bowtie$}}(4);
\draw(2l)--node[left]{$\rho(\gamma_{2l-1})$}node[pos=0.2]{\rotatebox{0}{\footnotesize$\bowtie$}}node[pos=0.8]{\rotatebox{0}{\footnotesize$\bowtie$}}(2l-1);
\draw[dotted](4)to[out=180,in=-60](2l-1);
\fill(1)circle(0.07);\fill(2)circle(0.07);\fill(3)circle(0.07);\fill(4)circle(0.07);\fill(2l)circle(0.07);\fill(2l-1)circle(0.07);
\end{tikzpicture}
\hspace{5mm}
\Omega(V)=
\begin{tikzpicture}[baseline=-1mm]
\coordinate(1)at(90:1);\coordinate(2)at(30:1);\coordinate(3)at(-30:1);\coordinate(4)at(-90:1);
\coordinate(2l)at(150:1);\coordinate(2l-1)at(210:1);
\draw(2)--node[right]{$\rho(\gamma_2)$}node[pos=0.2]{\rotatebox{0}{\footnotesize$\bowtie$}}node[pos=0.8]{\rotatebox{0}{\footnotesize$\bowtie$}}(3);
\draw(1)--node[above left,pos=0.4]{$\rho(\gamma_{2l})$}node[pos=0.2]{\rotatebox{120}{\footnotesize$\bowtie$}}node[pos=0.8]{\rotatebox{120}{\footnotesize$\bowtie$}}(2l);
\draw[dotted](4)to[out=180,in=-60](2l-1);
\fill(1)circle(0.07);\fill(2)circle(0.07);\fill(3)circle(0.07);\fill(4)circle(0.07);\fill(2l)circle(0.07);\fill(2l-1)circle(0.07);
\end{tikzpicture}
\]
\caption{Multi-sets $\Omega(U)$ and $\Omega(V)$ with $n(U_2,p)=n(V_2,p)$ for all punctures $p$ in the case that $G_T$ contains a cycle of even length}
\label{fig:even}
\end{figure}

%
%
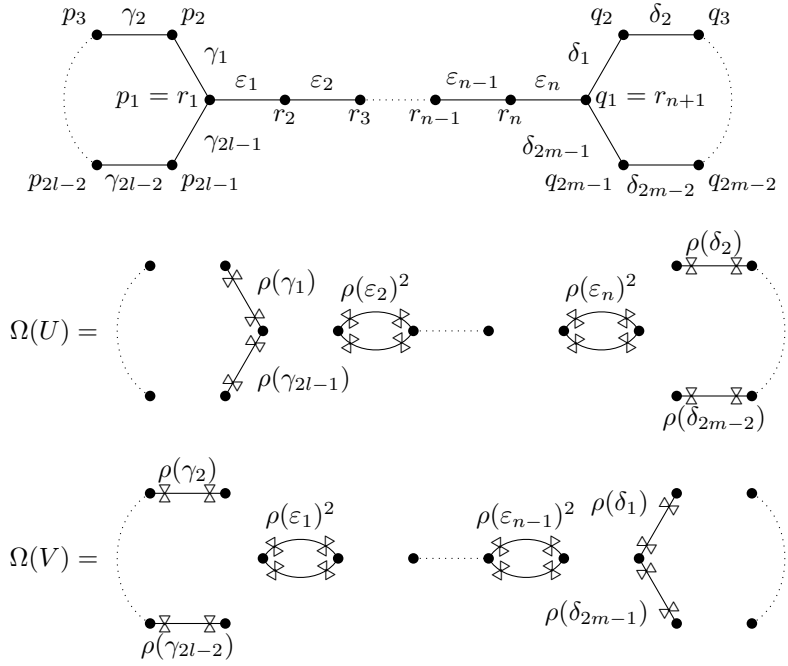
\begin{figure}[ht]
\[
\begin{tikzpicture}[baseline=-1mm]
\coordinate(p1)at(0:1);\coordinate(p2)at(60:1);\coordinate(p3)at(120:1);\coordinate(p2l-1)at(-60:1);\coordinate(p2l-2)at(-120:1);
\coordinate(r2)at($(p1)+(0:1)$);\coordinate(r3)at($(r2)+(0:1)$);\coordinate(rn-1)at($(r3)+(0:1)$);\coordinate(rn)at($(rn-1)+(0:1)$);
\coordinate(q1)at($(rn)+(0:1)$);\coordinate(q2)at($(q1)+(60:1)$);\coordinate(q3)at($(q2)+(0:1)$);
\coordinate(q2m-1)at($(q1)+(-60:1)$);\coordinate(q2m-2)at($(q2m-1)+(0:1)$);
\draw(p1)--node[above right,pos=0.4]{$\gamma_1$}(p2)--node[above]{$\gamma_2$}(p3);
\draw(p1)--node[below right,pos=0.4]{$\gamma_{2l-1}$}(p2l-1)--node[below]{$\gamma_{2l-2}$}(p2l-2);
\draw(p1)--node[above]{$\varepsilon_1$}(r2)--node[above]{$\varepsilon_2$}(r3);
\draw[dotted](r3)--(rn-1);
\draw(rn-1)--node[above]{$\varepsilon_{n-1}$}(rn)--node[above]{$\varepsilon_n$}(q1);
\draw(q1)--node[above left,pos=0.4]{$\delta_1$}(q2)--node[above]{$\delta_2$}(q3);
\draw(q1)--node[below left,pos=0.4]{$\delta_{2m-1}$}(q2m-1)--node[below]{$\delta_{2m-2}$}(q2m-2);
\draw[dotted](p3)to[out=-150,in=150](p2l-2);\draw[dotted](q3)to[out=-30,in=30](q2m-2);
\fill(p1)circle(0.07)node[left]{$p_1=r_1$};\fill(p2)circle(0.07)node[above right]{$p_2$};\fill(p3)circle(0.07)node[above left]{$p_3$};
\fill(p2l-1)circle(0.07)node[below right]{$p_{2l-1}$};\fill(p2l-2)circle(0.07)node[below left]{$p_{2l-2}$};
\fill(r2)circle(0.07)node[below]{$r_2$};\fill(r3)circle(0.07)node[below]{$r_3$};
\fill(rn-1)circle(0.07)node[below]{$r_{n-1}$};\fill(rn)circle(0.07)node[below]{$r_n$};
\fill(q1)circle(0.07)node[right]{$q_1=r_{n+1}$};\fill(q2)circle(0.07)node[above left]{$q_2$};\fill(q3)circle(0.07)node[above right]{$q_3$};
\fill(q2m-1)circle(0.07)node[below left]{$q_{2m-1}$};\fill(q2m-2)circle(0.07)node[below right]{$q_{2m-2}$};
\end{tikzpicture}
\]
\[
\Omega(U)=
\begin{tikzpicture}[baseline=-1mm]
\coordinate(p1)at(0:1);\coordinate(p2)at(60:1);\coordinate(p3)at(120:1);\coordinate(p2l-1)at(-60:1);\coordinate(p2l-2)at(-120:1);
\coordinate(r2)at($(p1)+(0:1)$);\coordinate(r3)at($(r2)+(0:1)$);\coordinate(rn-1)at($(r3)+(0:1)$);\coordinate(rn)at($(rn-1)+(0:1)$);
\coordinate(q1)at($(rn)+(0:1)$);\coordinate(q2)at($(q1)+(60:1)$);\coordinate(q3)at($(q2)+(0:1)$);
\coordinate(q2m-1)at($(q1)+(-60:1)$);\coordinate(q2m-2)at($(q2m-1)+(0:1)$);
\draw(p1)--node[above right,pos=0.4]{$\rho(\gamma_1)$}node[pos=0.2]{\rotatebox{30}{\footnotesize$\bowtie$}}node[pos=0.8]{\rotatebox{30}{\footnotesize$\bowtie$}}(p2);
\draw(p1)--node[below right,pos=0.4]{$\rho(\gamma_{2l-1})$}node[pos=0.2]{\rotatebox{-30}{\footnotesize$\bowtie$}}node[pos=0.8]{\rotatebox{-30}{\footnotesize$\bowtie$}}(p2l-1);
\draw(r2)to[out=60,in=120]node[above]{$\rho(\varepsilon_2)^2$}node[pos=0.2]{\rotatebox{120}{\footnotesize$\bowtie$}}node[pos=0.8]{\rotatebox{60}{\footnotesize$\bowtie$}}(r3);
\draw(r2)to[out=-60,in=-120]node[pos=0.2]{\rotatebox{-120}{\footnotesize$\bowtie$}}node[pos=0.8]{\rotatebox{-60}{\footnotesize$\bowtie$}}(r3);
\draw[dotted](r3)--(rn-1);
\draw(rn)to[out=60,in=120]node[above]{$\rho(\varepsilon_n)^2$}node[pos=0.2]{\rotatebox{120}{\footnotesize$\bowtie$}}node[pos=0.8]{\rotatebox{60}{\footnotesize$\bowtie$}}(q1);
\draw(rn)to[out=-60,in=-120]node[pos=0.2]{\rotatebox{-120}{\footnotesize$\bowtie$}}node[pos=0.8]{\rotatebox{-60}{\footnotesize$\bowtie$}}(q1);
\draw(q2)--node[above]{$\rho(\delta_2)$}node[pos=0.2]{\rotatebox{90}{\footnotesize$\bowtie$}}node[pos=0.8]{\rotatebox{90}{\footnotesize$\bowtie$}}(q3);
\draw(q2m-1)--node[below]{$\rho(\delta_{2m-2})$}node[pos=0.2]{\rotatebox{90}{\footnotesize$\bowtie$}}node[pos=0.8]{\rotatebox{90}{\footnotesize$\bowtie$}}(q2m-2);
\draw[dotted](p3)to[out=-150,in=150](p2l-2);\draw[dotted](q3)to[out=-30,in=30](q2m-2);
\fill(p1)circle(0.07);\fill(p2)circle(0.07);\fill(p3)circle(0.07);\fill(p2l-1)circle(0.07);\fill(p2l-2)circle(0.07);
\fill(r2)circle(0.07);\fill(r3)circle(0.07);\fill(rn-1)circle(0.07);\fill(rn)circle(0.07);
\fill(q1)circle(0.07);\fill(q2)circle(0.07);\fill(q3)circle(0.07);\fill(q2m-1)circle(0.07);\fill(q2m-2)circle(0.07);
\end{tikzpicture}
\]
\[
\Omega(V)=
\begin{tikzpicture}[baseline=-1mm]
\coordinate(p1)at(0:1);\coordinate(p2)at(60:1);\coordinate(p3)at(120:1);\coordinate(p2l-1)at(-60:1);\coordinate(p2l-2)at(-120:1);
\coordinate(r2)at($(p1)+(0:1)$);\coordinate(r3)at($(r2)+(0:1)$);\coordinate(rn-1)at($(r3)+(0:1)$);\coordinate(rn)at($(rn-1)+(0:1)$);
\coordinate(q1)at($(rn)+(0:1)$);\coordinate(q2)at($(q1)+(60:1)$);\coordinate(q3)at($(q2)+(0:1)$);
\coordinate(q2m-1)at($(q1)+(-60:1)$);\coordinate(q2m-2)at($(q2m-1)+(0:1)$);
\draw(p2)--node[above]{$\rho(\gamma_2)$}node[pos=0.2]{\rotatebox{90}{\footnotesize$\bowtie$}}node[pos=0.8]{\rotatebox{90}{\footnotesize$\bowtie$}}(p3);
\draw(p2l-1)--node[below]{$\rho(\gamma_{2l-2})$}node[pos=0.2]{\rotatebox{90}{\footnotesize$\bowtie$}}node[pos=0.8]{\rotatebox{90}{\footnotesize$\bowtie$}}(p2l-2);
\draw(p1)to[out=60,in=120]node[above]{$\rho(\varepsilon_1)^2$}node[pos=0.2]{\rotatebox{120}{\footnotesize$\bowtie$}}node[pos=0.8]{\rotatebox{60}{\footnotesize$\bowtie$}}(r2);
\draw(p1)to[out=-60,in=-120]node[pos=0.2]{\rotatebox{-120}{\footnotesize$\bowtie$}}node[pos=0.8]{\rotatebox{-60}{\footnotesize$\bowtie$}}(r2);
\draw[dotted](r3)--(rn-1);
\draw(rn-1)to[out=60,in=120]node[above]{$\rho(\varepsilon_{n-1})^2$}node[pos=0.2]{\rotatebox{120}{\footnotesize$\bowtie$}}node[pos=0.8]{\rotatebox{60}{\footnotesize$\bowtie$}}(rn);
\draw(rn-1)to[out=-60,in=-120]node[pos=0.2]{\rotatebox{-120}{\footnotesize$\bowtie$}}node[pos=0.8]{\rotatebox{-60}{\footnotesize$\bowtie$}}(rn);
\draw(q1)--node[above left]{$\rho(\delta_1)$}node[pos=0.2]{\rotatebox{-30}{\footnotesize$\bowtie$}}node[pos=0.8]{\rotatebox{-30}{\footnotesize$\bowtie$}}(q2);
\draw(q1)--node[below left]{$\rho(\delta_{2m-1})$}node[pos=0.2]{\rotatebox{30}{\footnotesize$\bowtie$}}node[pos=0.8]{\rotatebox{30}{\footnotesize$\bowtie$}}(q2m-1);
\draw[dotted](p3)to[out=-150,in=150](p2l-2);\draw[dotted](q3)to[out=-30,in=30](q2m-2);
\fill(p1)circle(0.07);\fill(p2)circle(0.07);\fill(p3)circle(0.07);\fill(p2l-1)circle(0.07);\fill(p2l-2)circle(0.07);
\fill(r2)circle(0.07);\fill(r3)circle(0.07);\fill(rn-1)circle(0.07);\fill(rn)circle(0.07);
\fill(q1)circle(0.07);\fill(q2)circle(0.07);\fill(q3)circle(0.07);\fill(q2m-1)circle(0.07);\fill(q2m-2)circle(0.07);
\end{tikzpicture}
\]
\caption{Multi-sets $\Omega(U)$ and $\Omega(V)$ with $n(U_2,p)=n(V_2,p)$ for all punctures $p$ in the case that $G_T$ contains two cycles of odd length connected by a path of even length}
\label{fig:odd even}
\end{figure}

%
%
\begin{figure}[ht]
\[
\Omega(U)=
\begin{tikzpicture}[baseline=-1mm]
\coordinate(p1)at(0:1);\coordinate(p2)at(60:1);\coordinate(p3)at(120:1);\coordinate(p2l-1)at(-60:1);\coordinate(p2l-2)at(-120:1);
\coordinate(r2)at($(p1)+(0:1)$);\coordinate(r3)at($(r2)+(0:1)$);\coordinate(rn-1)at($(r3)+(0:1)$);\coordinate(rn)at($(rn-1)+(0:1)$);
\coordinate(q1)at($(rn)+(0:1)$);\coordinate(q2)at($(q1)+(60:1)$);\coordinate(q3)at($(q2)+(0:1)$);
\coordinate(q2m-1)at($(q1)+(-60:1)$);\coordinate(q2m-2)at($(q2m-1)+(0:1)$);
\draw(p1)--node[above right,pos=0.4]{$\rho(\gamma_1)$}node[pos=0.2]{\rotatebox{30}{\footnotesize$\bowtie$}}node[pos=0.8]{\rotatebox{30}{\footnotesize$\bowtie$}}(p2);
\draw(p1)--node[below right,pos=0.4]{$\rho(\gamma_{2l-1})$}node[pos=0.2]{\rotatebox{-30}{\footnotesize$\bowtie$}}node[pos=0.8]{\rotatebox{-30}{\footnotesize$\bowtie$}}(p2l-1);
\draw(r2)to[out=60,in=120]node[above]{$\rho(\varepsilon_2)^2$}node[pos=0.2]{\rotatebox{120}{\footnotesize$\bowtie$}}node[pos=0.8]{\rotatebox{60}{\footnotesize$\bowtie$}}(r3);
\draw(r2)to[out=-60,in=-120]node[pos=0.2]{\rotatebox{-120}{\footnotesize$\bowtie$}}node[pos=0.8]{\rotatebox{-60}{\footnotesize$\bowtie$}}(r3);
\draw[dotted](r3)--(rn-1);
\draw(rn-1)to[out=60,in=120]node[above]{$\rho(\varepsilon_{n-1})^2$}node[pos=0.2]{\rotatebox{120}{\footnotesize$\bowtie$}}node[pos=0.8]{\rotatebox{60}{\footnotesize$\bowtie$}}(rn);
\draw(rn-1)to[out=-60,in=-120]node[pos=0.2]{\rotatebox{-120}{\footnotesize$\bowtie$}}node[pos=0.8]{\rotatebox{-60}{\footnotesize$\bowtie$}}(rn);
\draw(q1)--node[above left]{$\rho(\delta_1)$}node[pos=0.2]{\rotatebox{-30}{\footnotesize$\bowtie$}}node[pos=0.8]{\rotatebox{-30}{\footnotesize$\bowtie$}}(q2);
\draw(q1)--node[below left]{$\rho(\delta_{2m-1})$}node[pos=0.2]{\rotatebox{30}{\footnotesize$\bowtie$}}node[pos=0.8]{\rotatebox{30}{\footnotesize$\bowtie$}}(q2m-1);
\draw[dotted](p3)to[out=-150,in=150](p2l-2);\draw[dotted](q3)to[out=-30,in=30](q2m-2);
\fill(p1)circle(0.07);\fill(p2)circle(0.07);\fill(p3)circle(0.07);\fill(p2l-1)circle(0.07);\fill(p2l-2)circle(0.07);
\fill(r2)circle(0.07);\fill(r3)circle(0.07);\fill(rn-1)circle(0.07);\fill(rn)circle(0.07);
\fill(q1)circle(0.07);\fill(q2)circle(0.07);\fill(q3)circle(0.07);\fill(q2m-1)circle(0.07);\fill(q2m-2)circle(0.07);
\end{tikzpicture}
\]
\[
\Omega(V)=
\begin{tikzpicture}[baseline=-1mm]
\coordinate(p1)at(0:1);\coordinate(p2)at(60:1);\coordinate(p3)at(120:1);\coordinate(p2l-1)at(-60:1);\coordinate(p2l-2)at(-120:1);
\coordinate(r2)at($(p1)+(0:1)$);\coordinate(r3)at($(r2)+(0:1)$);\coordinate(rn-1)at($(r3)+(0:1)$);\coordinate(rn)at($(rn-1)+(0:1)$);
\coordinate(q1)at($(rn)+(0:1)$);\coordinate(q2)at($(q1)+(60:1)$);\coordinate(q3)at($(q2)+(0:1)$);
\coordinate(q2m-1)at($(q1)+(-60:1)$);\coordinate(q2m-2)at($(q2m-1)+(0:1)$);
\draw(p2)--node[above]{$\rho(\gamma_2)$}node[pos=0.2]{\rotatebox{90}{\footnotesize$\bowtie$}}node[pos=0.8]{\rotatebox{90}{\footnotesize$\bowtie$}}(p3);
\draw(p2l-1)--node[below]{$\rho(\gamma_{2l-2})$}node[pos=0.2]{\rotatebox{90}{\footnotesize$\bowtie$}}node[pos=0.8]{\rotatebox{90}{\footnotesize$\bowtie$}}(p2l-2);
\draw(p1)to[out=60,in=120]node[above]{$\rho(\varepsilon_1)^2$}node[pos=0.2]{\rotatebox{120}{\footnotesize$\bowtie$}}node[pos=0.8]{\rotatebox{60}{\footnotesize$\bowtie$}}(r2);
\draw(p1)to[out=-60,in=-120]node[pos=0.2]{\rotatebox{-120}{\footnotesize$\bowtie$}}node[pos=0.8]{\rotatebox{-60}{\footnotesize$\bowtie$}}(r2);
\draw[dotted](r3)--(rn-1);
\draw(rn)to[out=60,in=120]node[above]{$\rho(\varepsilon_n)^2$}node[pos=0.2]{\rotatebox{120}{\footnotesize$\bowtie$}}node[pos=0.8]{\rotatebox{60}{\footnotesize$\bowtie$}}(q1);
\draw(rn)to[out=-60,in=-120]node[pos=0.2]{\rotatebox{-120}{\footnotesize$\bowtie$}}node[pos=0.8]{\rotatebox{-60}{\footnotesize$\bowtie$}}(q1);
\draw(q2)--node[above]{$\rho(\delta_2)$}node[pos=0.2]{\rotatebox{90}{\footnotesize$\bowtie$}}node[pos=0.8]{\rotatebox{90}{\footnotesize$\bowtie$}}(q3);
\draw(q2m-1)--node[below]{$\rho(\delta_{2m-2})$}node[pos=0.2]{\rotatebox{90}{\footnotesize$\bowtie$}}node[pos=0.8]{\rotatebox{90}{\footnotesize$\bowtie$}}(q2m-2);
\draw[dotted](p3)to[out=-150,in=150](p2l-2);\draw[dotted](q3)to[out=-30,in=30](q2m-2);
\fill(p1)circle(0.07);\fill(p2)circle(0.07);\fill(p3)circle(0.07);\fill(p2l-1)circle(0.07);\fill(p2l-2)circle(0.07);
\fill(r2)circle(0.07);\fill(r3)circle(0.07);\fill(rn-1)circle(0.07);\fill(rn)circle(0.07);
\fill(q1)circle(0.07);\fill(q2)circle(0.07);\fill(q3)circle(0.07);\fill(q2m-1)circle(0.07);\fill(q2m-2)circle(0.07);
\end{tikzpicture}
\]
\caption{Multi-sets $\Omega(U)$ and $\Omega(V)$ with $n(U_2,p)=n(V_2,p)$ for all punctures $p$ in the case that $G_T$ contains two cycles of odd length connected by a path of odd length (cf. Figure \ref{fig:odd even})}
\label{fig:odd odd}
\end{figure}
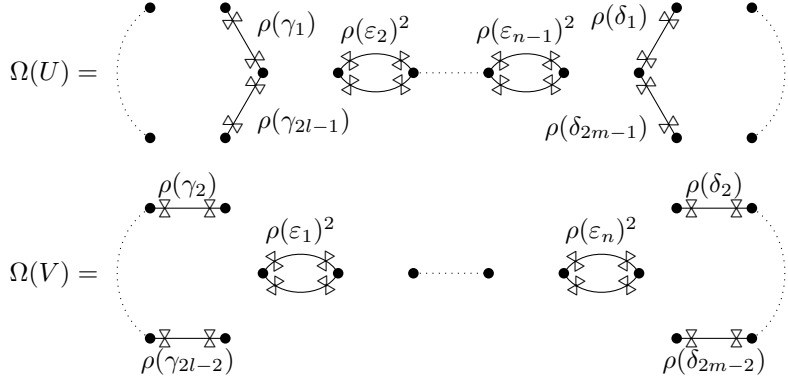

Now, we are ready to prove Theorems \ref{thm:unique Int} and \ref{thm:list cS}.

\begin{proof}[Proof of Theorem \ref{thm:unique Int}]
The assertion follows from Theorem \ref{thm:main} and Proposition \ref{prop:equivalence graph}.
\end{proof}

\begin{proof}[Proof of Theorem \ref{thm:list cS}]
Assume that the boundary of $\cS'$ is empty. Let $T$ be a tagged triangulation of $\cS'$. If there is a triangle piece of $\Omega(T)$ with a loop as an edge, then the other edges form either a cycle of length two or two loops in $G_T$. Thus $T$ does not satisfy Theorem \ref{thm:unique Int}(2). We assume that no such triangle piece exists. Let $\triangle$ and $\triangle'$ be adjacent puzzle pieces of $\Omega(T)$. We consider all cases of them:
\begin{itemize}
\item[(a)] Both of them are monogon pieces.
\item[(b)] One of them is a monogon piece and the other is a triangle piece.
\item[(c1)] Both of them are triangle pieces with exactly one common edge.
\item[(c2)] Both of them are triangle pieces with exactly two common edges.
\item[(c3)] Both of them are triangle pieces with three common edges.
\end{itemize}
In the cases (a) and (c3), $\cS'$ must be a sphere with exactly three punctures, thus it contradicts our assumption. In the cases (b), (c1), and (c2), their edges except for the common edges form a cycle of even length. Therefore, $T$ does not satisfy Theorem \ref{thm:unique Int}(2) in all cases. As a consequence, all tagged triangulations of $\cS'$ do not satisfy Theorem \ref{thm:unique Int}(2) if the boundary of $\cS'$ is empty.

Now, we assume that the boundary of $\cS'$ is not empty in the rest of the proof. Then we can take a tagged triangulation $T$ of $\cS'$ such that each puncture is enclosed by a punctured loop in $\Omega(T)$. Since $G_T$ is empty, it clearly satisfies Theorem \ref{thm:unique Int}(2). Therefore, the assertion (1) holds.

Finally, if $\cS'$ is one of the list in (2), then it is easy to check that $G_T$ has at most one cycle and its length is one for any tagged triangulation $T$ of $\cS'$. Thus it satisfies Theorem \ref{thm:unique Int}(2). In other cases, $\cS'$ satisfies at least one of the following:
\begin{itemize}
\item The genus is one or more and there is at least one puncture.
\item There are at least one boundary component and at least three punctures.
\item There are at least three boundary components and at least one puncture.
\item There are two boundary components and two punctures.
\end{itemize}
Then we can take a tagged triangulation $T$ of $\cS'$ such that a subgraph of $G_T$ consists of one vertex and two loops in the above three cases; a subgraph of $G_T$ consists of two vertices and two edges connecting them in the last case (see Figure \ref{fig:small graph}). Since it does not satisfy Theorem \ref{thm:unique Int}(2), the assertion (2) holds.
\end{proof}

%
%
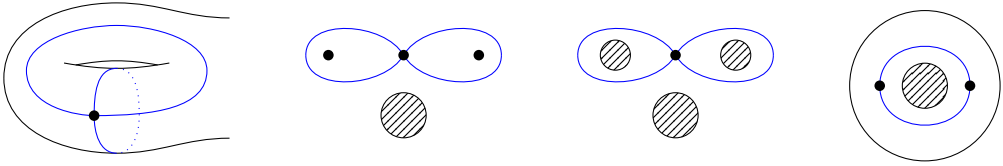
\begin{figure}[ht]
\[
\begin{tikzpicture}[baseline=-3mm]
\coordinate(p)at(-0.3,0);
\draw(-0.7,0.7)to[out=-10,in=-170]coordinate[pos=0.1](a)coordinate[pos=0.5](c)coordinate[pos=0.9](b)(0.7,0.7);
\draw(a)to[out=10,in=170](b);
\draw(-1.5,0.5)to[out=-90,in=-180](0,-0.5)to[out=0,in=-180](1.5,-0.3);
\draw(-1.5,0.5)to[out=90,in=-180](0,1.5)to[out=0,in=-180](1.5,1.3);
\draw[blue](p)..controls(-0.5,0)and(-1.2,0.1)..(-1.2,0.6)..controls(-1.2,1)and(-0.5,1.2)..(0,1.2)..controls(0.5,1.2)and(1.2,1)..(1.2,0.6)..controls(1.2,0.1)and(0.5,0)..(p);
\draw[blue](p)to[out=90,in=180](c);\draw[blue](p)to[out=-90,in=180](0,-0.5);
\draw[blue,dotted](c)to[out=0,in=90](0.3,0)to[out=-90,in=0](0,-0.5);
\fill(p)circle(0.07);
\end{tikzpicture}
\hspace{10mm}
\begin{tikzpicture}
\coordinate(0)at(0,0);\coordinate(p)at(0,0.5);\coordinate(l)at($(p)+(-1,0)$);\coordinate(r)at($(p)+(1,0)$);
\draw[pattern=north east lines](0,-0.3)circle(0.3);
\draw[blue](p)to[out=120,in=90]($(l)+(-0.3,0)$);\draw[blue](p)to[out=-120,in=-90]($(l)+(-0.3,0)$);
\draw[blue](p)to[out=60,in=90]($(r)+(0.3,0)$);\draw[blue](p)to[out=-60,in=-90]($(r)+(0.3,0)$);
\fill(p)circle(0.07);\fill(l)circle(0.07);\fill(r)circle(0.07);
\end{tikzpicture}
\hspace{10mm}
\begin{tikzpicture}
\coordinate(0)at(0,0);\coordinate(p)at(0,0.5);\coordinate(l)at($(p)+(-0.8,0)$);\coordinate(r)at($(p)+(0.8,0)$);
\draw[pattern=north east lines](0,-0.3)circle(0.3);
\draw[pattern=north east lines](l)circle(0.2);\draw[pattern=north east lines](r)circle(0.2);
\draw[blue](p)to[out=120,in=90]($(p)+(-1.3,0)$);\draw[blue](p)to[out=-120,in=-90]($(p)+(-1.3,0)$);
\draw[blue](p)to[out=60,in=90]($(p)+(1.3,0)$);\draw[blue](p)to[out=-60,in=-90]($(p)+(1.3,0)$);
\fill(p)circle(0.07);
\end{tikzpicture}
\hspace{10mm}
\begin{tikzpicture}[baseline=-7mm]
\coordinate(0)at(0,0);\coordinate(l)at(-0.6,0);\coordinate(r)at(0.6,0);
\draw[pattern=north east lines](0)circle(0.3);\draw(0)circle(1);
\draw[blue](l)..controls($(l)+(90:0.7)$)and($(r)+(90:0.7)$)..(r);\draw[blue](l)..controls($(l)+(-90:0.7)$)and($(r)+(-90:0.7)$)..(r);
\fill(l)circle(0.07);\fill(r)circle(0.07);
\end{tikzpicture}
\]
\caption{Subgraphs of $G_T$ that do not satisfy Theorem \ref{thm:unique Int}(2) in the case that the marked surface has non-empty boundary and does not appear in Theorem \ref{thm:list cS}(2)}
\label{fig:small graph}
\end{figure}

\begin{example}\label{ex:graph}
In the setting of Example \ref{ex:T Int}, the graph $G_T=(V_T,E_T)$ consists of the bottom puncture $p$ and two punctured loops $8$ and $9$ in $\Omega(T)$, that is, it contains two cycles of odd length. Thus $T$ does not satisfy the equivalent properties in Theorem \ref{thm:unique Int} and Proposition \ref{prop:equivalence graph}. In fact, for $U=\{\rho(4),\rho(5)\}$ and $V=\{\rho(6),\rho(7)\}$ in $\bM_{\cS}$, $\Omega(U)=U_2=\{\rho(8)\}$ and $\Omega(V)=V_2=\{\rho(9)\}$ are different, but $n(U_2,p)=n(V_2,p)=2$ as follows:
\[
G_T=
\begin{tikzpicture}[baseline=-1mm,scale=0.5]
\coordinate(0)at(0,0);\coordinate(d)at(0,-1.2);\coordinate(u)at(90:2);
\coordinate(l)at($(d)+(120:1.3)$);\coordinate(r)at($(d)+(60:1.3)$);
\draw[blue](d)to[out=30,in=90,relative]($(d)+(120:1.7)$);
\draw[blue](d)to[out=-30,in=-90,relative]($(d)+(120:1.7)$)node[above]{$8$};
\draw[blue](d)to[out=30,in=90,relative]($(d)+(60:1.7)$);
\draw[blue](d)to[out=-30,in=-90,relative]($(d)+(60:1.7)$)node[above]{$9$};
\fill(d)circle(0.14)node[below]{$p$};
\end{tikzpicture}
\ ,\hspace{3mm}\Omega(U)=U_2=
\begin{tikzpicture}[baseline=-1mm,scale=0.5]
\coordinate(0)at(0,0);\coordinate(d)at(0,-1.2);\coordinate(u)at(90:2);
\coordinate(l)at($(d)+(120:1.3)$);\coordinate(r)at($(d)+(60:1.3)$);
\draw(0,0)circle(2);
\draw[blue](d)to[out=40,in=90,relative]node[pos=0.2]{\scriptsize\rotatebox{55}{$\bowtie$}}($(d)+(120:1.7)$);
\draw[blue](d)to[out=-40,in=-90,relative]node[pos=0.2]{\scriptsize\rotatebox{5}{$\bowtie$}}($(d)+(120:1.7)$)node[above]{$\rho(8)$};
\fill(l)circle(0.14);\fill(r)circle(0.14);\fill(d)circle(0.14);\fill(u)circle(0.14);
\end{tikzpicture}
\ ,\hspace{3mm}\Omega(V)=V_2=
\begin{tikzpicture}[baseline=-1mm,scale=0.5]
\coordinate(0)at(0,0);\coordinate(d)at(0,-1.2);\coordinate(u)at(90:2);
\coordinate(l)at($(d)+(120:1.3)$);\coordinate(r)at($(d)+(60:1.3)$);
\draw(0,0)circle(2);
\draw[blue](d)to[out=40,in=90,relative]node[pos=0.2]{\scriptsize\rotatebox{-5}{$\bowtie$}}($(d)+(60:1.7)$);
\draw[blue](d)to[out=-40,in=-90,relative]node[pos=0.2]{\scriptsize\rotatebox{-55}{$\bowtie$}}($(d)+(60:1.7)$)node[above]{$\rho(9)$};
\fill(l)circle(0.14);\fill(r)circle(0.14);\fill(d)circle(0.14);\fill(u)circle(0.14);
\end{tikzpicture}\ .
\]
In particular, $\Int_T(U)=\Int_T(V)=(2,2,2,2,2,2,2)$.
\end{example}

\section{Modifications}\label{sec:modif}

In this section, we introduce and study a notion of modifications for the multi-set $S_U$ defined in Subsection \ref{subsec:pf thm:main} and prove Theorem \ref{thm:segment}. In Subsections \ref{subsec:modif tri} to \ref{subsec:recover}, we discuss in more general cases, and in Subsection \ref{subsec:proof} we apply them to $S_U$. We first consider local modifications. In fact, we give modifications on a triangle piece in Subsection \ref{subsec:modif tri}, and on a monogon piece in Subsection \ref{subsec:modif mono}. Their examples have already appeared in Example \ref{ex:glue}.

\subsection{Modifications at angles on a triangle piece}\label{subsec:modif tri}

Let $\triangle$ be a triangle piece, and $S$ be a multi-set of edges and segments in $\triangle$. Fix an angle $a$ of $\triangle$. We denote by $a^l$ (resp., $a^r$) the left (resp., right) angle of $\triangle$ with $a$ at the top as follows:
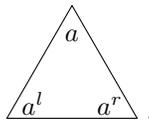
\[
\begin{tikzpicture}[baseline=-5mm]
\coordinate(l)at(-150:1);\coordinate(r)at(-30:1);\coordinate(u)at(90:1);\draw(u)--(l)--(r)--(u);
\node at(90:0.6){$a$};\node at($(l)+(30:0.4)$){$a^l$};\node at($(r)+(155:0.4)$){$a^r$};
\end{tikzpicture}\ .
\]

We consider the following condition of $S$:
\begin{equation*}
\tag{$\phi,a$} \text{$m_S(h_a)>0$ and $\Int(h_a,S)>0$}.
\end{equation*}
It is easy to check that $S$ satisfies exactly one of the following conditions when it satisfies $(\phi,a)$:
\begin{enumerate}
\item[$(\phi 1)$] $m_S(e_{a^l})>0$, $m_S(e_{a^r})>0$.
\item[$(\phi 2)$] $m_S(e_{a^l})>0$, $m_S(e_{a^r})=0$, $m_S(v_a)>0$.
\item[$(\phi 3)$] $m_S(e_{a^l})>0$, $m_S(e_{a^r})=0$, $m_S(v_a)=0$.
\item[$(\phi 4)$] $m_S(e_{a^l})=0$, $m_S(e_{a^r})>0$, $m_S(v_a)>0$.
\item[$(\phi 5)$] $m_S(e_{a^l})=0$, $m_S(e_{a^r})>0$, $m_S(v_a)=0$.
\item[$(\phi 6)$] $m_S(e_{a^l})=0$, $m_S(e_{a^r})=0$, $m_S(v_a)\ge2$.
\item[$(\phi 7)$] $m_S(e_{a^l})=0$, $m_S(e_{a^r})=0$, $m_S(v_a)=1$.
\end{enumerate}
We say that $S$ satisfies $(\phi,a)_k$ if it satisfies $(\phi,a)$ and $(\phi k)$ for $k\in\{1,\ldots,7\}$ (see Table \ref{table:smallest}). When $S$ satisfies $(\phi,a)$, we define $\phi_a(S)$ as the multi-set of edges and segments in $\triangle$ whose multiplicities satisfy Table \ref{table:def phi}. In particular, $\phi_a$ sends the second row to the fourth row in Table \ref{table:smallest}.

%
%
\renewcommand{\arraystretch}{1.3}
{\begin{table}[ht]
\[
\begin{tabular}{c||c|c|c|c|c|c|c}
$X$&$(\phi,a)_1$&$(\phi,a)_2$&$(\phi,a)_3$&$(\phi,a)_4$&$(\phi,a)_5$&$(\phi,a)_6$&$(\phi,a)_7$\\\hline\hline
$S$&
\begin{tikzpicture}[baseline=1mm]
\coordinate(l)at(-150:1);\coordinate(r)at(-30:1);\coordinate(u)at(90:1);\draw(u)--(l)--(r)--(u);
\draw[blue](u)to[out=20,in=160,relative](l);\draw[blue](u)to[out=-20,in=-160,relative](r);\draw[blue]($(u)!0.5!(l)$)--($(u)!0.5!(r)$);
\end{tikzpicture}&
\begin{tikzpicture}[baseline=1mm]
\coordinate(l)at(-150:1);\coordinate(r)at(-30:1);\coordinate(u)at(90:1);\draw(u)--(l)--(r)--(u);
\draw[blue](u)--($(l)!0.5!(r)$);\draw[blue](u)to[out=-20,in=-160,relative](r);\draw[blue]($(u)!0.5!(l)$)--($(u)!0.5!(r)$);
\end{tikzpicture}&
\begin{tikzpicture}[baseline=1mm]
\coordinate(l)at(-150:1);\coordinate(r)at(-30:1);\coordinate(u)at(90:1);\draw(u)--(l)--(r)--(u);
\draw[blue](u)to[out=-20,in=-160,relative](r);\draw[blue]($(u)!0.5!(l)$)--($(u)!0.5!(r)$);
\end{tikzpicture}&
\begin{tikzpicture}[baseline=1mm]
\coordinate(l)at(-150:1);\coordinate(r)at(-30:1);\coordinate(u)at(90:1);\draw(u)--(l)--(r)--(u);
\draw[blue](u)to[out=20,in=160,relative](l);\draw[blue](u)--($(l)!0.5!(r)$);\draw[blue]($(u)!0.5!(l)$)--($(u)!0.5!(r)$);
\end{tikzpicture}&
\begin{tikzpicture}[baseline=1mm]
\coordinate(l)at(-150:1);\coordinate(r)at(-30:1);\coordinate(u)at(90:1);\draw(u)--(l)--(r)--(u);
\draw[blue](u)to[out=20,in=160,relative](l);\draw[blue]($(u)!0.5!(l)$)--($(u)!0.5!(r)$);
\end{tikzpicture}&
\begin{tikzpicture}[baseline=1mm]
\coordinate(l)at(-150:1);\coordinate(r)at(-30:1);\coordinate(u)at(90:1);\draw(u)--(l)--(r)--(u);
\draw[blue](u)--($(l)!0.45!(r)$);\draw[blue](u)--($(l)!0.55!(r)$);\draw[blue]($(u)!0.5!(l)$)--($(u)!0.5!(r)$);
\end{tikzpicture}&
\begin{tikzpicture}[baseline=1mm]
\coordinate(l)at(-150:1);\coordinate(r)at(-30:1);\coordinate(u)at(90:1);\draw(u)--(l)--(r)--(u);
\draw[blue](u)--($(l)!0.5!(r)$);\draw[blue]($(u)!0.5!(l)$)--($(u)!0.5!(r)$);
\node at(0,1.1){};
\end{tikzpicture}
\end{tabular}
\vspace{3mm}\]\[
\begin{tabular}{c||c|c|c|c|c|c|c}
$X$&$(\psi,a)_1$&$(\psi,a)_2$&$(\psi,a)_3$&$(\psi,a)_4$&$(\psi,a)_5$&$(\psi,a)_6$&$(\psi,a)_7$\\\hline\hline
$S$&
\begin{tikzpicture}[baseline=1mm]
\coordinate(l)at(-150:1);\coordinate(r)at(-30:1);\coordinate(u)at(90:1);\draw(u)--(l)--(r)--(u);
\end{tikzpicture}&
\begin{tikzpicture}[baseline=1mm]
\coordinate(l)at(-150:1);\coordinate(r)at(-30:1);\coordinate(u)at(90:1);\draw(u)--(l)--(r)--(u);
\draw[blue]($(u)!0.5!(l)$)--($(l)!0.5!(r)$);
\end{tikzpicture}&
\begin{tikzpicture}[baseline=1mm]
\coordinate(l)at(-150:1);\coordinate(r)at(-30:1);\coordinate(u)at(90:1);\draw(u)--(l)--(r)--(u);
\draw[blue]($(u)!0.5!(l)$)--(r);
\end{tikzpicture}&
\begin{tikzpicture}[baseline=1mm]
\coordinate(l)at(-150:1);\coordinate(r)at(-30:1);\coordinate(u)at(90:1);\draw(u)--(l)--(r)--(u);
\draw[blue]($(u)!0.5!(r)$)--($(l)!0.5!(r)$);
\end{tikzpicture}&
\begin{tikzpicture}[baseline=1mm]
\coordinate(l)at(-150:1);\coordinate(r)at(-30:1);\coordinate(u)at(90:1);\draw(u)--(l)--(r)--(u);
\draw[blue]($(u)!0.5!(r)$)--(l);
\end{tikzpicture}&
\begin{tikzpicture}[baseline=1mm]
\coordinate(l)at(-150:1);\coordinate(r)at(-30:1);\coordinate(u)at(90:1);\draw(u)--(l)--(r)--(u);
\draw[blue]($(u)!0.5!(l)$)--($(l)!0.45!(r)$);\draw[blue]($(u)!0.5!(r)$)--($(l)!0.55!(r)$);
\end{tikzpicture}&
\begin{tikzpicture}[baseline=1mm]
\coordinate(l)at(-150:1);\coordinate(r)at(-30:1);\coordinate(u)at(90:1);\draw(u)--(l)--(r)--(u);
\draw[blue](0,0)--($(u)!0.5!(l)$);\draw[blue](0,0)--($(l)!0.5!(r)$);\draw[blue](0,0)--($(u)!0.5!(r)$);
\node at(0,1.05){};
\end{tikzpicture}
\end{tabular}\vspace{3mm}\]
\caption{The smallest multi-set $S$ that satisfies the condition $X$ on a triangle piece, where $a$ is the angle at the top of each triangle}
\label{table:smallest}
\end{table}}

%
%
\renewcommand{\arraystretch}{1.3}
{\begin{table}[ht]
\begin{tabular}{c||c|c|c|c|c|c|c|c|c|c}
$k\backslash s$&$e_a$&$e_{a^l}$&$e_{a^r}$&$h_a$&$h_{a^l}$&$h_{a^r}$&$v_a$&$v_{a^l}$&$v_{a^r}$&$y$\\\hline\hline
$1$&$0$&$-1$&$-1$ & $-1$&$0$&$0$ & $0$&$0$&$0$ & $0$\\\hline
$2$&$0$&$-1$&$0$ & $-1$&$1$&$0$ & $-1$&$0$&$0$ & $0$\\\hline
$3$&$0$&$-1$&$0$ & $-1$&$0$&$0$ & $0$&$0$&$1$ & $0$\\\hline
$4$&$0$&$0$&$-1$ & $-1$&$0$&$1$ & $-1$&$0$&$0$ & $0$\\\hline
$5$&$0$&$0$&$-1$ & $-1$&$0$&$0$ & $0$&$1$&$0$ & $0$\\\hline
$6$&$0$&$0$&$0$ & $-1$&$1$&$1$ & $-2$&$0$&$0$ & $0$\\\hline
$7$&$0$&$0$&$0$ & $-1$&$0$&$0$ & $-1$&$0$&$0$ & $1$
\end{tabular}\vspace{3mm}
\caption{The difference $m_{\phi_a(S)}(s)-m_S(s)$ when $S$ satisfies $(\phi,a)_k$ on a triangle piece}
\label{table:def phi}
\end{table}}

\begin{lem}\label{lem:Int tri}
Assume that $S$ satisfies $(\phi,a)$. Then $\Int(s,\phi_a(S))\le\Int(s,S)$ for an edge or segment $s$ in $\triangle$. In particular, $\Int(e_a,\phi_a(S))=\Int(e_a,S)$ and $m_{\phi_a(S)}(e_a)=m_S(e_a)$. Moreover, for $e\in\{e_{a^l},e_{a^r}\}$, the following hold:
\begin{itemize}
\item[(1)] If $m_S(e)=0$, then $\Int(e,\phi_a(S))=\Int(e,S)$ and $m_{\phi_a(S)}(e)=m_S(e)=0$.
\item[(2)] If $m_S(e)>0$, then $\Int(e,\phi_a(S))=\Int(e,S)-1$ and $m_{\phi_a(S)}(e)=m_S(e)-1$.
\end{itemize}
\end{lem}

\begin{proof}
The assertions follow from the definition of $\phi_a$ (see also Table \ref{table:smallest}).
\end{proof}

Notice that $S$ is not recoverable from $\phi_a(S)$ in general as in the following example.

\begin{example}
The multi-sets $S_1=\{e_{a^l},e_{a^r},h_a,h_{a^l}\}$ and $S_2=\{e_{a^l},h_a,v_a\}$ satisfy $(\phi,a)_1$ and $(\phi,a)_2$, respectively. Then $\phi_a(S_1)=\phi_a(S_2)=\{h_{a^l}\}$ as follows:
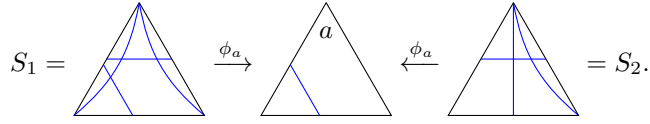
\[
S_1=
\begin{tikzpicture}[baseline=1mm]
\coordinate(l)at(-150:1);\coordinate(r)at(-30:1);\coordinate(u)at(90:1);\draw(u)--(l)--(r)--(u);
\draw[blue](u)to[out=20,in=160,relative](l);\draw[blue](u)to[out=-20,in=-160,relative](r);
\draw[blue]($(u)!0.5!(l)$)--($(u)!0.5!(r)$);\draw[blue]($(u)!0.55!(l)$)--($(l)!0.45!(r)$);
\end{tikzpicture}
\xrightarrow{\phi_a}
\begin{tikzpicture}[baseline=1mm]
\coordinate(l)at(-150:1);\coordinate(r)at(-30:1);\coordinate(u)at(90:1);\draw(u)--(l)--(r)--(u);
\draw[blue]($(u)!0.55!(l)$)--($(l)!0.45!(r)$);
\node at(90:0.6){$a$};
\end{tikzpicture}
\xleftarrow{\phi_a}
\begin{tikzpicture}[baseline=1mm]
\coordinate(l)at(-150:1);\coordinate(r)at(-30:1);\coordinate(u)at(90:1);\draw(u)--(l)--(r)--(u);
\draw[blue](u)--($(l)!0.5!(r)$);\draw[blue](u)to[out=-20,in=-160,relative](r);\draw[blue]($(u)!0.5!(l)$)--($(u)!0.5!(r)$);
\end{tikzpicture}=S_2.
\]
\end{example}

To give a sufficient condition of $S$ such that it is recoverable from $\phi_a(S)$, we consider a map sending the fourth row to the second row in Table \ref{table:smallest}. First, we consider the following condition of $S$:
\begin{equation*}
\tag{$\psi,a$} \text{If $\Int(s,t)>0$ for two segments $s$ and $t$ in $S$, then $\{s,t\}=\{h_a,v_a\}$}.
\end{equation*}
It is easy to check that $S$ satisfies exactly one of the following conditions when it satisfies $(\psi,a)$:
\begin{enumerate}
\item[$(\psi1)$] $m_S(h_{a^l})=m_S(h_{a^r})=m_S(v_{a^l})=m_S(v_{a^r})=m_S(y)=0$.
\item[$(\psi2)$] $m_S(h_{a^l})>0$, $m_S(h_{a^r})=m_S(v_{a^r})=m_S(y)=0$.
\item[$(\psi3)$] $m_S(v_{a^r})>0$.
\item[$(\psi4)$] $m_S(h_{a^r})>0$, $m_S(h_{a^l})=m_S(v_{a^l})=m_S(y)=0$.
\item[$(\psi5)$] $m_S(v_{a^l})>0$.
\item[$(\psi6)$] $m_S(h_{a^l})>0$, $m_S(h_{a^r})>0$, $m_S(y)=0$.
\item[$(\psi7)$] $m_S(y)=1$.
\end{enumerate}
We say that $S$ satisfies $(\psi,a)_k$ if it satisfies $(\psi,a)$ and $(\psi k)$ for $k\in\{1,\ldots,7\}$ (see Tables \ref{table:smallest} and \ref{table:max psi}). When $S$ satisfies $(\psi,a)$, we define $\psi_a(S)$ as the multi-set of edges and segments in $\triangle$ whose multiplicities satisfy Table \ref{table:def psi} (cf. Table \ref{table:def phi}). Then $\psi_a$ sends the fourth row to the second row in Table \ref{table:smallest}.

%
%
\renewcommand{\arraystretch}{1.3}
{\begin{table}[ht]\[
\begin{tabular}{c||c|c|c|c|c|c|c}
$k$&$1$&$2$&$3$&$4$&$5$&$6$&$7$\\\hline\hline
$S$&
\begin{tikzpicture}[baseline=1mm]
\coordinate(l)at(-150:1);\coordinate(r)at(-30:1);\coordinate(u)at(90:1);\draw(u)--(l)--(r)--(u);
\draw[blue](u)--($(l)!0.5!(r)$);\draw[blue]($(u)!0.45!(l)$)--($(u)!0.45!(r)$);
\end{tikzpicture}&
\begin{tikzpicture}[baseline=1mm]
\coordinate(l)at(-150:1);\coordinate(r)at(-30:1);\coordinate(u)at(90:1);\draw(u)--(l)--(r)--(u);
\draw[blue](u)--($(l)!0.5!(r)$);\draw[blue]($(u)!0.45!(l)$)--($(u)!0.45!(r)$);
\draw[blue]($(u)!0.55!(l)$)--($(l)!0.45!(r)$);
\end{tikzpicture}&
\begin{tikzpicture}[baseline=1mm]
\coordinate(l)at(-150:1);\coordinate(r)at(-30:1);\coordinate(u)at(90:1);\draw(u)--(l)--(r)--(u);
\draw[blue]($(u)!0.45!(l)$)--($(u)!0.45!(r)$);\draw[blue]($(u)!0.55!(l)$)--($(l)!0.45!(r)$);
\draw[blue]($(u)!0.5!(l)$)--(r);
\end{tikzpicture}&
\begin{tikzpicture}[baseline=1mm]
\coordinate(l)at(-150:1);\coordinate(r)at(-30:1);\coordinate(u)at(90:1);\draw(u)--(l)--(r)--(u);
\draw[blue](u)--($(l)!0.5!(r)$);\draw[blue]($(u)!0.45!(l)$)--($(u)!0.45!(r)$);
\draw[blue]($(u)!0.55!(r)$)--($(l)!0.55!(r)$);
\end{tikzpicture}&
\begin{tikzpicture}[baseline=1mm]
\coordinate(l)at(-150:1);\coordinate(r)at(-30:1);\coordinate(u)at(90:1);\draw(u)--(l)--(r)--(u);
\draw[blue]($(u)!0.45!(l)$)--($(u)!0.45!(r)$);\draw[blue]($(u)!0.55!(r)$)--($(l)!0.55!(r)$);
\draw[blue]($(u)!0.5!(r)$)--(l);
\end{tikzpicture}&
\begin{tikzpicture}[baseline=1mm]
\coordinate(l)at(-150:1);\coordinate(r)at(-30:1);\coordinate(u)at(90:1);\draw(u)--(l)--(r)--(u);
\draw[blue](u)--($(l)!0.5!(r)$);\draw[blue]($(u)!0.45!(l)$)--($(u)!0.45!(r)$);
\draw[blue]($(u)!0.55!(l)$)--($(l)!0.45!(r)$);\draw[blue]($(u)!0.55!(r)$)--($(l)!0.55!(r)$);
\end{tikzpicture}&
\begin{tikzpicture}[baseline=1mm]
\coordinate(l)at(-150:1);\coordinate(r)at(-30:1);\coordinate(u)at(90:1);\draw(u)--(l)--(r)--(u);
\draw[blue](0,0)--($(u)!0.5!(l)$);\draw[blue](0,0)--($(l)!0.5!(r)$);\draw[blue](0,0)--($(u)!0.5!(r)$);
\draw[blue]($(u)!0.45!(l)$)--($(u)!0.45!(r)$);\draw[blue]($(u)!0.55!(l)$)--($(l)!0.45!(r)$);\draw[blue]($(u)!0.55!(r)$)--($(l)!0.55!(r)$);
\node at(0,1.05){};
\end{tikzpicture}
\end{tabular}\vspace{3mm}\]
\caption{The maximal set $S$ of segments that satisfies $(\psi,a)_k$ (or $(\psi\ast,a)_k$) on a triangle piece, where $a$ is the angle at the top of each triangle}
\label{table:max psi}
\end{table}}

%
%
\renewcommand{\arraystretch}{1.3}
{\begin{table}[ht]
\begin{tabular}{c||c|c|c|c|c|c|c|c|c|c}
$k\backslash s$&$e_a$&$e_{a^l}$&$e_{a^r}$&$h_a$&$h_{a^l}$&$h_{a^r}$&$v_a$&$v_{a^l}$&$v_{a^r}$&$y$\\\hline\hline
$1$&$0$&$1$&$1$ & $1$&$0$&$0$ & $0$&$0$&$0$ & $0$\\\hline
$2$&$0$&$1$&$0$ & $1$&$-1$&$0$ & $1$&$0$&$0$ & $0$\\\hline
$3$&$0$&$1$&$0$ & $1$&$0$&$0$ & $0$&$0$&$-1$ & $0$\\\hline
$4$&$0$&$0$&$1$ & $1$&$0$&$-1$ & $1$&$0$&$0$ & $0$\\\hline
$5$&$0$&$0$&$1$ & $1$&$0$&$0$ & $0$&$-1$&$0$ & $0$\\\hline
$6$&$0$&$0$&$0$ & $1$&$-1$&$-1$ & $2$&$0$&$0$ & $0$\\\hline
$7$&$0$&$0$&$0$ & $1$&$0$&$0$ & $1$&$0$&$0$ & $-1$
\end{tabular}\vspace{3mm}
\caption{The difference $m_{\psi_a(S)}(s)-m_S(s)$ when $S$ satisfies $(\psi,a)_k$ on a triangle piece}
\label{table:def psi}
\end{table}}

When $S$ satisfies $(\phi,a)$, $\phi_a(S)$ does not always satisfy $(\psi,a)$. As a sufficient condition of $S$ such that $\phi_a(S)$ satisfies $(\psi,a)$, we give the following condition:
\begin{equation*}
\tag{$\ast,a$} \text{If $\Int(s,t)>0$ for $s\in S\cap\{e_{a^l},e_{a^r},v_a\}$ and $t\in S$, then $t=h_a$}.
\end{equation*}
For short, we say that $S$ satisfies $(\upsilon\ast,a)$ (resp., $(\upsilon\ast,a)_k$) if it satisfies $(\upsilon,a)$ (resp., $(\upsilon,a)_k$) and $(\ast,a)$ for $\upsilon\in\{\phi,\psi\}$ (see Tables \ref{table:max psi} and \ref{table:max phiast}).

%
%
\renewcommand{\arraystretch}{1.3}
{\begin{table}[ht]\[
\begin{tabular}{c||c|c|c|c|c|c}
$k$&$1$&$2$&$3$&$4$&$5$&$6$,$7$\\\hline\hline
$S$&
\begin{tikzpicture}[baseline=1mm]
\coordinate(l)at(-150:1);\coordinate(r)at(-30:1);\coordinate(u)at(90:1);\draw(u)--(l)--(r)--(u);
\draw[blue](u)--($(l)!0.5!(r)$);\draw[blue]($(u)!0.45!(l)$)--($(u)!0.45!(r)$);
\draw[blue](u)to[out=20,in=160,relative](l);\draw[blue](u)to[out=-20,in=-160,relative](r);
\end{tikzpicture}
\hspace{-1mm}\text{or}\hspace{-1mm}
\begin{tikzpicture}[baseline=1mm]
\coordinate(l)at(-150:1);\coordinate(r)at(-30:1);\coordinate(u)at(90:1);\draw(u)--(l)--(r)--(u);
\draw[blue]($(u)!0.45!(l)$)--($(u)!0.45!(r)$);
\draw[blue](u)to[out=20,in=160,relative](l);\draw[blue](u)to[out=-20,in=-160,relative](r);\draw[blue](l)to[out=20,in=160,relative](r);
\end{tikzpicture}&
\begin{tikzpicture}[baseline=1mm]
\coordinate(l)at(-150:1);\coordinate(r)at(-30:1);\coordinate(u)at(90:1);\draw(u)--(l)--(r)--(u);
\draw[blue](u)--($(l)!0.5!(r)$);\draw[blue]($(u)!0.45!(l)$)--($(u)!0.45!(r)$);
\draw[blue]($(u)!0.55!(l)$)--($(l)!0.45!(r)$);
\draw[blue](u)to[out=-20,in=-160,relative](r);
\end{tikzpicture}&
\begin{tikzpicture}[baseline=1mm]
\coordinate(l)at(-150:1);\coordinate(r)at(-30:1);\coordinate(u)at(90:1);\draw(u)--(l)--(r)--(u);
\draw[blue]($(u)!0.45!(l)$)--($(u)!0.45!(r)$);\draw[blue]($(u)!0.55!(l)$)--($(l)!0.45!(r)$);
\draw[blue]($(u)!0.5!(l)$)--(r);
\draw[blue](u)to[out=-20,in=-160,relative](r);\draw[blue](l)to[out=20,in=160,relative](r);
\end{tikzpicture}&
\begin{tikzpicture}[baseline=1mm]
\coordinate(l)at(-150:1);\coordinate(r)at(-30:1);\coordinate(u)at(90:1);\draw(u)--(l)--(r)--(u);
\draw[blue](u)--($(l)!0.5!(r)$);\draw[blue]($(u)!0.45!(l)$)--($(u)!0.45!(r)$);
\draw[blue]($(u)!0.55!(r)$)--($(l)!0.55!(r)$);
\draw[blue](u)to[out=20,in=160,relative](l);
\end{tikzpicture}&
\begin{tikzpicture}[baseline=1mm]
\coordinate(l)at(-150:1);\coordinate(r)at(-30:1);\coordinate(u)at(90:1);\draw(u)--(l)--(r)--(u);
\draw[blue]($(u)!0.45!(l)$)--($(u)!0.45!(r)$);\draw[blue]($(u)!0.55!(r)$)--($(l)!0.55!(r)$);
\draw[blue]($(u)!0.5!(r)$)--(l);
\draw[blue](u)to[out=20,in=160,relative](l);\draw[blue](l)to[out=20,in=160,relative](r);
\end{tikzpicture}&
\begin{tikzpicture}[baseline=1mm]
\coordinate(l)at(-150:1);\coordinate(r)at(-30:1);\coordinate(u)at(90:1);\draw(u)--(l)--(r)--(u);
\draw[blue](u)--($(l)!0.5!(r)$);\draw[blue]($(u)!0.45!(l)$)--($(u)!0.45!(r)$);
\draw[blue]($(u)!0.55!(l)$)--($(l)!0.45!(r)$);\draw[blue]($(u)!0.55!(r)$)--($(l)!0.55!(r)$);
\node at(0,1.05){};
\end{tikzpicture}
\end{tabular}\vspace{3mm}\]
\caption{The maximal underlying set $S$ of a multi-set that satisfies $(\phi\ast,a)_k$ on a triangle piece, where $a$ is the angle at the top of each triangle}
\label{table:max phiast}
\end{table}}

\begin{lem}\label{lem:except}
Assume that $S$ satisfies $(\phi\ast,a)$. Then all sets $\{s,t\}$ such that $s,t\in S\setminus\{h_a^{m_S(h_a)}\}$ and $\Int(s,t)>0$ are either $\{e_a,h_{a^l}\}$ or $\{e_a,h_{a^r}\}$.
\end{lem}

\begin{proof}
It is easy to check that the assertion holds in all cases (see Table \ref{table:max phiast}).
\end{proof}

\begin{prop}\label{prop:recover tri}
If $S$ satisfies $(\phi\ast,a)_k$ for $k\in\{1,\ldots,7\}$, then $\phi_a(S)$ satisfies $(\psi\ast,a)_k$ and $\psi_a\phi_a(S)=S$. In addition, if $S$ does not satisfy $(\phi,b)$ for $b\in\{a^l,a^r\}$, then $\phi_a(S)$ does not satisfy $(\phi,b)$ either.
\end{prop}

\begin{proof}
It follows from Tables \ref{table:smallest} and \ref{table:max phiast} that there is a multi-set satisfying $(\phi\ast,a)_k$ and containing $\phi_a(S)$ except for $k=7$. In which case, $\phi_a(S)$ only consists of segments $h_a$, $h_{a^l}$, $h_{a^r}$, and $y$. These mean that $\phi_a(S)$ satisfies $(\ast,a)$, and the set of all segments in $\phi_a(S)$ is contained in the corresponding maximal set in Table \ref{table:max psi}. Thus $\phi_a(S)$ satisfies $(\psi\ast,a)_k$. Tables \ref{table:def phi} and \ref{table:def psi} clearly induce $\psi_a\phi_a(S)=S$, and the last assertion follows from Lemmas \ref{lem:Int tri} and \ref{lem:except}.
\end{proof}

Note that we can not exchange $\phi$ and $\psi$ in Proposition \ref{prop:recover tri}. In fact, $\{e_a\}$ satisfies $(\psi\ast,a)_1$, but $\psi_a(\{e_a\})=\{e_a,e_{a^l},e_{a^r},h_a\}$ does not satisfy $(\ast,a)$.

The condition $(\ast,a)$ also gives the commutativity of $\phi_a$.

\begin{prop}\label{prop:commute tri}
Assume that $S$ satisfies $(\phi,a)$ and $(\phi,a^l)$. Then it satisfies $(\ast,a)$ if and only if it satisfies $(\ast,a^l)$. In which case, the following hold:
\begin{itemize}
\item[(1)] $S$ satisfy $(\phi\ast,a)_3$ and $(\phi\ast,a^l)_5$, in particular, it satisfies neither $(\phi,a^r)$ nor $(\ast,a^r)$.
\item[(2)] $\phi_a(S)$ satisfy $(\phi\ast,a^l)_5$ and $\phi_{a^l}(S)$ satisfy $(\phi\ast,a)_3$.
\item[(3)] $\phi_{a^l}\phi_a(S)=\phi_a\phi_{a^l}(S)$.
\end{itemize}
\end{prop}

\begin{proof}
The conditions $(\phi,a)$ and $(\phi,a^l)$ mean that $m_S(h_a)$, $m_S(h_{a^l})$, $m_S(e_{a^l})+m_S(e_{a^r})+m_S(v_a)$, and $m_S(e_{a^r})+m_S(e_a)+m_S(v_{a^l})$ are positive. If $S$ satisfies either $(\ast,a)$ or $(\ast,a^l)$, we must have $m_S(e_{a^r})=m_S(v_a)=m_S(v_{a^l})=0$. Thus $S$ satisfies $(\phi\ast,a)_3$ and $(\phi\ast,a^l)_5$, that is, the first assertion and (1) hold. Moreover, (2) follows from the definition (see Tables \ref{table:smallest} and \ref{table:def phi}). Finally, since both $\phi_{a^l}\phi_a(S)$ and $\phi_a\phi_{a^l}(S)$ are obtained from $S$ by removing $\{e_a,e_{a^l},h_a,h_{a^l}\}$ and adding $\{v_{a^r}^2\}$, (3) holds.
\end{proof}

\begin{propdef}
Let $k_b\in\bZ_{\ge 0}$ for all angles $b$ of $\triangle$. We set a formal product
\[
\phi=\prod_{b\in\{\text{angles of $\triangle$}\}}\phi_b^{k_b}.
\]
If $\phi_b^k(S)$ satisfies $(\phi\ast,b)$ for all angles $b$ of $\triangle$ and all $0\le k<k_b$, then there is at least one angle $c$ of $\triangle$ with $k_c=0$, and $\phi(S)$ is well-defined. In which case, $\phi$ or $\phi(S)$ is called a \emph{modification of $S$ at angles}. In addition, it is called \emph{maximal} if $\phi_b^{k_b}(S)$ does not satisfy $(\phi\ast,b)$ for all angles $b$ of $\triangle$.
\end{propdef}

\begin{proof}
The assertion follows from Proposition \ref{prop:commute tri}.
\end{proof}

We refer to modifications at angles as \emph{$a$-modifications}.

\begin{example}\label{exam:2phiast}
Assume that $S$ satisfies $(\phi\ast,a)$ and $(\phi\ast,a^l)$. By Proposition \ref{prop:commute tri},
\[
S=\{e_a^{m_S(e_a)},e_{a^l}^{m_S(e_{a^l})},h_a^{m_S(h_a)},h_{a^l}^{m_S(h_{a^l})},v_{a^r}^{m_S(v_{a^r})}\}.
\]
Note that its maximal underlying set appears in Table \ref{table:max phiast} as $k=3$. For $0\le k_a\le\min\{m_S(e_{a^l}),m_S(h_a)\}$ and $0\le k_{a^l}\le\min\{m_S(e_a),m_S(h_{a^l})\}$,
\[
\phi=\phi_a^{k_a}\phi_{a^l}^{k_{a^l}}=\phi_{a^l}^{k_{a^l}}\phi_a^{k_a}
\]
is an $a$-modification of $S$. If $k_a=\min\{m_S(e_{a^l}),m_S(h_a)\}$ and $k_{a^l}=\min\{m_S(e_a),m_S(h_{a^l})\}$, then $\phi$ is maximal and
\[
\phi(S)=\{e_a^{m_S(e_a)-k_{a^l}},e_{a^l}^{m_S(e_{a^l})-k_a},h_a^{m_S(h_a)-k_a},h_{a^l}^{m_S(h_{a^l})-k_{a^l}},v_{a^r}^{m_S(v_{a^r})+k_a+k_{a^l}}\}
\]
consists of pairwise compatible edges and segments. Moreover, it contains no edges if and only if $m_S(e_{a^l})\le m_S(h_a)$ and $m_S(e_a)\le m_S(h_{a^l})$.
\end{example}

\begin{thm}\label{thm:compatibility tri}
Assume that $S$ satisfies $(\phi\ast,a)$. Then its maximal $a$-modification consists of pairwise compatible edges and segments. In particular, it contains no edges if and only if 
\[
\max\{m_S(e_{b^l}),m_S(e_{b^r})\}\le m_S(h_b)
\]
for all angles $b$ where $S$ satisfies $(\phi,b)$.
\end{thm}

\begin{proof}
Assume that $S$ satisfies neither $(\phi,a^l)$ nor $(\phi,a^r)$. Then for its maximal $a$-modification $\phi(S)$, Lemmas \ref{lem:Int tri} and \ref{lem:except} induce that any two elements of $\phi(S)\setminus\{h_a^{m_{\phi(S)}(h_a)}\}$ are compatible. On the other hand, we know that $m_{\phi(S)}(h_a)=0$ or $\Int(h_a,\phi(S))=0$ since $\phi(S)$ does not satisfy $(\phi,a)$ by Proposition \ref{prop:recover tri}. Thus $\phi(S)$ consists of pairwise compatible edges and segments. The second assertion follows from $m_{\phi_a(S)}(h_a)=m_S(h_a)-1$ and Lemma \ref{lem:Int tri}.

If $S$ satisfies either $(\phi,a^l)$ or $(\phi,a^r)$, then the assertions follow from Proposition \ref{prop:commute tri} and Example \ref{exam:2phiast}. Since $S$ does not satisfy both $(\phi,a^l)$ and $(\phi,a^r)$ simultaneously by Proposition \ref{prop:commute tri}, the proof finishes.
\end{proof}

\subsection{Modifications at angles on a monogon piece}\label{subsec:modif mono}

Let $\triangle$ be a monogon piece with an angle $a$, and $S$ be a multi-set of non-loop edges and segments in $\triangle$. We make similar observations to those in the previous subsection. First, we consider the following condition of $S$:
\begin{equation*}
\tag{$\phi,a$} \text{$m_S(h_a)>0$ and $\{m_S(f_a),m_S(f_a^{\bowtie})\}=\{0,m\}$ for some $m\in\bZ_{>0}$}.
\end{equation*}
It is trivial that $S$ satisfies exactly one of the following conditions when it satisfies $(\phi,a)$:
\begin{enumerate}
\item[$(\phi 1)$] $m_S(f_a)>m_S(f_a^{\bowtie})=0$.
\item[$(\phi 2)$] $m_S(f_a^{\bowtie})>m_S(f_a)=0$.
\end{enumerate}
We say that $S$ satisfies $(\phi,a)_k$ if it satisfies $(\phi,a)$ and $(\phi k)$ for $k\in\{1,2\}$ (see Table \ref{table:smallest mono}). When $S$ satisfies $(\phi,a)$, we define $\phi_a(S)$ as the multi-set of non-loop edges and segments in $\triangle$ whose multiplicities satisfy Table \ref{table:def phi mono}. In particular, $\phi_a$ sends the second (resp., third) column to the fifth (resp., sixth) column in Table \ref{table:smallest mono}.

%
%
\renewcommand{\arraystretch}{1.3}
{\begin{table}[ht]
\begin{tabular}{c||c|c}
$X$&$(\phi,a)_1$&$(\phi,a)_2$\\\hline\hline
$S$&
\begin{tikzpicture}[baseline=1mm]
\coordinate(u)at(90:1);\coordinate(p)at(0,0.5);\coordinate(d)at(-90:0.5);
\draw(d)to[out=150,in=180]coordinate[pos=0.3](l)(u);\draw(d)to[out=30,in=0]coordinate[pos=0.3](r)(u);
\draw[blue](l)--(r);\draw[blue](d)--(p);
\fill(p)circle(0.07);\node at(0,1.1){};
\end{tikzpicture}&
\begin{tikzpicture}[baseline=1mm]
\coordinate(u)at(90:1);\coordinate(p)at(0,0.5);\coordinate(d)at(-90:0.5);
\draw(d)to[out=150,in=180]coordinate[pos=0.3](l)(u);\draw(d)to[out=30,in=0]coordinate[pos=0.3](r)(u);
\draw[blue](l)--(r);\draw[blue](d)--node[pos=0.8]{\footnotesize $\bowtie$}(p);
\fill(p)circle(0.07);\node at(0,1.1){};
\end{tikzpicture}
\end{tabular}
\hspace{7mm}
\begin{tabular}{c||c|c}
$X$&$(\psi,a)_1$&$(\psi,a)_2$\\\hline\hline
$S$&
\begin{tikzpicture}[baseline=1mm]
\coordinate(u)at(90:1);\coordinate(p)at(0,0.5);\coordinate(d)at(-90:0.5);
\draw(d)to[out=150,in=180](u);\draw(d)to[out=30,in=0](u);
\draw[blue](p)--(-0.4,0.5);\draw[blue](p)--(0.4,0.5);
\fill(p)circle(0.07);\node at(0,1.1){};
\end{tikzpicture}&
\begin{tikzpicture}[baseline=1mm]
\coordinate(u)at(90:1);\coordinate(p)at(0,0.5);\coordinate(d)at(-90:0.5);
\draw(d)to[out=150,in=180](u);\draw(d)to[out=30,in=0](u);
\draw[blue](p)--node[pos=0.4]{\rotatebox{90}{\footnotesize $\bowtie$}}(-0.4,0.5);
\draw[blue](p)--node[pos=0.4]{\rotatebox{90}{\footnotesize $\bowtie$}}(0.4,0.5);
\fill(p)circle(0.07);\node at(0,1.1){};
\end{tikzpicture}
\end{tabular}\vspace{3mm}
\caption{The smallest multi-set $S$ that satisfies the condition $X$ on a monogon piece with an angle $a$}
\label{table:smallest mono}
\end{table}}

%
%
\renewcommand{\arraystretch}{1.3}
{\begin{table}[ht]
\begin{tabular}{c||c|c|c|c|c}
$k\backslash s$&$f_a$&$f_a^{\bowtie}$&$h_a$&$i_a$&$i_a^{\bowtie}$\\\hline\hline
$1$&$-1$&$0$&$-1$&$2$&$0$\\\hline
$2$&$0$&$-1$&$-1$&$0$&$2$
\end{tabular}\vspace{3mm}
\caption{The difference $m_{\phi_a(S)}(s)-m_S(s)$ when $S$ satisfies $(\phi,a)_k$ on a monogon piece with an angle $a$}
\label{table:def phi mono}
\end{table}}

\begin{lem}\label{lem:Int mono}
If $S$ satisfies $(\phi,a)$, then $\Int(e_a,\phi_a(S))=\Int(e_a,S)$, and
\begin{align*}
\Int(f_a,\phi_a(S))-\Int(f_a,S)&=
\begin{cases}
-1&\text{if $S$ satisfies $(\phi,a)_1$};\\
1&\text{if $S$ satisfies $(\phi,a)_2$},
\end{cases}
\\
\Int(f_a^{\bowtie},\phi_a(S))-\Int(f_a^{\bowtie},S)&=
\begin{cases}
1&\text{if $S$ satisfies $(\phi,a)_1$};\\
-1&\text{if $S$ satisfies $(\phi,a)_2$}.
\end{cases}
\end{align*}
\end{lem}

\begin{proof}
The assertions follow from the definition of $\phi_a$ (see Table \ref{table:smallest mono}).
\end{proof}

Next, we consider the following condition of $S$:
\begin{equation*}
\tag{$\psi,a$} \text{$\{m_S(i_a),m_S(i_a^{\bowtie})\}=\{0,m\}$ for some $m\in\bZ_{\ge 2}$}.
\end{equation*}
It is trivial that $S$ satisfies exactly one of the following conditions when it satisfies $(\phi,a)$:
\begin{enumerate}
\item[$(\psi 1)$] $m_S(i_a)\ge 2$ and $m_S(i_a^{\bowtie})=0$.
\item[$(\psi 2)$] $m_S(i_a)=0$ and $m_S(i_a^{\bowtie})\ge2$.
\end{enumerate}
We say that $S$ satisfies $(\psi,a)_k$ if it satisfies $(\psi,a)$ and $(\psi k)$ for $k\in\{1,2\}$ (see Table \ref{table:smallest mono}). When $S$ satisfies $(\psi,a)$, we define $\psi_a(S)$ as the multi-set of non-loop edges and segments in $\triangle$ whose multiplicities satisfy Table \ref{table:def psi mono}. In particular, $\psi_a$ sends the fifth (resp., sixth) column to the second (resp., third) column in Table \ref{table:smallest mono}.

%
%
\renewcommand{\arraystretch}{1.3}
{\begin{table}[ht]
\begin{tabular}{c||c|c|c|c|c}
$k\backslash s$&$f_a$&$f_a^{\bowtie}$&$h_a$&$i_a$&$i_a^{\bowtie}$\\\hline\hline
$1$&$1$&$0$&$1$&$-2$&$0$\\\hline
$2$&$0$&$1$&$1$&$0$&$-2$
\end{tabular}\vspace{3mm}
\caption{The difference $m_{\psi_a(S)}(s)-m_S(s)$ when $S$ satisfies $(\psi,a)_k$ on a monogon piece with an angle $a$}
\label{table:def psi mono}
\end{table}}

As a sufficient condition of $S$ such that $\phi_a(S)$ satisfies $(\psi,a)$, we give the following condition:
\begin{equation*}
\tag*{$(\ast,a)$} \text{If $\Int(s,t)>0$ for $s, t\in S$, then either $s$ or $t$ is $h_a$}.
\end{equation*}

\begin{prop}\label{prop:recover mono}
Let $k\in\{1,2\}$.
\begin{itemize}
\item[(1)] If $S$ satisfies $(\phi\ast,a)_k$, then $\phi_a(S)$ satisfies $(\psi\ast,a)_k$, and $\psi_a\phi_a(S)=S$.
\item[(2)] If $S$ satisfies $(\psi\ast,a)_k$, then $\psi_a(S)$ satisfies $(\phi\ast,a)_k$, and $\phi_a\psi_a(S)=S$.
\end{itemize}
\end{prop}

\begin{proof}
The assertions immediately follow from the definitions.
\end{proof}

\begin{defn}
Let $k_a\in\bZ_{\ge 0}$. If $\phi_a^k(S)$ satisfies $(\phi\ast,a)$ for all $0\le k<k_a$, then $\phi_a^{k_a}$ or $\phi_a^{k_a}(S)$ is called an \emph{$a$-modification of $S$}. In addition, it is called \emph{maximal} if $\phi_a^{k_a}(S)$ does not satisfy $(\phi\ast,a)$.
\end{defn}

\begin{thm}\label{thm:compatibility mono}
Assume that $S$ satisfies $(\phi\ast,a)$. Then its maximal $a$-modification consists of pairwise compatible non-loop edges and segments. In particular, it contains no edges if and only if
\[
\max\{m_S(f_a),m_S(f_a^{\bowtie})\}\le m_S(h_a).
\]
\end{thm}

\begin{proof}
If $S$ satisfies $(\phi\ast,a)_1$, then $S=\{f_a^{m_S(f_a)},h_a^{m_S(h_a)},i_a^{m_S(i_a)}\}$. Since
\[
\phi_a^{k_a}(S)=\{f_a^{m_S(f_a)-k_a},h_a^{m_S(h_a)-k_a},i_a^{m_S(i_a)+2k_a}\},
\]
it is maximal if $k_a=\min\{m_S(f_a),m_S(h_a)\}$. Then the assertions hold. Similarly, we can prove them in the case that $S$ satisfies $(\phi\ast,a)_2$.
\end{proof}

\subsection{Modifications at angles}\label{subsec:modif}

Throughout the rest of this section, let $T$ be a tagged triangulation of $\cS$ satisfying \eqref{diamond}, and $S$ be a multi-set of non-loop edges and segments in puzzle pieces of $\Omega(T)$.

\begin{defn}
An \emph{$a$-modification} (resp., \emph{maximal $a$-modification}) \emph{of $S$} is a multi-set obtained from $S$ by replacing $S\cap\triangle$ with an $a$-modification (resp., maximal $a$-modification) of $S\cap\triangle$ for each puzzle piece $\triangle$ of $\Omega(T)$.
\end{defn}

Since a maximal $a$-modification of $S$ is uniquely determined, we denote it by $\Phi(S)$.

\begin{defn}
We say that $S$ is \emph{$a$-modifiable} if $S\cap\triangle$ consists of pairwise compatible non-loop edges and segments for each puzzle piece $\triangle$ without angles $a$ where $S\cap\triangle$ satisfies $(\phi\ast,a)$.
\end{defn}

\begin{thm}\label{thm:comp S}
If $S$ is $a$-modifiable, then $\Phi(S)$ consists of pairwise compatible non-loop edges and segments. In particular, it contains no edges if and only if the following hold: For each puzzle piece $\triangle$ of $\Omega(T)$ and each angle $a$ of $\triangle$ such that $S\cap\triangle$ satisfies $(\phi\ast,a)$,
\begin{itemize}
\item $\max\{m_S(e_{a^l}),m_S(e_{a^r})\}\le m_S(h_a)$ if $\triangle$ is a triangle piece;
\item $\max\{m_S(f_a),m_S(f_a^{\bowtie})\}\le m_S(h_a)$ if $\triangle$ is a monogon piece.
\end{itemize}
\end{thm}

\begin{proof}
The assertions follow from Theorems \ref{thm:compatibility tri} and \ref{thm:compatibility mono}.
\end{proof}

Theorem \ref{thm:comp S} can be applied to $S_U$, and $\Phi(S_U)$ will be a desired multi-set in Theorem \ref{thm:segment} (see also Example \ref{ex:glue}). We will study all about $S_U$ and $\Phi(S_U)$ in Subsection \ref{subsec:proof}.

In general, $S$ is not recoverable from $\Phi(S)$ even if it is $a$-modifiable. In the rest of this section, we will give a sufficient condition of $S$ such that it is recoverable from $\Phi(S)$, and show that $S_U$ satisfies it. For that, we introduce a notion of modifications around punctures ($p$-modifications for short). Under a certain condition, the maximal $a$-modifications coincide with the maximal $p$-modifications (Proposition \ref{prop:modif=max p}). As a stronger result than what we need, we give a sufficient condition of $S$ such that it is recoverable from a $p$-modification of it (Definition \ref{def:recoverable} and Theorem \ref{thm:unique S}).

\subsection{Modifications around punctures}\label{subsec:p-modif}

For an angle $a$ of $\Omega(T)$, there is a unique puzzle piece $\triangle_a$ with $a$. Thus we naturally extend the notations in Subsections \ref{subsec:modif tri} and \ref{subsec:modif mono} to $S$: For example, we say that $S$ satisfies $(\phi,a)$ if its sub-multi-set $S\cap\triangle_a$ satisfies $(\phi,a)$. In which case, $\phi_a(S)$ is obtained from $S$ by replacing $S\cap\triangle_a$ with $\phi_a(S\cap\triangle_a)$. Moreover, we prepare the following notations for a puncture $p$:
\begin{itemize}
\item $A_p$ is the set of all angles of $\Omega(T)$ at $p$.
\item $A_p^{\phi}(S):=\{a\in A_p\mid\text{$S$ satisfies $(\phi,a)$}\}$.
\item $m_S(c_p):=\min\{m_S(h_a)\mid a\in A_p\}$.
\item $A_p^{\min}(S):=\{a\in A_p\mid m_S(h_a)=m_S(c_p)\}$.
\end{itemize}
Note that $S$ forms a multi-set $\{c_p^{m_S(c_p)}\}$ around $p$, where $c_p$ is a simple closed curve enclosing exactly one puncture $p$. If $p$ is not incident to $\Omega(T)$, then $A_p=A_p^{\phi}(S)=A_p^{\min}(S)=\emptyset$.

First, we consider the following condition of $S$:
\begin{equation*}
\tag{$\phi,p$} A_p^{\phi}(S)\neq\emptyset\ \text{ and maps $\{\phi_a\mid a\in A_p^{\phi}(S)\}$ for $S$ commute},
\end{equation*}
where the second condition means that
\[
\phi_p(S):=\Biggl(\prod_{a\in A_p^{\phi}(S)}\phi_a\Biggr)(S)
\]
is well-defined. Note that $\phi_a\phi_b=\phi_b\phi_a$ always holds for $a,b\in A_p^{\phi}(S)$ with $\triangle_a\neq\triangle_b$. Moreover, if $S$ satisfies $(\ast,a)$ and $(\ast,b)$ for $a,b\in A_p^{\phi}(S)$ with $\triangle_a=\triangle_b$, then $\phi_a\phi_b(S)=\phi_b\phi_a(S)$ by Proposition \ref{prop:commute tri} (see Lemma \ref{lem:phi p}).

Next, we consider the following condition of $S$:
\begin{equation*}
\tag{$\psi,p$} \text{$S$ satisfies $(\psi,a)$ for all $a\in A_p^{\min}(S)\neq\emptyset$ and maps $\{\psi_a\mid a\in A_p^{\min}(S)\}$ for $S$ commute},
\end{equation*}
where the conditions mean that
\[
\psi_p(S):=\Biggl(\prod_{a\in A_p^{\min}(S)}\psi_a\Biggr)(S)
\]
is well-defined. As in the previous subsections, we give a sufficient condition of $S$ such that $\phi_p(S)$ satisfies $(\psi,p)$.

\begin{defn}\label{def:ast p}
For a puncture $p$, we say that $S$ satisfies $(\ast,p)$ if it satisfies the following conditions:
{\setlength{\leftmargini}{15mm}
\begin{itemize}
\item[$(\ast 1,p)$] $A_p^{\phi}(S)\subseteq A_p^{\min}(S)$.
\item[$(\ast 2,p)$] $S$ satisfies $(\ast,a)$ for all $a\in A_p^{\phi}(S)$.
\end{itemize}}
\end{defn}

\begin{lem}\label{lem:phi p}
If $S$ satisfies $(\ast 2,p)$ for a puncture $p$, then $S$ satisfies $(\phi,p)$ if and only if $A_p^{\phi}(S)\neq\emptyset$. 
\end{lem}

\begin{proof}
The assertion follows from Proposition \ref{prop:commute tri}.
\end{proof}

For short, we say that $S$ satisfies $(\phi\ast,p)$ (resp., $(\psi\ast,p)$) if it satisfies $(\phi,p)$ (resp., $(\psi,p)$) and $(\ast,p)$.

\begin{prop}\label{prop:recover p}
If $S$ satisfies $(\phi\ast,p)$ for a puncture $p$, then $\phi_p(S)$ satisfies $(\psi\ast,p)$ and $\psi_p\phi_p(S)=S$.
\end{prop}

\begin{proof}
By the definition of $\phi_a$ for an angle $a$ of $\Omega(T)$,
\[
m_{\phi_p(S)}(h_a)
\begin{cases}
=m_S(h_a)-1&\text{if $a\in A_p^{\phi}(S)$};\\
\ge m_S(h_a)&\text{otherwise}.
\end{cases}
\]
Thus $(\ast 1,p)$ induces that $A_p^{\phi}(S)=A_p^{\min}(\phi_p(S))$. Since $A_p^{\phi}(\phi_p(S))$ must be a subset of $A_p^{\phi}(S)$ by Lemma \ref{lem:Int tri}, $\phi_p(S)$ satisfies $(\ast 1,p)$. By Propositions \ref{prop:recover tri}, \ref{prop:recover mono}(1), and $(\ast 2,p)$, $\phi_p(S)$ satisfies $(\psi\ast,a)$ for all $a\in A_p^{\phi}(S)=A_p^{\min}(\phi_p(S))$. Thus $\phi_p(S)$ satisfies the first condition in $(\psi,p)$ and $(\ast 2,p)$ since $A_p^{\phi}(\phi_p(S))\subseteq A_p^{\phi}(S)$. The second condition in $(\psi,p)$ follows from  Proposition \ref{prop:commute tri} and $A_p^{\phi}(S)=A_p^{\min}(\phi_p(S))$. Therefore, $\phi_p(S)$ satisfies $(\psi\ast,p)$. The last assertion $\psi_p\phi_p(S)=S$ also follows from Propositions \ref{prop:recover tri}, \ref{prop:recover mono}(1), and $A_p^{\phi}(S)=A_p^{\min}(\phi_p(S))$.
\end{proof}

The condition $(\ast,p)$ also gives the commutativity of $\phi_p$.

\begin{prop}\label{prop:commute}
If $S$ satisfies $(\phi\ast,p)$ and $(\phi\ast,q)$ for distinct punctures $p$ and $q$, then $\phi_q(S)$ satisfies $(\phi\ast,p)$ and $\phi_p\phi_q(S)=\phi_q\phi_p(S)$.
\end{prop}

\begin{proof}
First, we prove the following: For $a\in A_p$,
\begin{itemize}
\item[(1)] $a\in A_p^{\phi}(\phi_q(S))$ if and only if $a\in A_p^{\phi}(S)$, in which case, $\phi_q(S)$ also satisfies $(\ast,a)$;
\item[(2)] $m_{\phi_q(S)}(h_a)\ge m_S(h_a)$, where the equality holds if $a\in A_p^{\phi}(S)$.
\end{itemize}
If there is not an angle $b$ of $\triangle_a$ such that $b\in A_q^{\phi}(S)$, then (1) and (2) hold since $\phi_q(S)\cap\triangle_a=S\cap\triangle_a$. Assume that there is such an angle $b$. In particular, $\triangle_a=\triangle_b$ is a triangle piece. If $a\notin A_p^{\phi}(S)$, then $a\notin A_p^{\phi}(\phi_q(S))$ by Proposition \ref{prop:recover tri}, and $m_{\phi_q(S)}(h_a)\ge m_S(h_a)$ since $\phi_b$ does not remove $h_{b^l}$ and $h_{b^r}$. If $a\in A_p^{\phi}(S)$, then Proposition \ref{prop:commute tri} induces that $\phi_q(S)\cap\triangle_a=\phi_b(S)\cap\triangle_a$ satisfies $(\phi\ast,a)$ and $m_{\phi_q(S)}(h_a)=m_{\phi_b(S)}(h_a)= m_S(h_a)$. Therefore, (1) and (2) hold for any $a\in A_p$.

Since $A_p^{\phi}(S)\subseteq A_p^{\min}(\phi_q(S))$ by (2) and $(\ast 1,p)$ of $S$, $\phi_q(S)$ satisfies $(\ast 1,p)$ by (1). By (1) again, $\phi_q(S)$ satisfies $(\ast 2,p)$ and $A_p^{\phi}(\phi_q(S))=A_p^{\phi}(S)\neq\emptyset$. Thus $\phi_q(S)$ satisfies $(\phi\ast,p)$ by Lemma \ref{lem:phi p}.

In the same way as above, we obtain $A_q^{\phi}(\phi_p(S))=A_q^{\phi}(S)$. Moreover, by Proposition \ref{prop:commute tri}, $\phi_a\phi_b=\phi_b\phi_a$ for any $a\in A_p^{\phi}(S)$ and $b\in A_q^{\phi}(S)$. Therefore,
\begin{align*}
\phi_p\phi_q(S)&=\Biggl(\prod_{a\in A_p^{\phi}(\phi_q(S))}\phi_a\Biggr)\Biggl(\prod_{b\in A_q^{\phi}(S)}\phi_b\Biggr)(S)
=\Biggl(\prod_{a\in A_p^{\phi}(S)}\phi_a\Biggr)\Biggl(\prod_{b\in A_q^{\phi}(S)}\phi_b\Biggr)(S)\\
&=\Biggl(\prod_{b\in A_q^{\phi}(S)}\phi_b\Biggr)\Biggl(\prod_{a\in A_p^{\phi}(S)}\phi_a\Biggr)(S)
=\Biggl(\prod_{b\in A_q^{\phi}(\phi_p(S))}\phi_b\Biggr)\Biggl(\prod_{a\in A_p^{\phi}(S)}\phi_a\Biggr)(S)
=\phi_q\phi_p(S).\qedhere
\end{align*}
\end{proof}

\begin{propdef}\label{propdef:p-modif}
Let $k_p\in\bZ_{\ge 0}$ for all punctures $p$. We set a formal product
\[
\phi=\prod_{p\in\{\text{punctures}\}}\phi_p^{k_p}.
\]
If $\phi_p^k(S)$ satisfies $(\phi\ast,p)$ for all punctures $p$ and all $0\le k<k_p$, then $\phi(S)$ is well-defined. In which case, $\phi$ or $\phi(S)$ is called a \emph{modification of $S$ around punctures}. In addition, it is called \emph{maximal} if $\phi_p^{k_p}(S)$ does not satisfy $(\phi\ast,p)$ for all punctures $p$.
\end{propdef}

\begin{proof}
It follows from Proposition \ref{prop:commute} that $\phi(S)$ is well-defined.
\end{proof}

We refer to modifications around punctures as \emph{$p$-modifications}. For a $p$-modification $\phi$ in Proposition-Definition \ref{propdef:p-modif}, we denote by $P_{\phi}$ the set of all punctures $p$ with $k_p>0$, that is, $P_{\phi}\subseteq\{p\mid\text{$S$ satisfies $(\phi\ast,p)$}\}$ and the equality holds if $\phi$ is maximal.

\begin{prop}\label{prop:modif=max p}
If each angle $a$ where $S$ satisfies $(\phi\ast,a)$ is at some puncture $p$ where $S$ satisfies $(\ast,p)$, then the maximal $p$-modification of $S$ coincides with its maximal $a$-modification.
\end{prop}

\begin{proof}
By the assumption and Lemma \ref{lem:phi p}, for an angle $a$, $S$ satisfies $(\phi\ast,a)$ if and only if $a\in A_p^{\phi}(S)$ for a puncture $p$ where $S$ satisfies $(\phi\ast,p)$. Then the maximal $p$-modification $\phi$ of $S$ is given by
\[
\phi=\prod_{p\in P_{\phi}}\phi_p^{k_p}=\prod_{p\in\{q\mid\text{$S$ satisfies $(\phi\ast,q)$}\}}\prod_{a\in A_p^{\phi}(S)}\phi_a^{k_a}=\prod_{a\in\{\text{angles}\}}\phi_a^{k_a},
\]
where $k_a\in\bZ_{\ge 0}$ such that $\phi_a^k(S)$ satisfies $(\phi\ast,a)$ for all $0\le k<k_a$ and $\phi_a^{k_a}(S)$ does not satisfy $(\phi\ast,a)$. It is just the maximal $a$-modification of $S$.
\end{proof}

\begin{example}\label{ex:p-modif}
The multi-set $S_U$ of non-loop edges and segments in Example \ref{ex:glue} satisfies $(\phi\ast,p)$ for the bottom puncture $p$. Then its maximal $p$-modification coincides with $\Phi(S_U)$ as follows:
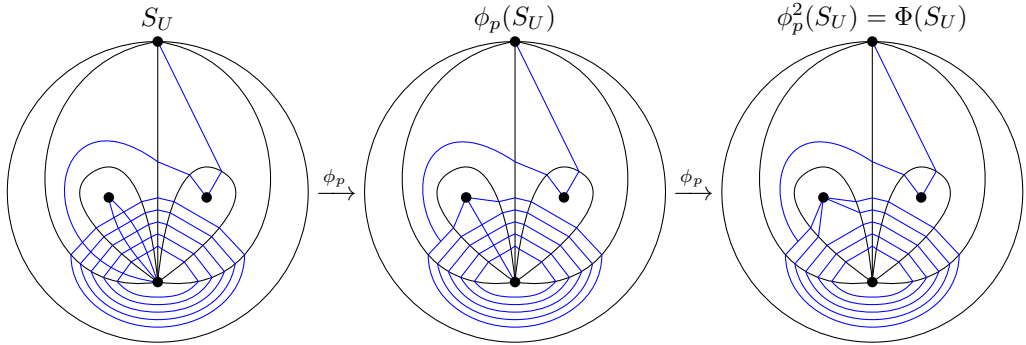
\[
\begin{tikzpicture}[baseline=-1mm]
\coordinate(d)at(0,-1.2);\coordinate(u)at(90:2);
\coordinate(l)at($(d)+(120:1.3)$);\coordinate(r)at($(d)+(60:1.3)$);
\draw(0,0)circle(2);
\draw(u)..controls(-2,1.7)and(-2,-1.5)..coordinate[pos=0.74](15)coordinate[pos=0.78](14)coordinate[pos=0.82](13)coordinate[pos=0.86](12)coordinate[pos=0.9](11)(d);
\draw(u)--coordinate[pos=0.5](26)coordinate[pos=0.65](25)coordinate[pos=0.7](24)coordinate[pos=0.75](23)coordinate[pos=0.8](22)coordinate[pos=0.85](21)(d);
\draw(u)..controls(2,1.7)and(2,-1.5)..coordinate[pos=0.74](35)coordinate[pos=0.78](34)coordinate[pos=0.82](33)coordinate[pos=0.86](32)coordinate[pos=0.9](31)(d);
\draw(d)to[out=25,in=90,relative]coordinate[pos=0.2](8l1)coordinate[pos=0.27](8l2)coordinate[pos=0.34](8l3)coordinate[pos=0.41](8l4)coordinate[pos=0.48](8l5)($(d)+(120:1.7)$);
\draw(d)to[out=-25,in=-90,relative]coordinate[pos=0.2](8r1)coordinate[pos=0.27](8r2)coordinate[pos=0.34](8r3)coordinate[pos=0.41](8r4)coordinate[pos=0.48](8r5)($(d)+(120:1.7)$);
\draw(d)to[out=25,in=90,relative]coordinate[pos=0.2](9l1)coordinate[pos=0.27](9l2)coordinate[pos=0.34](9l3)coordinate[pos=0.41](9l4)coordinate[pos=0.48](9l5)coordinate[pos=0.7](9l6)($(d)+(60:1.7)$);
\draw(d)to[out=-25,in=-90,relative]coordinate[pos=0.2](9r1)coordinate[pos=0.27](9r2)coordinate[pos=0.34](9r3)coordinate[pos=0.41](9r4)coordinate[pos=0.48](9r5)($(d)+(60:1.7)$);
\draw[blue](d)to[out=10,in=160,relative](l)(d)to[out=-10,in=-160,relative](l);
\draw[blue](11)--(8l1)--(8r1)--(21)--(9l1)--(9r1)--(31)to[out=-130,in=0](0,-1.4)to[out=180,in=-50](11);
\draw[blue](12)--(8l2)--(8r2)--(22)--(9l2)--(9r2)--(32)to[out=-120,in=0](0,-1.5)to[out=180,in=-60](12);
\draw[blue](13)--(8l3)--(8r3)--(23)--(9l3)--(9r3)--(33)to[out=-110,in=0](0,-1.6)to[out=180,in=-70](13);
\draw[blue](14)--(8l4)--(8r4)--(24)--(9l4)--(9r4)--(34)to[out=-100,in=0](0,-1.7)to[out=180,in=-80](14);
\draw[blue](15)--(8l5)--(8r5)--(25)--(9l5)--(9r5)--(35)to[out=-90,in=0](0,-1.8)to[out=180,in=-90](15);
\draw[blue](u)--($(d)+(60:1.7)$)--(r);
\draw[blue](r)--(9l6)--(26)..controls(-1.5,1.5)and(-1.8,-1)..(d);
\fill(l)circle(0.07);\fill(r)circle(0.07);\fill(d)circle(0.07);\fill(u)circle(0.07);
\node at(0,2.3){$S_U$};
\end{tikzpicture}
\xrightarrow{\phi_p}
\begin{tikzpicture}[baseline=-1mm]
\coordinate(d)at(0,-1.2);\coordinate(u)at(90:2);
\coordinate(l)at($(d)+(120:1.3)$);\coordinate(r)at($(d)+(60:1.3)$);
\draw(0,0)circle(2);
\draw(u)..controls(-2,1.7)and(-2,-1.5)..coordinate[pos=0.74](15)coordinate[pos=0.78](14)coordinate[pos=0.82](13)coordinate[pos=0.86](12)coordinate[pos=0.9](11)(d);
\draw(u)--coordinate[pos=0.5](26)coordinate[pos=0.65](25)coordinate[pos=0.7](24)coordinate[pos=0.75](23)coordinate[pos=0.8](22)coordinate[pos=0.85](21)(d);
\draw(u)..controls(2,1.7)and(2,-1.5)..coordinate[pos=0.74](35)coordinate[pos=0.78](34)coordinate[pos=0.82](33)coordinate[pos=0.86](32)coordinate[pos=0.9](31)(d);
\draw(d)to[out=25,in=90,relative]coordinate[pos=0.2](8l1)coordinate[pos=0.27](8l2)coordinate[pos=0.34](8l3)coordinate[pos=0.41](8l4)coordinate[pos=0.48](8l5)($(d)+(120:1.7)$);
\draw(d)to[out=-25,in=-90,relative]coordinate[pos=0.2](8r1)coordinate[pos=0.27](8r2)coordinate[pos=0.34](8r3)coordinate[pos=0.41](8r4)coordinate[pos=0.48](8r5)($(d)+(120:1.7)$);
\draw(d)to[out=25,in=90,relative]coordinate[pos=0.2](9l1)coordinate[pos=0.27](9l2)coordinate[pos=0.34](9l3)coordinate[pos=0.41](9l4)coordinate[pos=0.48](9l5)coordinate[pos=0.7](9l6)($(d)+(60:1.7)$);
\draw(d)to[out=-25,in=-90,relative]coordinate[pos=0.2](9r1)coordinate[pos=0.27](9r2)coordinate[pos=0.34](9r3)coordinate[pos=0.41](9r4)coordinate[pos=0.48](9r5)($(d)+(60:1.7)$);
\draw[blue](d)--(l);
\draw[blue](11)--(8l1)--(8r1)--(21)--(9l1)--(9r1)--(31)to[out=-130,in=0](0,-1.4)to[out=180,in=-50](11);
\draw[blue](12)--(8l2)--(8r2)--(22)--(9l2)--(9r2)--(32)to[out=-120,in=0](0,-1.5)to[out=180,in=-60](12);
\draw[blue](13)--(8l3)--(8r3)--(23)--(9l3)--(9r3)--(33)to[out=-110,in=0](0,-1.6)to[out=180,in=-70](13);
\draw[blue](14)--(8l4)--(8r4)--(24)--(9l4)--(9r4)--(34)to[out=-100,in=0](0,-1.7)to[out=180,in=-80](14);
\draw[blue](15)--(8l5)--(l)--(8r5)--(25)--(9l5)--(9r5)--(35)to[out=-90,in=0](0,-1.8)to[out=180,in=-90](15);
\draw[blue](u)--($(d)+(60:1.7)$)--(r);
\draw[blue](r)--(9l6)--(26)..controls(-1.2,1.2)and(-1.5,0)..($(15)!0.5!(8l5)$);
\fill(l)circle(0.07);\fill(r)circle(0.07);\fill(d)circle(0.07);\fill(u)circle(0.07);
\node at(0,2.3){$\phi_p(S_U)$};
\end{tikzpicture}
\xrightarrow{\phi_p}
\begin{tikzpicture}[baseline=-1mm]
\coordinate(d)at(0,-1.2);\coordinate(u)at(90:2);
\coordinate(l)at($(d)+(120:1.3)$);\coordinate(r)at($(d)+(60:1.3)$);
\draw(0,0)circle(2);
\draw(u)..controls(-2,1.7)and(-2,-1.5)..coordinate[pos=0.74](15)coordinate[pos=0.78](14)coordinate[pos=0.82](13)coordinate[pos=0.86](12)coordinate[pos=0.9](11)(d);
\draw(u)--coordinate[pos=0.5](26)coordinate[pos=0.65](25)coordinate[pos=0.7](24)coordinate[pos=0.75](23)coordinate[pos=0.8](22)coordinate[pos=0.85](21)(d);
\draw(u)..controls(2,1.7)and(2,-1.5)..coordinate[pos=0.74](35)coordinate[pos=0.78](34)coordinate[pos=0.82](33)coordinate[pos=0.86](32)coordinate[pos=0.9](31)(d);
\draw(d)to[out=25,in=90,relative]coordinate[pos=0.2](8l1)coordinate[pos=0.27](8l2)coordinate[pos=0.34](8l3)coordinate[pos=0.41](8l4)coordinate[pos=0.48](8l5)($(d)+(120:1.7)$);
\draw(d)to[out=-25,in=-90,relative]coordinate[pos=0.2](8r1)coordinate[pos=0.27](8r2)coordinate[pos=0.34](8r3)coordinate[pos=0.41](8r4)coordinate[pos=0.48](8r5)($(d)+(120:1.7)$);
\draw(d)to[out=25,in=90,relative]coordinate[pos=0.2](9l1)coordinate[pos=0.27](9l2)coordinate[pos=0.34](9l3)coordinate[pos=0.41](9l4)coordinate[pos=0.48](9l5)coordinate[pos=0.7](9l6)($(d)+(60:1.7)$);
\draw(d)to[out=-25,in=-90,relative]coordinate[pos=0.2](9r1)coordinate[pos=0.27](9r2)coordinate[pos=0.34](9r3)coordinate[pos=0.41](9r4)coordinate[pos=0.48](9r5)($(d)+(60:1.7)$);
\draw[blue](11)--(8l1)--(8r1)--(21)--(9l1)--(9r1)--(31)to[out=-130,in=0](0,-1.4)to[out=180,in=-50](11);
\draw[blue](12)--(8l2)--(8r2)--(22)--(9l2)--(9r2)--(32)to[out=-120,in=0](0,-1.5)to[out=180,in=-60](12);
\draw[blue](13)--(8l3)--(8r3)--(23)--(9l3)--(9r3)--(33)to[out=-110,in=0](0,-1.6)to[out=180,in=-70](13);
\draw[blue](14)--(8l4)--(l)--(8r4)--(24)--(9l4)--(9r4)--(34)to[out=-100,in=0](0,-1.7)to[out=180,in=-80](14);
\draw[blue](15)--(8l5)--(l)--(8r5)--(25)--(9l5)--(9r5)--(35)to[out=-90,in=0](0,-1.8)to[out=180,in=-90](15);
\draw[blue](u)--($(d)+(60:1.7)$)--(r);
\draw[blue](r)--(9l6)--(26)..controls(-1.2,1.2)and(-1.5,0)..($(15)!0.5!(8l5)$);
\fill(l)circle(0.07);\fill(r)circle(0.07);\fill(d)circle(0.07);\fill(u)circle(0.07);
\node at(0,2.3){$\phi_p^2(S_U)=\Phi(S_U)$};
\end{tikzpicture}\ .
\]
\end{example}

\subsection{Glueability and enclosed punctures}\label{subsec:glue}

To give a sufficient condition of $S$ such that it is recoverable from a $p$-modification of it (Definition \ref{def:recoverable} and Theorem \ref{thm:unique S}), we prepare Definitions \ref{def:glueable} and \ref{def:enclosed}. For an angle $a$ of $\Omega(T)$ at a marked point $p$, we denote by $\overline{a}$ the next angle clockwise from $a$ at $p$ as follows:
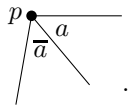
\[
\begin{tikzpicture}[baseline=-10mm]
\coordinate(p)at(0,0)node[left]{$p$};
\draw(0:1.2)--(p)--(-100:1.2) (p)--(-50:1.2);
\node at(-25:0.45){$a$};\node at(-75:0.45){$\overline{a}$};
\fill(p)circle(0.07);
\end{tikzpicture}.
\]
We set $a^l=a^r=a$ for an angle $a$ of a monogon piece.

\begin{defn}\label{def:glueable}
We say that $S$ is \emph{glueable} if it satisfies the following conditions:
\begin{itemize}
\item[$(G1)$] $\Int(e_{a^r},S)=\Int(e_{\overline{a}^l},S)$ for all angles $a$ of $\Omega(T)$.
\item[$(G2)$] $m_S(e_{a^r})=m_S(e_{\overline{a}^l})$ for all angles $a$ of $\Omega(T)$.
\end{itemize}
\end{defn}

Assume that $S$ is glueable. By $(G1)$, when we glue puzzle pieces of $\Omega(T)$, segments in $S$ can be glued simultaneously (see Example \ref{ex:glue}). We say that the ``curves'' obtained at this time are \emph{tagged branched curves}. More precisely, a \emph{branched curve} in $\cS$ is a connected graph in $\cS$ whose vertices are interior points or marked points in $\cS$ such that the former vertices have degree three. A \emph{tagged branched curve} is a branched curve whose ends at marked points are tagged in the same way as tagged arcs.

Therefore, we obtain a multi-set of tagged branched curves in $\cS$ from $S$ by gluing segments simultaneously. We denote by $\widehat{S}$ the disjoint union of this multi-set and the multi-set $\{\gamma^{\frac{1}{2}m_S(\gamma)}\mid\gamma\in\Omega(T)\}$, where $m_S(\gamma)=m_S(e)+m_S(f)$ for the edges $e$ and $f$ corresponding to $\gamma$, and it is even by $(G2)$.

We can extend the notion of intersection numbers to tagged branched curves. For example,
\begin{equation}\label{eq:Int c}
\Int(c_p,\widehat{S})=\sum_{a\in A_p}\Bigl(\Int(h_a,S_s)+\frac{1}{2}\Int(h_a,S_e)\Bigr),
\end{equation}
where $S=S_s\sqcup S_e$ and $S_e$ consists of all edges in $S$ corresponding to plain curves in $\Omega(T)$. Moreover,
\[
\Int_{\Omega(T)}(\widehat{S})=\left(\Int(\gamma,\widehat{S})\right)_{\gamma\in\Omega(T)}=\left(\frac{1}{2}\Int(\gamma,S)\right)_{\gamma\in\Omega(T)}=\frac{1}{2}\Int_{\Omega(T)}(S).
\]

\begin{lem}\label{lem:phi Int}
Assume that $S$ is glueable and satisfies $(\phi\ast,p)$ for a puncture $p$. Then $\phi_p(S)$ is also glueable and $\Int(c_p,\widehat{\phi_p(S)})\le\Int(c_p,\widehat{S})-1$, in particular, the equality holds if and only if the set $S_p$ of all elements of $S$ incident to $p$ is one of the sets $\{v_a\}$, $\{e_{a^r}^k,e_{\overline{a}^l}^k\}$, $\{f_a^k\}$, and $\{(f_a^{\bowtie})^k\}$ for some $a\in A_p$ and $k\in\bZ_{>0}$.
\end{lem}

\begin{proof}
By Lemmas \ref{lem:Int tri} and \ref{lem:Int mono}, $\phi_p(S)$ still satisfies (G1) and (G2). Since for an angle $a$ of $\Omega(T)$
\[
\Int(h_a,\phi_p(S))-\Int(h_a,S)
\begin{cases}
\le -1&\text{if $a\in A_p^{\phi}$};\\
=0&\text{otherwise},
\end{cases}
\]
the desired inequality holds. Moreover, it also follows from $(G2)$ that $\Int(c_p,\widehat{\phi_p(S)})=\Int(c_p,\widehat{S})-1$ if and only if $S_p=\{e_{a^r}^k,e_{\overline{a}^l}^k\}$ or $\Int(h_a,\phi_p(S))=\Int(h_a,S)-1$ for exactly one $a\in A_p^{\phi}$. The latter always holds if $\triangle_a$ is a monogon piece. If $\triangle_a$ is a triangle piece, it is equivalent that $S$ satisfies $(\phi,a)_7$. Thus we can obtain the desired list of $S_p$.
\end{proof}

\begin{defn}\label{def:enclosed}
Assume that $S$ is glueable and satisfies $(\phi\ast,p)$ for a puncture $p$. We say that $p$ is \emph{enclosed} in $\phi_p(S)$ if $\Int(c_p,\widehat{\phi_p(S)})=\Int(c_p,\widehat{S})-1$.
\end{defn}

By Lemma \ref{lem:phi Int}, if a puncture $p$ is enclosed in $S$, then $\widehat{S}$ contains one of the following tagged branched curves $\gamma$ around $p$:
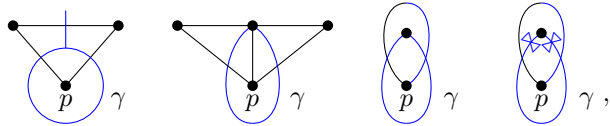
\[
\begin{tikzpicture}[baseline=-2mm]
\coordinate(0)at(0,0);\coordinate(l)at(-0.7,0.8);\coordinate(r)at(0.7,0.8);\draw(0)--(l)--(r)--(0);
\draw[blue](0)circle(0.5);\draw[blue](0,0.5)--(0,1);
\fill(0)circle(0.07)node[below]{$p$};\fill(l)circle(0.07);\fill(r)circle(0.07);\node at(0.7,-0.2){$\gamma$};
\end{tikzpicture}
\hspace{5mm}
\begin{tikzpicture}[baseline=-2mm]
\coordinate(0)at(0,0);\coordinate(l)at(-1,0.8);\coordinate(r)at(1,0.8);\coordinate(c)at(0,0.8);\draw(0)--(l)--(r)--(0)--(c);
\draw[blue](c)to[out=-150,in=180](0,-0.5);\draw[blue](c)to[out=-30,in=0](0,-0.5);
\fill(0)circle(0.07)node[below]{$p$};\fill(l)circle(0.07);\fill(r)circle(0.07);\fill(c)circle(0.07);\node at(0.6,-0.2){$\gamma$};
\end{tikzpicture}
\hspace{5mm}
\begin{tikzpicture}[baseline=-2mm]
\coordinate(0)at(0,0);\coordinate(c)at(0,0.7);
\draw(0)to[out=150,in=180](0,1.1);\draw[blue](0)to[out=30,in=0](0,1.1);
\draw[blue](c)to[out=-150,in=180](0,-0.5);\draw[blue](c)to[out=-30,in=0](0,-0.5);
\fill(0)circle(0.07)node[below]{$p$};\fill(c)circle(0.07);\node at(0.6,-0.2){$\gamma$};
\end{tikzpicture}
\hspace{5mm}
\begin{tikzpicture}[baseline=-2mm]
\coordinate(0)at(0,0);\coordinate(c)at(0,0.7);
\draw(0)to[out=150,in=180](0,1.1);\draw[blue](0)to[out=30,in=0](0,1.1);
\draw[blue](c)to[out=-150,in=180]node[pos=0.13]{\rotatebox{-45}{\footnotesize $\bowtie$}}(0,-0.5);
\draw[blue](c)to[out=-30,in=0]node[pos=0.13]{\rotatebox{45}{\footnotesize $\bowtie$}}(0,-0.5);
\fill(0)circle(0.07)node[below]{$p$};\fill(c)circle(0.07);\node at(0.6,-0.2){$\gamma$};
\end{tikzpicture},
\]
where the three punctured loops on the right are compatible with $\widehat{S}$.

\begin{lem}\label{lem:enclosed}
Assume that $S$ is glueable and satisfies $(\phi\ast,p)$ for a puncture $p$. If $p$ is enclosed in $\phi_p(S)$ and a puncture $q$ is not enclosed in $S$, then $q$ is not enclosed in $\phi_p(S)$.
\end{lem}

\begin{proof}
Assume that $q$ is enclosed in $\phi_p(S)$. It must be enclosed by a tagged branched curve including at least one segment in $\phi_p(S)\setminus S$. It follows from the above local configurations that there is a punctured loop in $\widehat{\phi_p(S)}$ enclosing $p$ and $q$. This means that $\cS$ is a sphere with exactly three punctures. It contradicts our assumption.
\end{proof}

\begin{example}\label{ex:enclosed}
The multi-sets $S_U$, $\phi_p(S_U)$, and $\phi_p^2(S_U)=\Phi(S_U)$ in Example \ref{ex:p-modif} are glueable. The figures in Example \ref{ex:p-modif} also represent the multi-sets $\widehat{S_U}$, $\widehat{\phi_p(S_U)}$, and $\widehat{\phi_p^2(S_U)}$. Moreover, the bottom puncture $p$ is enclosed in $\phi_p^2(S_U)$, but not enclosed in $S_U$ and $\phi_p(S_U)$.
\end{example}

\subsection{Recoverability}\label{subsec:recover}

Assume that $S$ is glueable thoroughout this subsection. To prove Theorem \ref{thm:unique S}, for each puncture $p$, we consider two integers $m_S(c_p)$ and
\[
d_p^S:=m_S(c_p)-\Int'(c_p,\widehat{S}),
\]
where for a set $X$, we define the following notation:
\[
\Int'(x,X):=
\begin{cases}
\Int(x,X)&\text{if $x\in X$};\\
0&\text{otherwise}.
\end{cases}
\]
Here, we remark that $m_S(c_p)=m_{\widehat{S}}(c_p)$. The following lemma gives their changes by a single $p$-modification.

\begin{lem}\label{lem:phi difference}
Let $p$ be a puncture in $\cS$. If $S$ is glueable and satisfies $(\phi\ast,p)$ and $(\ast 1,q)$ for all punctures $q$, then there is a (possibly empty) set $Q$ of punctures such that $S$ does not satisfy $(\phi,q)$ for all $q\in Q$, and
\[
m_{\phi_p(S)}(c_q)-m_S(c_q)=
\left\{
\begin{aligned}
&-1&&\text{if $q=p$}\\
&1&&\text{if $q\in Q$}\\
&0&&\text{otherwise}
\end{aligned}
\right\}
\text{ and }\ 
d_q^{\phi_p(S)}-d_q^S
\begin{cases}
\ge 0&\text{if $q=p$ and $Q=\emptyset$};\\
>\#Q&\text{if $q=p$ and $Q\neq\emptyset$};\\
=1&\text{if $q\in Q$};\\
=0&\text{otherwise}.
\end{cases}
\]
\end{lem}

\begin{proof}
It follows from the definition of $\phi_p$ and Lemma \ref{lem:phi Int} that $m_{\phi_p(S)}(c_p)-m_S(c_p)=-1$ and $d_p^{\phi_p(S)}-d_p^S\ge 0$. We take $Q$ as the set of all punctures $q$ such that $q\neq p$ and $m_{\phi_p(S)}(c_q)\neq m_S(c_q)$.

First, we consider $h_b$ for an angle $b\notin A_p$ instead of $c_q$. If there is not $a\in A_p^{\phi}$ with $\triangle_a=\triangle_b$, then $m_{\phi_p(S)}(h_b)=m_S(h_b)$ and $\Int(h_b,\phi_p(S))=\Int(h_b,S)$ since $\phi_p(S)\cap\triangle_b=S\cap\triangle_b$. Assume that there is $a\in A_p^{\phi}$ with $\triangle_a=\triangle_b$. In particular, $\triangle_a=\triangle_b$ is a triangle piece. The following observations give $\Int'(h_b,\phi_p(S))=\Int'(h_b,S)$, in particular, $\Int'(c_q,\widehat{\phi_p(S)})=\Int'(c_q,\widehat{S})$ for all punctures $q\neq p$.
\begin{itemize}
\item If $\Int(h_b,S)>0$ and $m_S(h_b)=0$, then $m_{\phi_p(S)}(h_b)=0$ by the definition of $\phi_a$ (see Tables \ref{table:smallest} and \ref{table:def phi}). Thus $\Int'(h_b,\phi_p(S))=\Int'(h_b,S)=0$.
\item If $\Int(h_b,S)>0$ and $m_S(h_b)>0$, then Proposition \ref{prop:commute tri} gives that $m_{\phi_p(S)}(h_b)=m_S(h_b)$ and $\Int(h_b,\phi_p(S))=\Int(h_b,S)$.
\item If $\Int(h_b,S)=0$, then $\Int(h_b,\phi_p(S))\le\Int(h_b,S)=0$ by Lemma \ref{lem:Int tri}, and $m_{\phi_p(S)}(h_b)-m_S(h_b)$ is either zero or one by the definition of $\phi_a$. In particular, if it is one, then $S$ satisfies $(\phi\ast,a)_k$ for some $k\in\{2,4,6\}$.
\end{itemize}

It follows from the above observations that the assertion holds for punctures that are not in $Q\sqcup\{p\}$. Let $q\in Q$. The above observations also give that for each $b\in A_q^{\min}(S)$, there must be $a\in A_p^{\phi}(S)$ with $\triangle_a=\triangle_b$, and $\Int(h_b,S)=0$ and
\[
m_{\phi_p(S)}(c_q)-m_S(c_q)=m_{\phi_p(S)}(h_b)-m_S(h_b)=1.
\]
Moreover, since $\Int(h_b,S)=0$ for all $b\in A_q^{\min}(S)$ and $S$ satisfies $(\ast 1,q)$, $A_q^{\phi}(S)$ must be empty. Thus $S$ does not satisfy $(\phi,q)$.

Finally, we only need to consider the difference $d_p^{\phi_p(S)}-d_p^S$ in the case that $Q\neq\emptyset$. Let $q\in Q$. For each $b\in A_q^{\min}(S)$, there is $a\in A_p^{\phi}(S)$ with $\triangle_a=\triangle_b$ such that $m_{\phi_p(S)}(h_b)-m_S(h_b)=1$, and $S$ satisfies $(\phi\ast,a)_k$ for some $k\in\{2,4,6\}$. In particular, since $m_S(v_c)=0$ for the third angle $c$ of $\triangle_a=\triangle_b$,
\begin{equation}\label{eq:ec}
\Int(e_c,S)=m_S(h_a)+m_S(h_b).
\end{equation}
For such a pair $(a,b)\in A_p^{\phi}(S)\times A_q^{\min}(S)$, we take a pair of angles in $A_p\times A_q$
\[
(a',b'):=
\left\{
\begin{aligned}
&(\overline{a},\overline{a}^r)&&\text{if $b=a^l$}\\
&(\overline{b}^r,\overline{b})&&\text{if $b=a^r$}
\end{aligned}
\right\}
\text{ that is, }\ 
\begin{tikzpicture}[baseline=-1mm]
\coordinate(u)at(0,1);\coordinate(d)at(0,-1);\coordinate(l)at(-1,0);\coordinate(r)at(1,0);
\draw(u)--(r)--(d)--(l)--(u)--(d);
\node at($(u)+(-67.5:0.5)$){$a$};\node at($(u)+(-112.5:0.5)$){$a'$};
\node at($(d)+(67.5:0.5)$){$b$};\node at($(d)+(112.5:0.5)$){$b'$};
\node at($(r)+(180:0.3)$){$c$};
\fill(u)circle(0.07);\fill(d)circle(0.07);\node at(0.3,1.1){$p$};\node at(0.3,-1.1){$q$};
\end{tikzpicture}
\ \text{ or }\ 
\begin{tikzpicture}[baseline=-1mm]
\coordinate(u)at(0,1);\coordinate(d)at(0,-1);\coordinate(l)at(-1,0);\coordinate(r)at(1,0);
\draw(u)--(r)--(d)--(l)--(u)--(d);
\node at($(u)+(-67.5:0.5)$){$a'$};\node at($(u)+(-112.5:0.5)$){$a$};
\node at($(d)+(67.5:0.5)$){$b'$};\node at($(d)+(112.5:0.5)$){$b$};
\node at($(l)+(0:0.3)$){$c$};
\fill(u)circle(0.07);\fill(d)circle(0.07);\node at(0.3,1.1){$p$};\node at(0.3,-1.1){$q$};
\end{tikzpicture}\ .
\]
Then the following inequalities hold:
\begin{align*}
m_S(c_q)&\le m_S(h_{b'})&&\text{(by the minimality of $m_S(c_q)$)}\\
&\le m_S(h_b)+m_S(h_a)-m_S(h_{a'})&&\text{(by (G1) and \eqref{eq:ec})}\\
&=m_S(c_q)+m_S(c_p)-m_S(h_{a'})&&\text{(by $(a,b)\in A_p^{\phi}(S)\times A_q^{\min}(S)$ and $(\ast 1,p)$)}\\
&\le m_S(c_q)&&\text{(by the minimality of $m_S(c_p)$)}.
\end{align*}
Thus $m_S(h_{b'})=m_S(c_q)$. Since
\[
m_{\phi_p(S)}(h_{b'})\ge m_{\phi_p(S)}(c_q)=m_S(c_q)+1=m_S(h_{b'})+1,
\]
$S$ satisfies $(\phi\ast,a')_k$ for some $k\in\{2,4,6\}$. This means that $A_q$ must contain at least two angles each of which is an angle of $\triangle_d$ for $d\in A_p^{\phi}(S)$ where $S$ satisfies $(\phi\ast,d)_k$ for some $k\in\{2,4,6\}$. For $k\in\{2,4,6\}$, we set
\[
A_k:=\{d\in A_p^{\phi}(S)\mid\text{$S$ satisfies $(\phi\ast,d)_k$}\}.
\]
Since $\#(\phi_d(S)\setminus S)=1$ for $d\in A_2\sqcup A_4$ and $\#(\phi_d(S)\setminus S)=2$ for $d\in A_6$, we have the inequality
\begin{equation}\label{eq:246}
\#Q\le\frac{1}{2}\#(A_2\sqcup A_4)+\#A_6.
\end{equation}
Therefore,
\begin{align*}
\Int'(c_p,\widehat{\phi_p(S)})-\Int'(c_p,\widehat{S})&\le-\frac{3}{2}\#(A_2\sqcup A_4)-2\#A_6&&\text{(by \eqref{eq:Int c} and the definition of $\phi_d$)}\\
&\le-2\#Q&&\text{(by \eqref{eq:246})}.
\end{align*}
If $\#Q=1$, then the distinct angles $a$ and $a'$ above are contained in $A_2\sqcup A_4\sqcup A_6$, that is,
\[
-\frac{3}{2}\#(A_p^2\sqcup A_4)-2\#A_6\le-\frac{3}{2}\left(\#(A_2\sqcup A_4)+\#A_6\right)\le-3<-2=-2\#Q.
\]
Therefore, if $Q\neq\emptyset$, then
\[
d_p^{\phi_p(S)}-d_p^S=-1-\left(\Int'(c_p,\widehat{\phi_p(S)})-\Int'(c_p,\widehat{S})\right)
\begin{cases}
>2\#Q-1=\#Q&\text{if $\#Q=1$};\\
\ge 2\#Q-1>\#Q&\text{if $\#Q>1$}.
\end{cases}
\]
Thus the desired inequality holds.
\end{proof}

\begin{example}\label{ex:new}
In this example, we keep the notations in Example \ref{ex:T Int}. Let $p$ (resp., $q$, $r$) be the bottom (resp., left, right) puncture in $\cS$.
 
(1) From Example \ref{ex:p-modif}, we obtain the following equalities:
\begin{gather*}
m_{\phi_p(S_U)}(c_s)-m_{S_U}(c_s)=
\left\{
\begin{aligned}
&-1&&\text{if $s=p$}\\
&0&&\text{otherwise}
\end{aligned}
\right\}
\text{ and }\ 
d_s^{\phi_p(S_U)}-d_s^{S_U}=
\begin{cases}
1&\text{if $s=p$};\\
0&\text{otherwise},
\end{cases}
\\
m_{\phi_p^2(S_U)}(c_s)-m_{\phi_p(S_U)}(c_s)=
\left\{
\begin{aligned}
&-1&&\text{if $s=p$}\\
&0&&\text{otherwise}
\end{aligned}
\right\}
\text{ and }\ 
d_s^{\phi_p^2(S_U)}-d_s^{\phi_p(S_U)}=0\ \text{ for any $s$}.
\end{gather*}

(2) We consider a new tagged triangulation of $\cS$
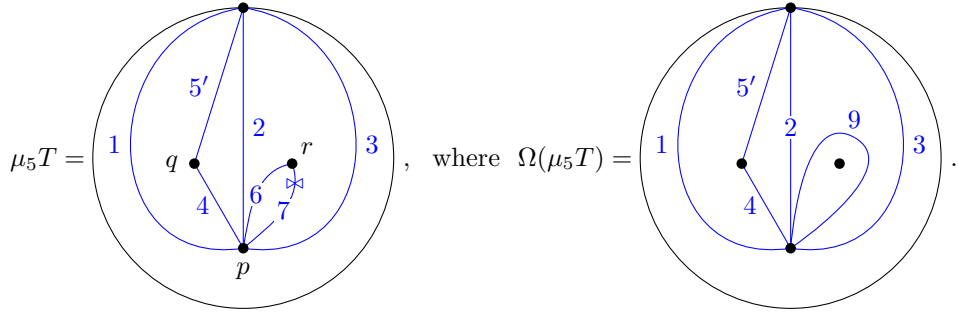
\[
\mu_5T=
\begin{tikzpicture}[baseline=-1mm]
\coordinate(d)at(0,-1.2);\coordinate(u)at(90:2);
\coordinate(l)at($(d)+(120:1.3)$);\coordinate(r)at($(d)+(60:1.3)$);
\draw(0,0)circle(2);
\draw[blue](u)..controls(-2,1.7)and(-2,-1.5)..node[left]{$1$}(d);
\draw[blue](u)--node[right]{$2$}(d)--node[left]{$4$}(l)--node[left]{$5'$}(u);
\draw[blue](u)..controls(2,1.7)and(2,-1.5)..node[right]{$3$}(d);
\draw[blue](d)to[out=20,in=130,relative]node[fill=white,inner sep=1]{$6$}(r);
\draw[blue](d)to[out=-20,in=-130,relative]node[fill=white,inner sep=1]{$7$}node[pos=0.8]{\scriptsize\rotatebox{0}{$\bowtie$}}(r);
\fill(l)circle(0.07);\fill(r)circle(0.07);\fill(d)circle(0.07);\fill(u)circle(0.07);
\node at($(d)+(0,-0.3)$){$p$};\node at($(l)+(-0.3,0)$){$q$};\node at($(r)+(0.2,0.2)$){$r$};
\end{tikzpicture}
\ ,\ \text{ where }\ 
\Omega(\mu_5T)=
\begin{tikzpicture}[baseline=-1mm]
\coordinate(d)at(0,-1.2);\coordinate(u)at(90:2);
\coordinate(l)at($(d)+(120:1.3)$);\coordinate(r)at($(d)+(60:1.3)$);
\draw(0,0)circle(2);
\draw[blue](u)..controls(-2,1.7)and(-2,-1.5)..node[left]{$1$}(d);
\draw[blue](u)--node[fill=white,inner sep=1]{$2$}(d)--node[left]{$4$}(l)--node[left]{$5'$}(u);
\draw[blue](u)..controls(2,1.7)and(2,-1.5)..node[right]{$3$}(d);
\draw[blue](d)to[out=25,in=90,relative]($(d)+(60:1.7)$);
\draw[blue](d)to[out=-25,in=-90,relative]($(d)+(60:1.7)$)node[above]{$9$};
\fill(l)circle(0.07);\fill(r)circle(0.07);\fill(d)circle(0.07);\fill(u)circle(0.07);
\end{tikzpicture}\ .
\]
For a multi-set $V=\{\rho(4)^3,\rho(5)^2\}\in\bM_{\cS}$, where $\Omega(V)=\{\rho(4),\rho(8)^2\}$, we define $S_V$ as a multi-set of non-loop edges and segments in puzzle pieces of $\Omega(\mu_5T)$. Then $S_V$ is glueable and satisfies $(\phi\ast,p)$, $(\ast,q)$, and $(\ast,r)$ since $A_q^{\phi}(S_V)=\emptyset=A_r$. Thus $\widehat{S_V}=V^{\circ}=\{8^2,c_p^5,c_q\}$ and the maximal $p$-modification $\phi_p(S_V)$ of $S_V$ are given as follows:
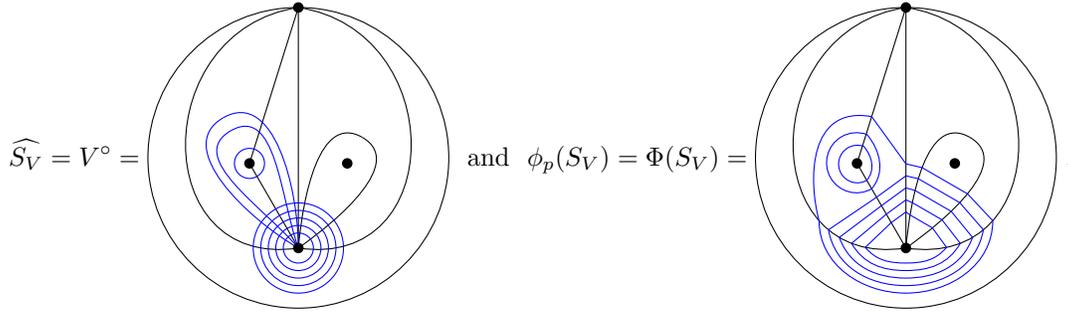
\[
\widehat{S_V}=V^{\circ}=
\begin{tikzpicture}[baseline=-1mm]
\coordinate(d)at(0,-1.2);\coordinate(u)at(90:2);
\coordinate(l)at($(d)+(120:1.3)$);\coordinate(r)at($(d)+(60:1.3)$);
\draw(0,0)circle(2);
\draw(u)..controls(-2,1.7)and(-2,-1.5)..(d);
\draw(u)--(d)--(l)--(u);
\draw(u)..controls(2,1.7)and(2,-1.5)..(d);
\draw(d)to[out=25,in=90,relative]($(d)+(60:1.7)$);
\draw(d)to[out=-25,in=-90,relative]($(d)+(60:1.7)$);
\draw[blue](l)circle(0.2);\draw[blue](d)circle(0.2)circle(0.3)circle(0.4)circle(0.5)circle(0.6);
\draw[blue](d)to[out=30,in=90,relative]($(d)+(120:2)$);\draw[blue](d)to[out=-30,in=-90,relative]($(d)+(120:2)$);
\draw[blue](d)to[out=10,in=90,relative]($(d)+(120:1.8)$);\draw[blue](d)to[out=-10,in=-90,relative]($(d)+(120:1.8)$);
\fill(l)circle(0.07);\fill(r)circle(0.07);\fill(d)circle(0.07);\fill(u)circle(0.07);
\end{tikzpicture}
\ \text{ and }\ 
\phi_p(S_V)=\Phi(S_V)=
\begin{tikzpicture}[baseline=-1mm]
\coordinate(d)at(0,-1.2);\coordinate(u)at(90:2);
\coordinate(l)at($(d)+(120:1.3)$);\coordinate(r)at($(d)+(60:1.3)$);
\draw(0,0)circle(2);
\draw(u)..controls(-2,1.7)and(-2,-1.5)..coordinate[pos=0.74](15)coordinate[pos=0.78](14)coordinate[pos=0.82](13)coordinate[pos=0.86](12)coordinate[pos=0.9](11)(d);
\draw(u)--coordinate[pos=0.65](25)coordinate[pos=0.7](24)coordinate[pos=0.75](23)coordinate[pos=0.8](22)coordinate[pos=0.85](21)(d);
\draw(u)..controls(2,1.7)and(2,-1.5)..coordinate[pos=0.74](35)coordinate[pos=0.78](34)coordinate[pos=0.82](33)coordinate[pos=0.86](32)coordinate[pos=0.9](31)(d);
\draw(d)--coordinate[pos=0.3](41)coordinate[pos=0.4](42)coordinate[pos=0.5](43)coordinate[pos=0.6](44)coordinate[pos=0.7](45)coordinate[pos=0.8](46)(l);
\draw(l)--coordinate[pos=0.1](51)coordinate[pos=0.2](52)coordinate[pos=0.3](53)(u);
\draw(d)to[out=25,in=90,relative]coordinate[pos=0.2](9l1)coordinate[pos=0.27](9l2)coordinate[pos=0.34](9l3)coordinate[pos=0.41](9l4)coordinate[pos=0.48](9l5)($(d)+(60:1.7)$);
\draw(d)to[out=-25,in=-90,relative]coordinate[pos=0.2](9r1)coordinate[pos=0.27](9r2)coordinate[pos=0.34](9r3)coordinate[pos=0.41](9r4)coordinate[pos=0.48](9r5)($(d)+(60:1.7)$);
\draw[blue](11)--(41)--(21)--(9l1)--(9r1)--(31)to[out=-130,in=0](0,-1.4)to[out=180,in=-50](11);
\draw[blue](12)--(42)--(22)--(9l2)--(9r2)--(32)to[out=-120,in=0](0,-1.5)to[out=180,in=-60](12);
\draw[blue](13)--(43)--(23)--(9l3)--(9r3)--(33)to[out=-110,in=0](0,-1.6)to[out=180,in=-70](13);
\draw[blue](14)--(44)--(24)--(9l4)--(9r4)--(34)to[out=-100,in=0](0,-1.7)to[out=180,in=-80](14);
\draw[blue](15)..controls(-1.3,-0.3)and(-1.3,0.7)..(53)--(25)--(9l5)--(9r5)--(35)to[out=-90,in=0](0,-1.8)to[out=180,in=-90](15);
\draw[blue](45)arc(-60:-290:0.4)to[out=-50,in=50](45);\draw[blue](46)arc(-60:-290:0.25)to[out=-50,in=50](46);
\fill(l)circle(0.07);\fill(r)circle(0.07);\fill(d)circle(0.07);\fill(u)circle(0.07);
\end{tikzpicture}\ .
\]
Therefore,
\[
m_{\phi_p(S_V)}(c_s)-m_{S_V}(c_s)=
\left\{
\begin{aligned}
&-1&&\text{if $s=p$}\\
&1&&\text{if $s=q$}\\
&0&&\text{if $s=r$}
\end{aligned}
\right\}
\text{ and }\ 
d_s^{\phi_p(S_V)}-d_s^{S_V}=
\begin{cases}
3&\text{if $s=p$};\\
1&\text{if $s=q$};\\
0&\text{if $s=r$}.
\end{cases}
\]
Remark that $\phi_p(S_V)$ satisfies $(\psi\ast,p)$ and $(\psi\ast,q)$, and $\psi_p\phi_p(S_V)\neq\psi_q\phi_p(S_V)$ (see Example \ref{ex:recoverable}).
\end{example}

Repeating Lemma \ref{lem:phi difference}, we obtain a similar result for general $p$-modifications.

\begin{prop}\label{prop:phi difference}
Assume that $S$ is glueable and satisfies $(\ast 1,p)$ for all punctures $p$. Let
\[
\phi=\prod_{p\in P_{\phi}}\phi_p^{k_p}
\]
be a $p$-modification of $S$, where $k_p\in\bZ_{>0}$ for $p\in P_{\phi}$. Then there is a (possibly empty) set $Q_{\phi}$ of punctures with $Q_{\phi}\cap P_{\phi}=\emptyset$, and $a_p\in\bZ_{> 0}$ for $p\in Q_{\phi}$ such that
\[
m_{\phi(S)}(c_p)-m_S(c_p)=
\left\{
\begin{aligned}
&-k_p&&\text{if $p\in P_{\phi}$}\\
&a_p&&\text{if $p\in Q_{\phi}$}\\
&0&&\text{otherwise}
\end{aligned}
\right\}
\text{ and }\ 
d_p^{\phi(S)}-d_p^S
\begin{cases}
\ge 0&\text{if $p\in P_{\phi}$};\\
=a_p&\text{if $p\in Q_{\phi}$};\\
=0&\text{otherwise}.
\end{cases}
\]
Moreover, if $Q_{\phi}\neq\emptyset$, then
\[
\sum_{p\in P_{\phi}}(d_p^{\phi(S)}-d_p^S)>\sum_{p\in Q_{\phi}}a_p.
\]
\end{prop}

\begin{proof}
Let $p\in P_{\phi}$. Applying Lemma \ref{lem:phi difference} to $\phi_p^s(S)$ and $p$ for $s\in\{0,1,\ldots,k_p-1\}$, we obtain the set $Q_p^s$ of punctures in Lemma \ref{lem:phi difference}. We consider the sum of $Q_p^s$
\[
\overline{Q}_p:=\bigsqcup_{s=0}^{k_p-1}Q_p^s
\]
as a multi-set and its underlying set $Q_p$. Then $Q_p\cap P_{\phi}=\emptyset$ and
\[
m_{\phi_p^{k_p}(S)}(c_r)-m_S(c_r)=
\left\{
\begin{aligned}
&-k_p&&\text{if $r=p$}\\
&m_{\overline{Q}_p}(r)&&\text{otherwise}
\end{aligned}
\right\},\ 
d_r^{\phi_p^{k_p}(S)}-d_r^S
\begin{cases}
\ge 0&\text{if $r=p$ and $Q_p=\emptyset$};\\
>\#\overline{Q}_p&\text{if $r=p$ and $Q_p\neq\emptyset$};\\
=m_{\overline{Q}_p}(r)&\text{otherwise}.
\end{cases}
\]

Next, let $q\in P_{\phi}\setminus\{p\}$. Applying Lemma \ref{lem:phi difference} to $\phi_{q}^{t}\phi_p^{k_p}(S)$ and $q$ for $t\in\{0,1,\ldots,k_{q}-1\}$, we obtain the set $Q_{q}^{t}$ of punctures in Lemma \ref{lem:phi difference}. We consider the sum of $Q_{q}^{t}$
\[
\overline{Q}_{q}:=\bigsqcup_{t=0}^{k_{q}-1}Q_{q}^{t}
\]
as a multi-set and its underlying set $Q_{q}$. Then $Q_{q}\cap P_{\phi}=\emptyset$ and
\[
m_{\phi_{q}^{k_{q}}\phi_p^{k_p}(S)}(c_r)-m_S(c_r)=
\begin{cases}
-k_{p}&\text{if $r=p$};\\
-k_{q}&\text{if $r=q$};\\
m_{\overline{Q}_p}(r)+m_{\overline{Q}_{q}}(r)&\text{otherwise},
\end{cases}
\]
\[
d_r^{\phi_{q}^{k_{q}}\phi_p^{k_p}(S)}-d_r^S
\begin{cases}
\ge 0&\text{if $r$ is $p$ or $q$, and $Q_p=Q_{q}=\emptyset$};\\
>\#\overline{Q}_p&\text{if $r=p$ and $Q_p\neq\emptyset$};\\
>\#\overline{Q}_{q}&\text{if $r=q$ and $Q_{q}\neq\emptyset$};\\
=m_{\overline{Q}_{q}}(r)&\text{otherwise}.
\end{cases}
\]
Repeating these processes, we obtain a multi-set $\overline{Q}_p$ for each $p\in P_{\phi}$, and the underlying set $Q_{\phi}$ of
\[
\overline{Q}_{\phi}:=\bigsqcup_{p\in P_{\phi}}\overline{Q}_p
\]
is the desired set. In fact, taking $a_p=m_{\overline{Q}_{\phi}}(p)$ for $p\in Q_{\phi}$, the assertion holds. In particular, if $Q_{\phi}\neq\emptyset$, then
\[
\sum_{p\in P_{\phi}}(d_p^{\phi(S)}-d_p^S)>\sum_{p\in P_{\phi}}\#\overline{Q}_p=\#\overline{Q}_{\phi}=\sum_{p\in Q_{\phi}}m_{\overline{Q}_{\phi}}(p)=\sum_{p\in Q_{\phi}}a_p,
\]
where the first inequality follows from $Q_{\phi}\cap P_{\phi}=\emptyset$. Finally, we remark that each $\overline{Q}_p$ depends on the order of choosing $p\in P_{\phi}$ in the above processes, but $\overline{Q}_{\phi}$ is independent of that since $\phi(S)$ is well-defined.
\end{proof}

We recall that $V_T$ is the set of all punctures in $\cS$ incident to $\Omega(T)$, and $E_T$ is the set of all plain curves in $\Omega(T)$ connecting punctures. For $p,q\in V_T$, let $E_T^{pq}$ be the set of all plain curves in $\Omega(T)$ connecting $p$ and $q$. Notice that for $\gamma\in E_T^{pq}$,
\begin{equation}\label{eq:difference}
m_S(c_p)+m_S(c_q)\le\Int(\gamma,\widehat{S}).
\end{equation}

\begin{defn}\label{def:characterized}
We say that $S$ is \emph{characterized by $(n_{\gamma})_{\gamma\in E_T}\in\bZ_{\ge 0}^{E_T}$} if the following hold:
\begin{itemize}
\item[(C1)] For $\gamma\in E_T^{pq}$, if $n_{\gamma}>0$, then $m_S(c_p)+m_S(c_q)=\Int(\gamma,\widehat{S})$.
\item[(C2)] For $p\in V_T$,
\[
d_p^S=\sum_{q\in V_T\setminus\{p\}}\sum_{\gamma\in E_T^{pq}}n_{\gamma}+2\sum_{\gamma\in E_T^{pp}}n_{\gamma}.
\]
\end{itemize}
\end{defn}

We are ready to give a sufficient condition of $S$ such that it is recoverable from a $p$-modification of it.

\begin{defn}\label{def:recoverable}
We say that $S$ is \emph{$p$-recoverable} if it satisfies the following:
\begin{itemize}
\item[(1)] It satisfies $(\ast 1,p)$ for all punctures $p$ (Definition \ref{def:ast p}).
\item[(2)] It is glueable (Definition \ref{def:glueable}).
\item[(3)] It has no enclosed punctures (Definition \ref{def:enclosed}).
\item[(4)] It is characterized by a non-negative integer vector (Definition \ref{def:characterized}).
\end{itemize}
\end{defn}

\begin{thm}\label{thm:unique S}
Let $S$ and $S'$ be $p$-recoverable multi-sets of non-loop edges and segments in puzzle pieces of $\Omega(T)$, and $\phi$ and $\phi'$ be $p$-modifications of $S$ and $S'$, respectively. If $\phi(S)=\phi'(S')$, then $S=S'$.
\end{thm}

\begin{proof}
We set
\[
\phi=\prod_{p\in P_{\phi}}\phi_p^{k_p}\ \text{ and }\ \phi'=\prod_{p\in P_{\phi'}}\phi_p^{h_p},
\]
where $P_{\phi}$ and $P_{\phi'}$ are sets of some punctures and $k_p,h_{p'}\in\bZ_{>0}$ for $p\in P_{\phi}$ and $p'\in P_{\phi'}$.

First, we assume that $P_{\phi}\cap P_{\phi'}=\emptyset$. Since $S$ and $S'$ are glueable and satisfy $(\ast 1,p)$ for all punctures $p$, we can apply Proposition \ref{prop:phi difference} to both pairs $(S,\phi)$ and $(S',\phi')$. Then there are (possibly empty) subsets $Q_{\phi}$ and $Q_{\phi'}$ of $V_T$ with $Q_{\phi}\cap P_{\phi}=\emptyset$ and $Q_{\phi'}\cap P_{\phi'}=\emptyset$, and $a_p,b_p\in\bZ_{\ge 0}$ with $a_p=0$ (resp., $b_p=0$) if and only if $p\notin Q_{\phi}$ (resp., $p\notin Q_{\phi'}$) such that
\[
m_{S'}(c_p)-m_S(c_p)=\begin{array}{l}m_{S'}(c_p)-m_{\phi'(S')}(c_p)\\\ +m_{\phi(S)}(c_p)-m_S(c_p)\end{array}=
\begin{cases}
-k_p-b_p&\text{if $p\in P_{\phi}$};\\
h_p+a_p&\text{if $p\in P_{\phi'}$};\\
a_p-b_p&\text{if $p\in(Q_{\phi}\cup Q_{\phi'})\setminus(P_{\phi}\cup P_{\phi'})$};\\
0&\text{otherwise},
\end{cases}
\]
\[
d_p^{S'}-d_p^S=d_p^{S'}-d_p^{\phi'(S')}+d_p^{\phi(S)}-d_p^S
\begin{cases}
\ge -b_p&\text{if $p\in P_{\phi}$};\\
\le a_p&\text{if $p\in P_{\phi'}$};\\
=a_p-b_p&\text{if $p\in(Q_{\phi}\cup Q_{\phi'})\setminus(P_{\phi}\cup P_{\phi'})$};\\
=0&\text{otherwise}.
\end{cases}
\]
Moreover, if $Q_{\phi}\neq\emptyset$, then
\begin{equation}\label{eq:>}
\sum_{p\in P_{\phi}}(d_p^{S'}-d_p^S)=\sum_{p\in P_{\phi}}(d_p^{S'}-d_p^{\phi'(S')})+\sum_{p\in P_{\phi}}(d_p^{\phi(S)}-d_p^S)
>-\sum_{p\in P_{\phi}}b_p+\sum_{p\in Q_{\phi}}a_p.
\end{equation}
Similarly, if $Q_{\phi'}\neq\emptyset$, then
\begin{equation}\label{eq:<}
\sum_{p\in P_{\phi'}}(d_p^{S'}-d_p^S)<-\sum_{p\in Q_{\phi'}}b_p+\sum_{p\in P_{\phi'}}a_p.
\end{equation}

We consider two sets
\[
V_+:=\{p\in V_T\mid m_{S'}(c_p)-m_S(c_p)>0\}\ \text{ and }\ V_-:=\{q\in V_T\mid m_{S'}(c_q)-m_S(c_q)<0\},
\]
where $P_{\phi'}\subseteq V_+\subseteq P_{\phi'}\cup Q_{\phi}$ and $P_{\phi}\subseteq V_-\subseteq P_{\phi}\cup Q_{\phi'}$. If $Q_{\phi}=\emptyset$, then
\[
\sum_{q\in V_-}(d_q^{S'}-d_q^S)\ge\sum_{q\in V_-}(-b_q)\ge-\sum_{q\in Q_{\phi'}}b_q
\]
since $b_q=0$ for $q\notin Q_{\phi'}$. If $Q_{\phi}\neq\emptyset$, then
\begin{align*}
\sum_{q\in V_-}(d_q^{S'}-d_q^S)&=\sum_{q\in P_{\phi}}(d_q^{S'}-d_q^S)+\sum_{q\in V_-\setminus P_{\phi}}(d_q^{S'}-d_q^S)&&\text{(by $P_{\phi}\subseteq V_-$)}\\
&>\sum_{p\in Q_{\phi}}a_p-\sum_{q\in P_{\phi}}b_q+\sum_{q\in V_-\setminus P_{\phi}}(a_q-b_q)&&\text{(by \eqref{eq:>} and $(V_-\setminus P_{\phi})\subseteq Q_{\phi'}$)}\\
&\ge\sum_{p\in Q_{\phi}}a_p-\sum_{q\in V_-}b_q\\
&\ge\sum_{p\in Q_{\phi}}a_p-\sum_{q\in Q_{\phi'}}b_q&&\text{(since $b_q=0$ for $q\notin Q_{\phi'}$)}.
\end{align*}
Similarly, if $Q_{\phi'}=\emptyset$, then
\[
\sum_{p\in V_+}(d_p^{S'}-d_p^S)\le\sum_{p\in V_+}a_p\le\sum_{p\in Q_{\phi}}a_p
\]
since $a_p=0$ for $p\notin Q_{\phi}$. If $Q_{\phi'}\neq\emptyset$, then
\begin{align*}
\sum_{p\in V_+}(d_p^{S'}-d_p^S)&=\sum_{p\in P_{\phi'}}(d_p^{S'}-d_p^S)+\sum_{p\in V_+\setminus P_{\phi'}}(d_p^{S'}-d_p^S)&&\text{(by $P_{\phi'}\subseteq V_+$)}\\
&<\sum_{p\in P_{\phi'}}a_p-\sum_{q\in Q_{\phi'}}b_q+\sum_{p\in V_+\setminus P_{\phi'}}(a_p-b_p)&&\text{(by \eqref{eq:<} and $(V_+\setminus P_{\phi'})\subseteq Q_{\phi}$)}\\
&\le\sum_{p\in V_+}a_p-\sum_{q\in Q_{\phi'}}b_q\\
&\le\sum_{p\in Q_{\phi}}a_p-\sum_{q\in Q_{\phi'}}b_q&&\text{(since $a_p=0$ for $p\notin Q_{\phi}$)}.
\end{align*}
Therefore, in all cases,
\begin{equation}\label{eq:P+-}
\sum_{p\in V_+}(d_p^{S'}-d_p^S)\le\sum_{p\in Q_{\phi}}a_p-\sum_{q\in Q_{\phi'}}b_q\le\sum_{q\in V_-}(d_q^{S'}-d_q^S),
\end{equation}
and if the equalities hold, then $Q_{\phi}=Q_{\phi'}=\emptyset$.

On the other hand, by the assumption, $S$ is characterized by some $(n_{\gamma})_{\gamma\in E_T}\in\bZ_{\ge 0}^{E_T}$ and $S'$ is characterized by some $(n_{\gamma}')_{\gamma\in E_T}\in\bZ_{\ge 0}^{E_T}$. For $p\in V_+$, if there is no $\gamma\in E_T^{pq}$ with $n_{\gamma}>0$ for some $q\in V_T$, then $d_p^S=0$ by (C2); If there is $\gamma\in E_T^{pq}$ with $n_{\gamma}>0$ for some $q\in V_T$, then $m_S(c_p)$, $m_S(c_q)$, and $m_{S'}(c_p)$ are positive by (C2) and $p\in V_+$. If $m_{S'}(c_q)>0$, then
\begin{align*}
m_{S'}(c_q)&\le\Int(\gamma,\widehat{S'})-m_{S'}(c_p)&&\text{(by \eqref{eq:difference})}\\
&=\Int(\gamma,\widehat{\phi'(S')})-m_{S'}(c_p)&&\text{(by Lemmas \ref{lem:Int tri} and \ref{lem:Int mono})}\\
&<\Int(\gamma,\widehat{\phi(S)})-m_S(c_p)&&\text{(by $\phi(S)=\phi'(S')$ and $p\in V_+$)}\\
&=\Int(\gamma,\widehat{S})-m_S(c_p)&&\text{(by Lemmas \ref{lem:Int tri} and \ref{lem:Int mono})}\\
&=m_S(c_q)&&\text{(by (C1))},
\end{align*}
thus $q\in V_-$. If $m_{S'}(c_q)=0$, then it is clear that $q\in V_-$. Therefore, since $V_+\cap V_-=\emptyset$, (C2) gives
\begin{equation}\label{eq:P+ dm}
\sum_{p\in V_+}d_p^S=\sum_{p\in V_+}\sum_{q\in V_-}\sum_{\gamma\in E_T^{pq}}n_{\gamma}
\le\sum_{p\in V_T}\sum_{q\in V_-}\sum_{\gamma\in E_T^{pq}}n_{\gamma}\le\sum_{q\in V_-}d_q^S.
\end{equation}
Similarly, for $q\in V_-$, if there is no $\gamma\in E_T^{pq}$ with $n_{\gamma}'>0$ for some $p\in V_T$, then $d_q^S=0$ by (C2); If there is $\gamma\in E_T^{pq}$ with $n_{\gamma}'>0$ for some $p\in V_T$, then $m_S(c_q)$, $m_{S'}(c_p)$, and $m_{S'}(c_q)$ are positive by (C2) and $q\in V_-$. If $m_S(c_p)>0$, then we get that $m_S(c_p)<m_{S'}(c_p)$ in the same way as above, thus $p\in V_+$. If $m_S(c_p)=0$, then it is clear that $p\in V_+$. Therefore, since $V_+\cap V_-=\emptyset$, (C2) gives
\begin{equation}\label{eq:P- dm}
\sum_{q\in V_-}d_q^{S'}=\sum_{q\in V_-}\sum_{p\in V_+}\sum_{\gamma\in E_T^{pq}}n_{\gamma}'
\le\sum_{q\in V_T}\sum_{p\in V_+}\sum_{\gamma\in E_T^{pq}}n_{\gamma}'\le\sum_{p\in V_+}d_p^{S'}.
\end{equation}

The inequalities \eqref{eq:P+-}, \eqref{eq:P+ dm}, and \eqref{eq:P- dm} induce
\[
0\le\sum_{q\in V_-}d_q^S-\sum_{p\in V_+}d_p^S\le\sum_{q\in V_-}d_q^{S'}-\sum_{p\in V_+}d_p^{S'}\le 0.
\]
This means that the equalities in \eqref{eq:P+-} hold, that is, $Q_{\phi}=Q_{\phi'}=\emptyset$. Therefore, $V_+=P_{\phi'}$ and $V_-=P_{\phi}$. Moreover, the equalities in \eqref{eq:P+ dm} and \eqref{eq:P- dm} also hold, that is, 
\[
\sum_{q\in P_{\phi}}d_q^S=\sum_{p\in P_{\phi'}}d_p^S\ \text{ and }\ \sum_{q\in P_{\phi}}d_q^{S'}=\sum_{p\in P_{\phi'}}d_p^{S'}.
\]
On the other hand, by Proposition \ref{prop:phi difference},
\[
\sum_{q\in P_{\phi}}d_q^{\phi(S)}\ge\sum_{q\in P_{\phi}}d_q^S\ \text{ and }\ \sum_{p\in P_{\phi'}}d_p^{\phi(S)}=\sum_{p\in P_{\phi'}}d_p^S,
\]
where the first equality holds if and only if each $q\in P_{\phi}$ is enclosed in $\phi(S)$ by Lemma \ref{lem:phi Int}. Similarly, by Proposition \ref{prop:phi difference},
\[
\sum_{q\in P_{\phi}}d_q^{\phi'(S')}=\sum_{q\in P_{\phi}}d_q^{S'}\ \text{ and }\ \sum_{p\in P_{\phi'}}d_p^{\phi'(S')}\ge\sum_{p\in P_{\phi'}}d_p^{S'},
\]
where the second equality holds if and only if each $p\in P_{\phi'}$ is enclosed in $\phi'(S')$ by Lemma \ref{lem:phi Int}. Therefore, these inequalities and $\phi(S)=\phi'(S')$ induce
\begin{align*}
\sum_{q\in P_{\phi}}d_q^{\phi(S)}&\ge\sum_{q\in P_{\phi}}d_q^S=\sum_{p\in P_{\phi'}}d_p^S=\sum_{p\in P_{\phi'}}d_p^{\phi(S)}=\sum_{p\in P_{\phi'}}d_p^{\phi'(S')}\\
&\ge\sum_{p\in P_{\phi'}}d_p^{S'}=\sum_{q\in P_{\phi}}d_q^{S'}=\sum_{q\in P_{\phi}}d_q^{\phi'(S')}=\sum_{q\in P_{\phi}}d_q^{\phi(S)}.
\end{align*}
Thus each puncture in $P_{\phi}\sqcup P_{\phi'}$ is enclosed in $\phi(S)=\phi'(S')$. However, since $S$ has no enclosed punctures and $P_{\phi}\cap P_{\phi'}=\emptyset$, each puncture in $P_{\phi'}$ is not enclosed in $\phi(S)$ by Lemma \ref{lem:enclosed}, that is, $P_{\phi'}=\emptyset$. Similarly, each puncture in $P_{\phi}$ is not enclosed in $\phi'(S')$ by Lemma \ref{lem:enclosed}, that is, $P_{\phi}=\emptyset$. Therefore, $S=\phi(S)=\phi'(S')=S'$.

Finally, we assume that $P_{\phi}\cap P_{\phi'}\neq\emptyset$. We take
\[
\psi=\prod_{p\in P_{\phi}\cap P_{\phi'}}\psi_p^{\min\{k_p,h_p\}}.
\]
By Proposition \ref{prop:commute}, $\psi\phi(S)=\psi\phi'(S')$ is well-defined. Then $P_{\psi\phi}\cap P_{\psi\phi'}=\emptyset$ since 
\[
P_{\psi\phi}=P_{\phi}\setminus\{p\in P_{\phi}\cap P_{\phi'}\mid k_p\le h_p\}\ \text{ and }\ P_{\psi\phi'}=P_{\phi'}\setminus\{p\in P_{\phi}\cap P_{\phi'}\mid k_p\ge h_p\}.
\]
Therefore, the proof reduces to the above case.
\end{proof}

\begin{example}\label{ex:recoverable}
We consider (1) and (2) in Example \ref{ex:new}.

(1) Since the bottom puncture $p$ is enclosed in $\phi_p^2(S_U)$ by Example \ref{ex:enclosed}, $\phi_p^2(S_U)$ is not $p$-recoverable. If $\phi_p(S_U)$ is characterized by $(n_{\gamma})_{\gamma\in E_T}=(n_8,n_9)$, then $d_p^{\phi_p(S_U)}=4-1=3$ must be equal to $2(n_8+n_9)$ by $(C2)$, a contradiction. Thus $\phi_p(S_U)$ is not $p$-recoverable. Finally, it is not difficult to check that $S_U$ is $p$-recoverable (see Proposition \ref{prop:properties}) and does not satisfy $(\psi\ast,p)$. Therefore, $S_U$ is a unique $p$-recoverable multi-set such that $\phi_p^2(S_U)$ is its $p$-modification.

(2) We can see that $S_V$ is a unique $p$-recoverable multi-set such that $\phi_p(S_V)$ is its $p$-modification as follows: The multi-set $\phi_p(S_V)$ satisfies $(\psi\ast,p)$ and $(\psi\ast,q)$, but not $(\psi\ast,r)$. Moreover, $\psi_q\phi_p(S_V)$ does not satisfy $(\psi\ast,p)$, and $\psi_p\phi_p(S_V)=S_V$ satisfies neither $(\psi\ast,p)$ nor $(\psi\ast,q)$. Therefore, a multi-set such that $\phi_p(S_V)$ is its $p$-modification is given by either $S_V$, $\phi_p(S_V)$, or one of the following:
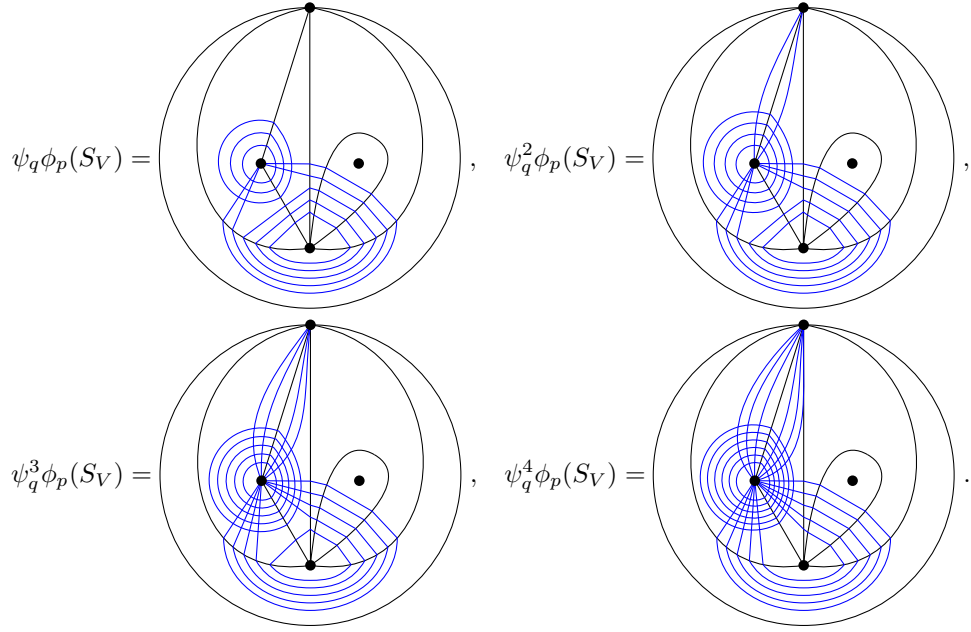
\begin{gather*}
\psi_q\phi_p(S_V)=
\begin{tikzpicture}[baseline=-1mm]
\coordinate(d)at(0,-1.2);\coordinate(u)at(90:2);
\coordinate(l)at($(d)+(120:1.3)$);\coordinate(r)at($(d)+(60:1.3)$);
\draw(0,0)circle(2);
\draw(u)..controls(-2,1.7)and(-2,-1.5)..coordinate[pos=0.74](15)coordinate[pos=0.78](14)coordinate[pos=0.82](13)coordinate[pos=0.86](12)coordinate[pos=0.9](11)(d);
\draw(u)--coordinate[pos=0.65](25)coordinate[pos=0.7](24)coordinate[pos=0.75](23)coordinate[pos=0.8](22)coordinate[pos=0.85](21)(d);
\draw(u)..controls(2,1.7)and(2,-1.5)..coordinate[pos=0.74](35)coordinate[pos=0.78](34)coordinate[pos=0.82](33)coordinate[pos=0.86](32)coordinate[pos=0.9](31)(d);
\draw(d)--coordinate[pos=0.3](41)coordinate[pos=0.4](42)coordinate[pos=0.5](43)coordinate[pos=0.6](44)coordinate[pos=0.7](45)coordinate[pos=0.8](46)(l);
\draw(l)--coordinate[pos=0.1](51)coordinate[pos=0.2](52)coordinate[pos=0.3](53)(u);
\draw(d)to[out=25,in=90,relative]coordinate[pos=0.2](9l1)coordinate[pos=0.27](9l2)coordinate[pos=0.34](9l3)coordinate[pos=0.41](9l4)coordinate[pos=0.48](9l5)($(d)+(60:1.7)$);
\draw(d)to[out=-25,in=-90,relative]coordinate[pos=0.2](9r1)coordinate[pos=0.27](9r2)coordinate[pos=0.34](9r3)coordinate[pos=0.41](9r4)coordinate[pos=0.48](9r5)($(d)+(60:1.7)$);
\draw[blue](11)--(41)--(21)--(9l1)--(9r1)--(31)to[out=-130,in=0](0,-1.4)to[out=180,in=-50](11);
\draw[blue](12)--(42)--(22)--(9l2)--(9r2)--(32)to[out=-120,in=0](0,-1.5)to[out=180,in=-60](12);
\draw[blue](13)--(43)--(23)--(9l3)--(9r3)--(33)to[out=-110,in=0](0,-1.6)to[out=180,in=-70](13);
\draw[blue](14)--(l)--(24)--(9l4)--(9r4)--(34)to[out=-100,in=0](0,-1.7)to[out=180,in=-80](14);
\draw[blue](15)--(l)--(25)--(9l5)--(9r5)--(35)to[out=-90,in=0](0,-1.8)to[out=180,in=-90](15);
\draw[blue](45)arc(-60:-290:0.4)to[out=-50,in=50](45);\draw[blue](46)arc(-60:-290:0.25)to[out=-50,in=50](46);
\draw[blue](44)arc(-60:-290:0.55)to[out=-50,in=50](44);
\fill(l)circle(0.07);\fill(r)circle(0.07);\fill(d)circle(0.07);\fill(u)circle(0.07);
\end{tikzpicture}\ ,
\hspace{3mm}
\psi_q^2\phi_p(S_V)=
\begin{tikzpicture}[baseline=-1mm]
\coordinate(d)at(0,-1.2);\coordinate(u)at(90:2);
\coordinate(l)at($(d)+(120:1.3)$);\coordinate(r)at($(d)+(60:1.3)$);
\draw(0,0)circle(2);
\draw(u)..controls(-2,1.7)and(-2,-1.5)..coordinate[pos=0.74](15)coordinate[pos=0.78](14)coordinate[pos=0.82](13)coordinate[pos=0.86](12)coordinate[pos=0.9](11)(d);
\draw(u)--coordinate[pos=0.65](25)coordinate[pos=0.7](24)coordinate[pos=0.75](23)coordinate[pos=0.8](22)coordinate[pos=0.85](21)(d);
\draw(u)..controls(2,1.7)and(2,-1.5)..coordinate[pos=0.74](35)coordinate[pos=0.78](34)coordinate[pos=0.82](33)coordinate[pos=0.86](32)coordinate[pos=0.9](31)(d);
\draw(d)--coordinate[pos=0.3](41)coordinate[pos=0.4](42)coordinate[pos=0.5](43)coordinate[pos=0.6](44)coordinate[pos=0.7](45)coordinate[pos=0.8](46)(l);
\draw(l)--coordinate[pos=0.1](51)coordinate[pos=0.2](52)coordinate[pos=0.3](53)(u);
\draw(d)to[out=25,in=90,relative]coordinate[pos=0.2](9l1)coordinate[pos=0.27](9l2)coordinate[pos=0.34](9l3)coordinate[pos=0.41](9l4)coordinate[pos=0.48](9l5)($(d)+(60:1.7)$);
\draw(d)to[out=-25,in=-90,relative]coordinate[pos=0.2](9r1)coordinate[pos=0.27](9r2)coordinate[pos=0.34](9r3)coordinate[pos=0.41](9r4)coordinate[pos=0.48](9r5)($(d)+(60:1.7)$);
\draw[blue](11)--(41)--(21)--(9l1)--(9r1)--(31)to[out=-130,in=0](0,-1.4)to[out=180,in=-50](11);
\draw[blue](12)--(42)--(22)--(9l2)--(9r2)--(32)to[out=-120,in=0](0,-1.5)to[out=180,in=-60](12);
\draw[blue](13)--(l)--(23)--(9l3)--(9r3)--(33)to[out=-110,in=0](0,-1.6)to[out=180,in=-70](13);
\draw[blue](14)--(l)--(24)--(9l4)--(9r4)--(34)to[out=-100,in=0](0,-1.7)to[out=180,in=-80](14);
\draw[blue](15)--(l)--(25)--(9l5)--(9r5)--(35)to[out=-90,in=0](0,-1.8)to[out=180,in=-90](15);
\draw[blue](45)arc(-60:-290:0.39)to[out=-50,in=50](45);\draw[blue](46)arc(-60:-290:0.25)to[out=-50,in=50](46);
\draw[blue](44)arc(-60:-290:0.53)to[out=-50,in=50](44);\draw[blue](43)arc(-60:-290:0.67)to[out=-50,in=50](43);
\draw[blue](l)to[out=20,in=170,relative](u);\draw[blue](l)to[out=-20,in=-175,relative](u);
\fill(l)circle(0.07);\fill(r)circle(0.07);\fill(d)circle(0.07);\fill(u)circle(0.07);
\end{tikzpicture}\ ,
\\
\psi_q^3\phi_p(S_V)=
\begin{tikzpicture}[baseline=-1mm]
\coordinate(d)at(0,-1.2);\coordinate(u)at(90:2);
\coordinate(l)at($(d)+(120:1.3)$);\coordinate(r)at($(d)+(60:1.3)$);
\draw(0,0)circle(2);
\draw(u)..controls(-2,1.7)and(-2,-1.5)..coordinate[pos=0.74](15)coordinate[pos=0.78](14)coordinate[pos=0.82](13)coordinate[pos=0.86](12)coordinate[pos=0.9](11)(d);
\draw(u)--coordinate[pos=0.65](25)coordinate[pos=0.7](24)coordinate[pos=0.75](23)coordinate[pos=0.8](22)coordinate[pos=0.85](21)(d);
\draw(u)..controls(2,1.7)and(2,-1.5)..coordinate[pos=0.74](35)coordinate[pos=0.78](34)coordinate[pos=0.82](33)coordinate[pos=0.86](32)coordinate[pos=0.9](31)(d);
\draw(d)--coordinate[pos=0.3](41)coordinate[pos=0.48](42)coordinate[pos=0.56](43)coordinate[pos=0.64](44)coordinate[pos=0.72](45)coordinate[pos=0.8](46)(l);
\draw(l)--coordinate[pos=0.1](51)coordinate[pos=0.2](52)coordinate[pos=0.3](53)(u);
\draw(d)to[out=25,in=90,relative]coordinate[pos=0.2](9l1)coordinate[pos=0.27](9l2)coordinate[pos=0.34](9l3)coordinate[pos=0.41](9l4)coordinate[pos=0.48](9l5)($(d)+(60:1.7)$);
\draw(d)to[out=-25,in=-90,relative]coordinate[pos=0.2](9r1)coordinate[pos=0.27](9r2)coordinate[pos=0.34](9r3)coordinate[pos=0.41](9r4)coordinate[pos=0.48](9r5)($(d)+(60:1.7)$);
\draw[blue](11)--(41)--(21)--(9l1)--(9r1)--(31)to[out=-130,in=0](0,-1.4)to[out=180,in=-50](11);
\draw[blue](12)--(l)--(22)--(9l2)--(9r2)--(32)to[out=-120,in=0](0,-1.5)to[out=180,in=-60](12);
\draw[blue](13)--(l)--(23)--(9l3)--(9r3)--(33)to[out=-110,in=0](0,-1.6)to[out=180,in=-70](13);
\draw[blue](14)--(l)--(24)--(9l4)--(9r4)--(34)to[out=-100,in=0](0,-1.7)to[out=180,in=-80](14);
\draw[blue](15)--(l)--(25)--(9l5)--(9r5)--(35)to[out=-90,in=0](0,-1.8)to[out=180,in=-90](15);
\draw[blue](45)arc(-60:-290:0.36)to[out=-50,in=50](45);\draw[blue](46)arc(-60:-290:0.25)to[out=-50,in=50](46);
\draw[blue](44)arc(-60:-290:0.47)to[out=-50,in=50](44);\draw[blue](43)arc(-60:-290:0.58)to[out=-50,in=50](43);
\draw[blue](42)arc(-60:-290:0.69)to[out=-50,in=50](42);
\draw[blue](l)to[out=20,in=170,relative](u);\draw[blue](l)to[out=-20,in=-175,relative](u);
\draw[blue](l)to[out=40,in=160,relative](u);\draw[blue](l)to[out=-40,in=-170,relative](u);
\fill(l)circle(0.07);\fill(r)circle(0.07);\fill(d)circle(0.07);\fill(u)circle(0.07);
\end{tikzpicture}\ ,
\hspace{3mm}
\psi_q^4\phi_p(S_V)=
\begin{tikzpicture}[baseline=-1mm]
\coordinate(d)at(0,-1.2);\coordinate(u)at(90:2);
\coordinate(l)at($(d)+(120:1.3)$);\coordinate(r)at($(d)+(60:1.3)$);
\draw(0,0)circle(2);
\draw(u)..controls(-2,1.7)and(-2,-1.5)..coordinate[pos=0.74](15)coordinate[pos=0.78](14)coordinate[pos=0.82](13)coordinate[pos=0.86](12)coordinate[pos=0.9](11)(d);
\draw(u)--coordinate[pos=0.65](25)coordinate[pos=0.7](24)coordinate[pos=0.75](23)coordinate[pos=0.8](22)coordinate[pos=0.85](21)(d);
\draw(u)..controls(2,1.7)and(2,-1.5)..coordinate[pos=0.74](35)coordinate[pos=0.78](34)coordinate[pos=0.82](33)coordinate[pos=0.86](32)coordinate[pos=0.9](31)(d);
\draw(d)--coordinate[pos=0.5](41)coordinate[pos=0.56](42)coordinate[pos=0.62](43)coordinate[pos=0.68](44)coordinate[pos=0.74](45)coordinate[pos=0.8](46)(l);
\draw(l)--coordinate[pos=0.1](51)coordinate[pos=0.2](52)coordinate[pos=0.3](53)(u);
\draw(d)to[out=25,in=90,relative]coordinate[pos=0.2](9l1)coordinate[pos=0.27](9l2)coordinate[pos=0.34](9l3)coordinate[pos=0.41](9l4)coordinate[pos=0.48](9l5)($(d)+(60:1.7)$);
\draw(d)to[out=-25,in=-90,relative]coordinate[pos=0.2](9r1)coordinate[pos=0.27](9r2)coordinate[pos=0.34](9r3)coordinate[pos=0.41](9r4)coordinate[pos=0.48](9r5)($(d)+(60:1.7)$);
\draw[blue](11)--(l)--(21)--(9l1)--(9r1)--(31)to[out=-130,in=0](0,-1.4)to[out=180,in=-50](11);
\draw[blue](12)--(l)--(22)--(9l2)--(9r2)--(32)to[out=-120,in=0](0,-1.5)to[out=180,in=-60](12);
\draw[blue](13)--(l)--(23)--(9l3)--(9r3)--(33)to[out=-110,in=0](0,-1.6)to[out=180,in=-70](13);
\draw[blue](14)--(l)--(24)--(9l4)--(9r4)--(34)to[out=-100,in=0](0,-1.7)to[out=180,in=-80](14);
\draw[blue](15)--(l)--(25)--(9l5)--(9r5)--(35)to[out=-90,in=0](0,-1.8)to[out=180,in=-90](15);
\draw[blue](45)arc(-60:-290:0.34)to[out=-50,in=50](45);\draw[blue](46)arc(-60:-290:0.25)to[out=-50,in=50](46);
\draw[blue](44)arc(-60:-290:0.43)to[out=-50,in=50](44);\draw[blue](43)arc(-60:-290:0.52)to[out=-50,in=50](43);
\draw[blue](42)arc(-60:-290:0.61)to[out=-50,in=50](42);\draw[blue](41)arc(-60:-290:0.7)to[out=-50,in=50](41);
\draw[blue](l)to[out=20,in=170,relative](u);\draw[blue](l)to[out=-20,in=-175,relative](u);
\draw[blue](l)to[out=40,in=160,relative](u);\draw[blue](l)to[out=-40,in=-170,relative](u);
\draw[blue](l)to[out=60,in=150,relative](u);\draw[blue](l)to[out=-60,in=-165,relative](u);
\fill(l)circle(0.07);\fill(r)circle(0.07);\fill(d)circle(0.07);\fill(u)circle(0.07);
\end{tikzpicture}\ .
\end{gather*}
Then $\psi_q^k\phi_p(S_V)$ never satisfy (C2) for $1\le k\le 4$ since $d_q^{\psi_q^k\phi_p(S_V)}\le -1$ by the above figures. Assume that $\phi_p(S_V)$ is characterized by $(n_{\gamma})_{\gamma\in E_V}=(n_4,n_9)$. Since $n_9$ must be zero by (C1), both $d_p^{\phi_p(S_V)}=4$ and $d_q^{\phi_p(S_V)}=2$ must be equal to $n_4$ by (C2), a contradiction. Therefore, $S_V$ is a unique $p$-recoverable multi-set such that $\phi_p(S_V)$ is its $p$-modification. Here, it is not difficult to check that $S_V$ is $p$-recoverable (see Proposition \ref{prop:properties}).
\end{example}

\subsection{Proof of Theorem \ref{thm:segment}}\label{subsec:proof}

Let $U\in\bM_{\cS}$ with $U\cap T=\emptyset$. We freely use the notations in the previous subsections and sections. In particular, we know that $\widehat{S_U}=U^{\circ}$. We list some facts that immediately follow from the definitions of notations:
{\setlength{\leftmargini}{20mm}
\begin{itemize}
\item[(Fact $1$)] For $s,t\in S_U$, if $\Int(s,t)>0$, then either $s$ or $t$ is $h_a$ for an angle $a$ of $\Omega(T)$ at a puncture $p$ in $V_T$ with $n(\Omega(U),p)>0$.
\item[(Fact $2$)] For $\gamma\in T$, if $m_{U_0}(\gamma)>0$, then there is $\delta\in U_1$ such that $\{\gamma,\delta\}$ is a pair of conjugate arcs and $m_{U_1}(\delta)=m_{U_0}(\gamma)$. In particular, if $\gamma$ is a plain arc, then $\delta$ is a $1$-notched arc.
\item[(Fact $3$)] The multi-set $S_U$ is glueable.
\end{itemize}}

First, we consider about the maximal $a$-modification $\Phi(S_U)$ of $S_U$.

\begin{prop}\label{prop:compatibility}
The multi-set $S_U$ is $a$-modifiable and $\Phi(S_U)$ consists of pairwise compatible segments.
\end{prop}

\begin{proof}
It follows from (Fact 1) and (Fact 2) that $S_U$ satisfies $(\ast,a)$ for all angles $a$ of $\Omega(T)$ where $S_U$ satisfies $(\phi,a)$. Then it is $a$-modifiable by (Fact 1). Moreover, it satisfies the conditions in Theorem \ref{thm:comp S} since $n(\Omega(U),p)=m_{S_U}(c_p)$ for all $p\in V_T$. Thus $\Phi(S_U)$ consists of pairwise compatible segments by Theorem \ref{thm:comp S}.
\end{proof}

Second, we focus on intersection numbers with elements of $\Omega(T)$.

\begin{prop}\label{prop:Int=Int}
Let $\triangle$ be a puzzle piece of $\Omega(T)$ with $\gamma\in T\cup\Omega(T)$ as an edge. Then
\[
\Int(\gamma,\Phi(S_U)\cap\triangle)=\Int(\gamma,U).
\]
\end{prop}

\begin{proof}
If $\gamma\in\Omega(T)$, then
\begin{align*}
\Int(\gamma,\Phi(S_U)\cap\triangle)
&=\Int(\gamma,S_U\cap\triangle)-m_{S_U\cap\triangle}(\gamma)&&\text{(by Lemmas \ref{lem:Int tri} and \ref{lem:Int mono})}\\
&=\Int(\gamma,U^{\circ})-m_{U^{\circ}}(\gamma)&&\text{(by (Fact 3))}\\
&=\Int(\gamma,\Omega(U))&&\text{(by Definition \ref{def:Int} and (Fact 2))}\\
&=\Int(\gamma,U)&&\text{(by Proposition \ref{prop:Int conjugate})}.
\end{align*}
If $\gamma\in T\setminus\Omega(T)$, then there is $\gamma'\in T\setminus\Omega(T)$ such that $\{\gamma,\gamma'\}$ is a pair of conjugate arcs. Without loss of generality, we can assume that $\gamma=f_a$ and $\gamma'=f_a^{\bowtie}$ for some angle $a$ of a monogon piece $\triangle$. If $\gamma$ is incident to the boundary of $\cS$, then the assertion holds since $\Phi(S_U)\cap\triangle=S_U\cap\triangle$. Assume that $\gamma$ connects punctures. Then $\{\gamma,\rho(\gamma')\}$ and $\{\gamma',\rho(\gamma)\}$ are pairs of conjugate arcs. Therefore,
\begin{align*}
\Int(\gamma,\Phi(S_U)\cap\triangle)
&=\Int(\gamma,S_U)-m_{S_U}(f_a)+m_{S_U}(f_a^{\bowtie})&&\text{(by Lemma \ref{lem:Int mono})}\\
&=\Int(\gamma,U^{\circ})-m_{U_0}(\gamma)+m_{U_0}(\gamma')&&\text{}\\
&=\Int(\gamma,U^{\circ})-m_{U_1}(\rho(\gamma'))+m_{U_1}(\rho(\gamma))&&\text{(by (Fact 2))}\\
&=\Int(\gamma,\Omega(U))&&\text{(by Definition \ref{def:Int})}\\
&=\Int(\gamma,U)&&\text{(by Proposition \ref{prop:Int conjugate})}.\qedhere
\end{align*}
\end{proof}

Third, we show that $S_U$ is $p$-recoverable.

\begin{prop}\label{prop:properties}
The multi-set $S_U$ satisfies the following properties:
\begin{itemize}
\item[(1)] It satisfies $(\ast,p)$ for all punctures $p$.
\item[(2)] It has no enclosed punctures.
\item[(3)] It is characterized by $(m_{U_2}(\rho(\gamma)))_{\gamma\in E_T}$.
\end{itemize}
Therefore, it is $p$-recoverable.
\end{prop}

\begin{proof}
(1) For a puncture $p\notin V_T$, $S_U$ clearly satisfies $(\ast,p)$ since $A_p=\emptyset$. Let $p\in V_T$. Since $m_{S_U}(h_a)=n(\Omega(T),p)=m_{S_U}(c_p)$ for $a\in A_p^{\phi}$ by the compatibility of $U_0$, $S_U$ satisfies $(\ast 1,p)$. It follows from (Fact 1) and (Fact 2) that $S_U$ satisfies $(\ast 2,p)$.

(2) Assume that there is a puncture $p$ enclosed in $S_U$. When we see $U^{\circ}=\widehat{S_U}$ as a multi-set of tagged branched curves in $\cS$, it has no vertices with degree three. By Lemma \ref{lem:phi Int}, $p$ must be enclosed by a punctured loop $\delta\in U_0$ with endpoint $q$ such that the following hold (see the figures above Lemma \ref{lem:enclosed}):
\begin{itemize}
\item[(a)] There is a plain arc $\gamma$ in $T$ enclosed by $\delta$.
\item[(b)] If $q\in V_T\cup(\cM\cap\partial\cS)$, then tags of $\delta$ are plain.
\end{itemize}
Let $\{\delta',\delta''\}$ be a pair of conjugate arcs in $U$ such that $\Omega(\{\delta',\delta''\})=\delta$. If $q\in V_T\cup(\cM\cap\partial\cS)$, then either $\delta'$ or $\delta''$ is a plain arc by (b), and it is just $\gamma$ in (a). It contradicts $U\cap T=\emptyset$. Assume that $q\notin V_T\cup(\cM\cap\partial\cS)$. Then there is a pair $\{\gamma,\gamma'\}$ of conjugate arcs in $T$. It is easy to see that $\{\gamma,\gamma'\}\cap\{\delta',\delta''\}\neq\emptyset$. It contradicts $U\cap T=\emptyset$ again.

(3) Let $\gamma\in E_T^{pq}$ with $m_{U_2}(\rho(\gamma))>0$ for $p,q\in V_T$. Then
\[
m_{S_U}(c_p)+m_{S_U}(c_q)=n(\Omega(U),p)+n(\Omega(U),q)=\Int(\gamma,U^{\circ})=\Int(\gamma,\widehat{S_U}),
\]
where the second equality follows from $\Int(\gamma,U_0)=\Int(\rho(\gamma),U_1)=0$. Thus (C1) holds. On the other hand, let $p\in V_T$. If $c_p\notin U^{\circ}=\widehat{S_U}$, then $n(\Omega(U),p)=0$ and $\Int'(c_p,\widehat{S_U})=0$. In particular, there are no $\gamma\in E_T$ incident to $p$ such that $m_{U_2}(\rho(\gamma))>0$. Thus (C2) holds. If $c_p\in U^{\circ}=\widehat{S_U}$, then
\begin{align*}
d_p^{S_U}&=m_{S_U}(c_p)-\Int'(c_p,\widehat{S_U})\\
&=n(\Omega(U),p)-n(U_0,p)\\
&=n(U_2,p)\\
&=\sum_{q\in V_T\setminus\{p\}}\sum_{\gamma\in E_T^{pq}}m_{U_2}(\rho(\gamma))+2\sum_{\gamma\in E_T^{pp}}m_{U_2}(\rho(\gamma)).
\end{align*}
Therefore, (C2) holds.

Finally, (1), (2), (3), and (Fact 3) mean that $S_U$ is $p$-recoverable.
\end{proof}

We are ready to prove Theorem \ref{thm:segment}.

\begin{proof}[Proof of Theorem \ref{thm:segment}]
The first assertion follows from Propositions \ref{prop:compatibility} and \ref{prop:Int=Int}. By Propositions \ref{prop:modif=max p}, \ref{prop:properties}(1), and (Fact 1), $\Phi(S_U)$ is also the maximal $p$-modification of $S_U$. Therefore, the second assertion follows from Theorem \ref{thm:unique S} and Proposition \ref{prop:properties}.
\end{proof}

\subsection{On closed surfaces with exactly one puncture}\label{subsec:1-punc}

In this subsection, we assume that $\cS$ is a closed surface with exactly one puncture $p$. The observation here will be used in Subsection \ref{subsec:JT}. We first remark that any tagged triangulation of $\cS$ decomposes $\cS$ into only triangle pieces, and $\Omega(U)=U$ for $U\in\bM_{\cS}$.

Let $T$ be a tagged triangulation of $\cS$ satisfying \eqref{diamond} and $n\in\bZ_{>0}$. For tagged arcs $\gamma$ and $\delta$ in $\cS$, their \emph{n-intersection number $\Int^n(\gamma,\delta)$} is defined by $A_{\gamma,\delta}+nB_{\gamma,\delta}$, where $A_{\gamma,\delta}$ and $B_{\gamma,\delta}$ are defined in Definition \ref{def:Int}.

\begin{remark}
In a general marked surface $\cS'$, the $n$-intersection number of $\gamma,\delta\in\bAL_{\cS'}$ can be defined by $A_{\gamma,\delta}+n(B_{\gamma,\delta}+C_{\gamma,\delta})$ (cf. Definition \ref{def:Int}). In this paper, we only need it for a closed surface with exactly one puncture. In which case, since a pair of conjugate arcs does not appear, we can omit $C_{\gamma,\delta}$.
\end{remark}

For $U\in\bM_{\cS}$, the \emph{n-intersection vector of $U$ with respect to $T$} is the non-negative vector
\[
\Int^n_T(U):=\left(\Int^n(t,U)\right)_{t\in T}:=\Biggl(\sum_{u\in U}\Int^n(t,u)\Biggr)_{t\in T}\in\bZ_{\ge 0}^T.
\]
We also denote $\Int^n_T(\{\gamma\})$ by $\Int^n_T(\gamma)$. In particular, for a tagged arc $\gamma$ in $\cS$ whose tags are different from ones in $T$, its $n$-intersection vector with respect to $T$ is given by
\begin{equation}\label{eq:Int rho}
\Int^n_T(\gamma)=(A_{t,\gamma}+nB_{t,\gamma})_{t\in T}=(A_{t,\rho(\gamma)}+4n)_{t\in T}=\Int^n_T(\rho(\gamma))+(4n,\ldots,4n).
\end{equation}

For $U\in\bM_{\cS}$ with $U\cap T=\emptyset$, we consider multi-sets $U_0$, $U_1$, $U_2$, and $n(U,p)$ as in Subsection \ref{subsec:pf thm:main}. Moreover, we consider the following analogues of $U^{\circ}$ and $S_U$:
\[
U^{\circ,n}:=U_0\sqcup\{c_p^{n\cdot n(U,p)}\}
\ \text{ and }\ S_{U,n}:=\bigsqcup_{\triangle}(U^{\circ,n}\cap\triangle),
\]
where $\triangle$ runs over all triangle pieces of $T$. In particular, $S_{U,n}$ is a glueable multi-set of segments and
\begin{equation}\label{eq:Int n}
\Int_T(\widehat{S_{U,n}})=\Int_T(U^{\circ,n})=\Int^n_T(U).
\end{equation}

As in previous subsections, for a glueable multi-set $S$ of segments in puzzle pieces of $T$, we consider the difference
\[
d_{p,n}^S:=m_S(c_p)-n\cdot\Int'(c_p,\widehat{S}).
\]
Moreover, we also consider the following two conditions of $S$: (a) $d_{p,n}^S=0$. (b) $S$ does not satisfy $(\psi,p)$. The multi-set $S_{U,n}$ satisfies (a) or (b). In fact, if $U_2=\emptyset$, then it satisfies (a); If $U_2\neq\emptyset$, then there is $\gamma\in T$ such that $\rho(\gamma)\in U_2$ and $\Int(\gamma,\widehat{S_{U,n}})=2m_{S_{U,n}}(c_p)$. This means that $m_{S_{U,n}}(h_{a^l})=m_{S_{U,n}}(h_{a^r})=m_{S_{U,n}}(c_p)$ for a triangle piece of $T$ with an angle $a$ such that $e_a=\gamma$. Since $\psi_{a^l}(S_{U,n})$ (resp., $\psi_{a^r}(S_{U,n})$) does not satisfy $(\psi,a^r)$ (resp., $(\psi,a^l)$), $S_{U,n}$ satisfies (b).

\begin{thm}\label{thm:unique Intn}
Assume that $\cS$ is a closed surface with exactly one puncture $p$. Fix $n\in\bZ_{>0}$. Let $T$ be a tagged triangulation of $\cS$, and $U,V\in\bM_{\cS}$ with $U\cap T=V\cap T=\emptyset$. If $\Int^n_T(U)=\Int^n_T(V)$, then $U_1=V_1$ and $n(U_2,p)=n(V_2,p)$.
\end{thm}

\begin{proof}
Since $S_{U,n}$ is $a$-modifiable and satisfies $(\ast,p)$ like $S_U$, Proposition \ref{prop:modif=max p} implies that there is $m\in\bZ_{\ge 0}$ such that $\Phi(S_{U,n})=\phi_p^m(S_{U,n})$. By Proposition \ref{prop:recover p} and Lemma \ref{lem:phi Int}, $\phi_p^k(S_{U,n})$ satisfies neither (a) nor (b) for $1\le k\le m$. We take $m'\in\bZ_{\ge 0}$ such that $\psi_p^{k'}\Phi(S_{U,n})$ satisfies neither (a) nor (b) for $1\le k'<m'$, and $\psi_p^{m'}\Phi(S_{U,n})$ satisfies (a) or (b). Then it is clear that $m=m'$. Thus $S_{U,n}=\psi_p^{m'}\Phi(S_{U,n})$ is recoverable from $\Phi(S_{U,n})$.  

On the other hand, $\Phi(S_{U,n})$ consists of pairwise compatible segments by Theorem \ref{thm:comp S}. Moreover, Lemma \ref{lem:Int tri} and \eqref{eq:Int n} induce
\[
\Int_T(\widehat{\Phi(S_{U,n})})=\Int_T(\widehat{S_{U,n}})=\Int^n_T(U).
\]
Then it follows from Proposition \ref{prop:Int triangle} that $\Phi(S_{U,n})$ is uniquely determined by $\Int^n_T(U)$. Therefore, $S_{U,n}$ is uniquely determined by $\Int^n_T(U)$. The assertion is given in the same way as the proof of Theorem \ref{thm:main} in Subsection \ref{subsec:pf thm:main}.
\end{proof}

\begin{cor}\label{cor:unique tag tri}
Assume that $\cS$ is a closed surface with exactly one puncture. Fix $n\in\bZ_{>0}$. Let $T$, $U$, and $V$ be tagged triangulations of $\cS$. If $\Int_T^n(U)=\Int_T^n(V)$, then $U=V$.
\end{cor}

\begin{proof}
The assertion follows from Theorem \ref{thm:unique Intn} and the proof of Theorem \ref{thm:unique tag tri} in Subsection \ref{subsec:pf thm:main}.
\end{proof}

\section{Cluster algebra theory}\label{sec:cluster}

\subsection{Cluster algebras}

We recall (skew-symmetric) cluster algebras with principal coefficients \cite{FZ02,FZ07}. For that, we need to prepare some notations. For a quiver $Q$, we denote by $Q_0$ the set of its vertices and by $Q_1$ the set of its arrows. An oriented cycle of length two is called a \emph{$2$-cycle}. Let $n \in \bZ_{> 0}$ and $\cF:=\bQ(t_1,\ldots,t_{2n})$ be the field of rational functions in $2n$ variables over $\bQ$. A \emph{seed with coefficients} is a pair $(\cl,Q)$ consisting of the following data:
\begin{enumerate}
 \item $\cl=(x_1,\ldots,x_n,y_1,\ldots,y_n)$ is a free generating set of $\cF$ over $\bQ$.
 \item $Q$ is a quiver without loops nor $2$-cycles such that $Q_0=\{1,\ldots,2n\}$.
\end{enumerate}
 Then we refer to the tuple $(x_1,\ldots,x_n)$ as the \emph{cluster}, to each $x_i$ as a \emph{cluster variable} and to $y_i$ as a \emph{coefficient}. For a seed $(\cl,Q)$ with coefficients and $k\in\{1,\ldots,n\}$, the \emph{mutation $\mu_k(\cl,Q)=(\cl',\mu_kQ)$ at $k$} is defined as follows:
\begin{enumerate}
 \item $\cl'=(x'_1,\ldots,x'_n,y_1,\ldots,y_n)$ is defined by $x'_i = x_i$ for $i\neq k$, and
 \[
  x_k x'_k = \prod_{(j \rightarrow k)\in Q_1}x_j\prod_{(j \rightarrow k)\in Q_1}y_{j-n}+\prod_{(j \leftarrow k)\in Q_1}x_j\prod_{(j \leftarrow k)\in Q_1}y_{j-n},
 \]
where $x_{n+1}=\cdots=x_{2n}=1=y_{1-n}=\cdots=y_0$.
 \item $\mu_kQ$ is the quiver obtained from $Q$ by the following steps:\par
 \begin{enumerate}
  \item For any path $i \rightarrow k \rightarrow j$, add an arrow $i \rightarrow j$.
  \item Reverse all arrows incident to $k$.
  \item Remove a maximal set of disjoint $2$-cycles.
 \end{enumerate}
\end{enumerate}

Note that $\mu_k$ is an involution, that is, $\mu_k\mu_k(\cl,Q)=(\cl,Q)$. Moreover, it is easy to see that $\mu_k(\cl,Q)$ is also a seed with coefficients.

For a quiver $Q$ without loops nor $2$-cycles such that $Q_0=\{1,\ldots,n\}$, we obtain the quiver $\hat{Q}$ from $Q$ by adding vertices $\{1',\ldots,n'\}$ and arrows $\{i \rightarrow i' \mid 1 \le i \le n\}$. We fix a seed $(\cl=(x_1,\ldots,x_n,y_1,\ldots,y_n),\hat{Q})$ with coefficients, called the \emph{initial seed}. We also call the tuple $(x_1,\ldots,x_n)$ the \emph{initial cluster}, and each $x_i$ the \emph{initial cluster variable}.

\begin{defn}
The \emph{cluster algebra $\cA(Q)=\cA(\cl,\hat{Q})$ with principal coefficients} for the initial seed $(\cl,\hat{Q})$ is a $\bZ$-subalgebra of $\cF$ generated by all cluster variables and coefficients obtained from $(\cl,\hat{Q})$ by sequences of mutations.
\end{defn}

One of the remarkable properties of cluster algebras with principal coefficients is the strongly Laurent phenomenon as follows.

\begin{prop}[{\cite[Proposition 3.6]{FZ07}}]\label{prop:Laurent phenomenon}
Every nonzero element $x$ of $\cA(Q)$ is expressed by a Laurent polynomial of $x_1,\ldots,x_n$, $y_1,\ldots,y_n$
\[
x=\frac{F(x_1, \ldots,x_n,y_1,\ldots,y_n)}{x_1^{d_1} \cdots x_n^{d_n}},
\]
where $d_i\in\bZ$ and $F(x_1, \ldots, x_n,y_1,\ldots,y_n)\in\bZ[x_1,\ldots, x_n,y_1,\ldots,y_n]$ is not divisible by any $x_i$.
\end{prop}

\begin{defn}
Keeping the notations in Proposition \ref{prop:Laurent phenomenon}, we call $d(x):=(d_i)_{1 \le i \le n}$ the \emph{denominator vector} of $x$, and $f(x):=(f_i)_{1 \le i \le n}$ the \emph{$f$-vector} of $x$, where $f_i$ is the maximal degree of $y_i$ in the \emph{$F$-polynomial} $F(1,\ldots,1,y_1,\ldots,y_n)$ of $x$.
\end{defn}

\begin{defn}
A \emph{cluster monomial} is a monomial in cluster variables belonging to the same cluster. It is called \emph{non-initial} if it is a monomial in non-initial cluster variables.
\end{defn}

Fomin and Zelevinsky \cite{FZ03} conjectured that different cluster monomials have different denominator vectors (Conjecture \ref{conj:denominator}). Recently, Fei \cite{F} gave its counterexample as follows.

\begin{example}[\cite{F}]\label{ex:counterexample}
We consider the cluster algebra $\cA(Q)$ associated with a quiver
\[
Q=
\begin{tikzpicture}[baseline=-6mm]
\node(1)at(0,0){$1$};\node(2)at($(1)+(-90:1)$){$2$};\node(3)at($(2)+(0:1)$){$3$};\node(4)at($(3)+(90:1)$){$4$};
\draw[transform canvas={xshift=2},->](1)--(2);\draw[transform canvas={xshift=-2},->](1)--(2);
\draw[transform canvas={xshift=2},->](3)--(4);\draw[transform canvas={xshift=-2},->](3)--(4);
\draw[->](2)--(3);\draw[->](4)--(1);
\end{tikzpicture}\ .
\]
Applying sequences $\mu_2\mu_1\mu_4\mu_2\mu_3\mu_1\mu_4\mu_2$ and $\mu_4\mu_3\mu_4\mu_2\mu_3\mu_1\mu_4\mu_2$ of mutations to the initial seed, we obtain different cluster variables $x$ and $y$ with the same denominator vector $(4,6,4,6)$. In fact, their $g$-vectors are $(6,-5,3,-2)$ and $(3,-2,6,-5)$, respectively (see \cite[Section 6]{FZ07} for the notion of $g$-vectors). This gives a counterexample of Conjecture \ref{conj:denominator}.
\end{example}

For all initial cluster variables $x_i$, it is clear that
\begin{equation}\label{eq:initial d f}
d(x_i)=-\be_i\ \text{ and }\ f(x_i)=0,
\end{equation}
where $\be_1,\ldots,\be_n$ are the standard basis vectors in $\bZ^n$. The following theorem was conjectured in \cite[Conjecture 7.4(2)]{FZ07}.

\begin{thm}[{\cite[Theorem 11(ii)]{CL20}}]\label{thm:denominator}
Let $x$ be a non-initial cluster variable in $\cA(Q)$ with $d(x)=(d_1,\ldots,d_n)\in\bZ^n$. Then every $d_i$ is non-negative. Moreover, if there is a cluster containing $x$ and an initial cluster variable $x_k$ for $1\le k\le n$, then $d_k=0$.
\end{thm}

By Theorem \ref{thm:denominator}, we only need to consider non-initial cluster monomials to see if Conjecture \ref{conj:denominator} holds for a given cluster algebra as follows.

\begin{lem}\label{lem:conj non-initial}
The following are equivalent:
\begin{itemize}
\item[(1)] For any cluster monomials $x$ and $x'$ in $\cA(Q)$, if $d(x)=d(x')$, then $x=x'$ (Conjecture \ref{conj:denominator}).
\item[(2)] For any non-initial cluster monomials $x$ and $x'$ in $\cA(Q)$, if $d(x)=d(x')$, then $x=x'$.
\end{itemize}
\end{lem}

\begin{proof}
It is clear that $\cA(Q)$ satisfies (2) if it satisfies (1). Assume that $\cA(Q)$ satisfies (2). For a cluster monomial $x$ in $\cA(Q)$, there is a decomposition $x=x_Ix_N$, where $x_I$ (resp., $x_N$) is the maximal sub-monomial in initial (resp., non-initial) cluster variables. For $d(x)=(d_1,\ldots,d_n)$, \eqref{eq:initial d f} and Theorem \ref{thm:denominator} induce that $d(x_I)=-([-d_1]_+,\ldots,[-d_n]_+)$ and $d(x_N)=([d_1]_+,\ldots,[d_n]_+)$, where $[d]_+:=\max\{d,0\}$. Therefore, if $d(x)=d(x')$ for cluster monomials $x$ and $x'$ in $\cA(Q)$, then $d(x_I)=d(x'_I)$ and $d(x_N)=d(x'_N)$. It follows from \eqref{eq:initial d f} and (2) that $x_I=x'_I$ and $x_N=x'_N$, respectively. This means that $\cA(Q)$ satisfies (1).
\end{proof}

\subsection{Cluster algebras associated with triangulated surfaces}\label{subsec:cA T}

To a tagged triangulation $T$ of $\cS$, we associate a quiver $\bar{Q}_T$ with $(\bar{Q}_T)_0=T$ whose arrows correspond to angles between tagged arcs in $T$ as in Table \ref{table:Qtri}. We obtain a quiver $Q_T$ without loops nor $2$-cycles from $\bar{Q}_T$ by removing $2$-cycles. This construction commutes with flips and mutations, that is, $Q_{\mu_{\gamma}T}=\mu_{\gamma}Q_T$ for $\gamma\in T$ \cite[Proposition 4.8 and Lemma 9.7]{FoST08}.

%
%
\renewcommand{\arraystretch}{2}
{\begin{table}[ht]
\begin{tabular}{c|c|c|c|c}
$\triangle$
&
\begin{tikzpicture}[baseline=0mm]
 \coordinate(l)at(-150:1); \coordinate(r)at(-30:1); \coordinate(u)at(90:1);
 \draw(u)--node[left]{$1$}(l)--node[below]{$3$}(r)--node[right]{$2$}(u);
 \fill(u)circle(0.07); \fill(l)circle(0.07); \fill(r)circle(0.07);
\end{tikzpicture}
&
\begin{tikzpicture}[baseline=-2mm]
 \coordinate(c)at(0,0); \coordinate(u)at(0,1); \coordinate(d)at(0,-1);
 \draw(d)to[out=180,in=180]node[left]{$1$}(u);
 \draw(d)to[out=0,in=0]node[right]{$2$}(u);
 \draw(d)to[out=150,in=-150]node[fill=white,inner sep=1]{$3$}(c);
 \draw(d)to[out=30,in=-30]node[pos=0.8]{\rotatebox{40}{\footnotesize $\bowtie$}}node[fill=white,inner sep=1]{$4$}(c);
 \fill(c)circle(0.07); \fill(u)circle(0.07); \fill(d)circle(0.07);
\end{tikzpicture}
&
\begin{tikzpicture}[baseline=-2mm]
 \coordinate(l)at(-0.5,0); \coordinate(r)at(0.5,0); \coordinate(d)at(0,-1);
 \draw(0,0)circle(1); \node at(0,0.7){$1$};
 \draw(d)to[out=170,in=-130]node[fill=white,inner sep=1]{$2$}(l);
 \draw(d)to[out=95,in=0]node[left,inner sep=1,pos=0.4]{$3$}node[pos=0.8]{\rotatebox{60}{\footnotesize $\bowtie$}}(l);
 \draw(d)to[out=85,in=180]node[right,inner sep=1,pos=0.4]{$4$}(r);
 \draw(d)to[out=10,in=-50] node[pos=0.8]{\rotatebox{10}{\footnotesize $\bowtie$}}node[fill=white,inner sep=1]{$5$}(r);
 \fill(l)circle(0.07); \fill(r)circle(0.07); \fill(d)circle(0.07);
\end{tikzpicture}
&
\begin{tikzpicture}[baseline=1mm]
 \coordinate(c)at(0,0); \coordinate(u)at(90:1); \coordinate(r)at(-30:1); \coordinate(l)at(210:1);
 \draw(c)to[out=60,in=120,relative]node[fill=white,inner sep=1]{$1$}(u);
 \draw(c)to[out=-60,in=-120,relative]node[pos=0.8]{\rotatebox{40}{\footnotesize $\bowtie$}}node[fill=white,inner sep=1]{$2$}(u);
 \draw(c)to[out=60,in=120,relative]node[fill=white,inner sep=1]{$3$}(l);
 \draw(c)to[out=-60,in=-120,relative] node[pos=0.8]{\rotatebox{160}{\footnotesize $\bowtie$}}node[fill=white,inner sep=1]{$4$}(l);
 \draw(c)to[out=60,in=120,relative]node[fill=white,inner sep=1]{$5$}(r);
 \draw(c)to[out=-60,in=-120,relative]node[pos=0.8]{\rotatebox{-80}{\footnotesize $\bowtie$}}node[fill=white,inner sep=1]{$6$}(r);
 \fill(c)circle(0.07); \fill(u)circle(0.07); \fill(l)circle(0.07); \fill(r)circle(0.07);
\end{tikzpicture}
\\\hline
$Q_{\triangle}$
&
\begin{tikzpicture}[baseline=-3mm,scale=0.8]
 \node(1)at(150:1){$1$}; \node(2)at(30:1){$2$}; \node(3)at(-90:1){$3$};
 \draw[->](3)--(2); \draw[->](2)--(1); \draw[->](1)--(3);
\end{tikzpicture}
&
\begin{tikzpicture}[baseline=-2mm]
 \node(1)at(-0.7,0.5){$1$}; \node(2)at(0.7,0.5){$2$}; \node(3)at(0,-0.2){$3$}; \node(4)at(0,-0.8){$4$};
 \draw[->](3)--(2); \draw[->](2)--(1); \draw[->](1)--(3); \draw[->](4)--(2); \draw[->](1)--(4);
\end{tikzpicture}
&
\begin{tikzpicture}[baseline=-1mm]
 \node(1)at(0,0.8){$1$}; \node(2)at(-170:1){$2$}; \node(3)at(-130:1){$3$};
 \node(4)at(-10:1){$4$}; \node(5)at(-50:1){$5$};
 \draw[->] (1)--(2); \draw[->] (2)--(4); \draw[->] (4)--(1);
 \draw[->] (1)--(3); \draw[->] (3)--(4);
 \draw[->] (2)--(5); \draw[->] (5)--(1); \draw[->] (3)--(5);
\end{tikzpicture}
&
\begin{tikzpicture}[baseline=0mm]
 \node(1)at(90:0.5){$1$}; \node(2)at(90:1.1){$2$}; \node(3)at(-30:0.5){$5$};
 \node(4)at(-30:1.1){$6$}; \node(5)at(210:0.5){$3$}; \node(6)at(210:1.1){$4$};
 \draw[->](1)--(3); \draw[->](3)--(5); \draw[->](5)--(1);
 \draw[->] (2) to [out=40,in=140,relative] (4); \draw[->] (4) to [out=40,in=140,relative] (6); \draw[->] (6) to [out=40,in=140,relative] (2);
 \draw[->] (2) to [out=40,in=140,relative] (3); \draw[->] (4) to [out=40,in=140,relative] (5); \draw[->] (6) to [out=40,in=140,relative] (1);
 \draw[->] (1) to [out=40,in=140,relative] (4); \draw[->] (3) to [out=40,in=140,relative] (6); \draw[->] (5) to [out=40,in=140,relative] (2);
\end{tikzpicture}
\end{tabular}\vspace{3mm}
\caption{The quiver $Q_{\triangle}$ associated with each puzzle piece $\triangle$ of tagged triangulations}
\label{table:Qtri}
\end{table}}

For the associated cluster algebra $\cA(Q_T)$, the following are due to \cite[Theorem 7.11]{FoST08}, \cite[Theorem 6.1]{FT18}, and \cite[Theorem 7]{Y19}.

\begin{thm}[{\cite{FoST08,FT18,Y19}}]\label{thm:bijection x}
Let $T$ be a tagged triangulation of $\cS$.
\begin{itemize}
\item[(1)] If $\cS$ is not a closed surface with exactly one puncture, then there is a bijection
\[
x_T:\{\text{Tagged arcs in $\cS$}\}\rightarrow\{\text{Cluster variables in $\cA(Q_T)$}\}
\]
such that $\Int_T(\gamma)=f(x_T(\gamma))$ for all tagged arcs $\gamma$ in $\cS$. Moreover, it induces a bijection
\[
x_T:\{\text{Tagged triangulations in $\cS$}\}\rightarrow\{\text{Clusters in $\cA(Q_T)$}\}
\]
that sends $T$ to the initial cluster in $\cA(Q_T)$ and commutes with flips and mutations.
\item[(2)] If $\cS$ is a closed surface with exactly one puncture, then there is a bijection
\[
x_T:
\Biggl\{\begin{gathered}\text{Tagged arcs in $\cS$ with}\\\text{the same tags as ones in $T$}\end{gathered}\Biggr\}
\rightarrow\{\text{Cluster variables in $\cA(Q_T)$}\}
\]
such that $\Int_T(\gamma)=f(x_T(\gamma))$ for all tagged arcs $\gamma$ in $\cS$ with the same tags as ones in $T$. Moreover, it induces a bijection
\[
x_T:
\Biggl\{\begin{gathered}\text{Tagged triangulations in $\cS$ with}\\\text{the same tags as ones in $T$}\end{gathered}\Biggr\}
\rightarrow\{\text{Clusters in $\cA(Q_T)$}\}.
\]
that sends $T$ to the initial cluster in $\cA(Q_T)$ and commutes with flips and mutations.
\end{itemize}
\end{thm}

The bijections in Theorem \ref{thm:bijection x} also induce the following: If $\cS$ is not a closed surface with exactly one puncture, then $x_T$ in Theorem \ref{thm:bijection x}(1) induces a bijection
\begin{equation}\label{eq:bij}
x_T:\{U\in\bM_{\cS}\mid U\cap T=\emptyset\}\rightarrow\left\{\text{Non-initial cluster monomials in $\cA(Q_T)$}\right\};
\end{equation}
If $\cS$ is a closed surface with exactly one puncture, then $x_T$ in Theorem \ref{thm:bijection x}(2) induces a bijection
\begin{equation}\label{eq:bij 1}
x_T:\bM_{\cS}':=\left\{U\in\bM_{\cS}\ \middle|\ 
\begin{gathered}\text{$U\cap T=\emptyset$, and tags in $U$}\\\text{are the same as ones in $T$}\end{gathered}\right\}
\rightarrow
\Biggl\{\begin{gathered}\text{Non-initial cluster}\\\text{monomials in $\cA(Q_T)$}\end{gathered}\Biggr\}.
\end{equation}
In both cases, $\Int_T(U)=f(x_T(U))$.

\begin{prop}\label{prop:1-punc f}
Assume that $\cS$ is a closed surface with exactly one puncture. Let $T$ be a tagged triangulation of $\cS$. For any non-initial cluster monomials $x$ and $x'$ in $\cA(Q_T)$, if $f(x)=f(x')$, then $x=x'$.
\end{prop}

\begin{proof}
For $U,V\in\bM_{\cS}'$, we know that $U_2=V_2=\emptyset$. By Theorem \ref{thm:main}, if $\Int_T(U)=\Int_T(V)$, then $U=V$. Therefore, the assertion follows from \eqref{eq:bij 1}.
\end{proof}

Finally, we recall a relation between denominator vectors and $f$-vectors in $\cA(Q_T)$.

\begin{thm}[{\cite[Theorem 1.3 and Corollary 3.9]{Y24}}]\label{thm:d=f}
Let $T$ be a tagged triangulation of $\cS$. Then the following are equivalent:
\begin{itemize}
\item[(1)] For all non-initial cluster variables $x$ in $\cA(Q_T)$, $d(x)=f(x)$.
\item[(2)] The marked surface $\cS$ is a closed surface with exactly one puncture, or $T$ has neither loops nor tagged arcs connecting punctures.
\end{itemize}
\end{thm}

Note that the equivalent properties in Theorem \ref{thm:d=f} are also equivalent that $d(x)=f(x)$ for all non-initial cluster monomials $x$ in $\cA(Q_T)$. We are ready to prove Theorems \ref{thm:unique f-vector} and \ref{thm:conj}.

\begin{proof}[Proof of Theorem \ref{thm:unique f-vector}]
If $\cS$ is not a closed surface with exactly one puncture, then the assertion follows from Theorem \ref{thm:unique Int} and \eqref{eq:bij}. If $\cS$ is a closed surface with exactly one puncture, then the assertion follows from Proposition \ref{prop:1-punc f}.
\end{proof}

\begin{proof}[Proof of Theorem \ref{thm:conj}]
Since a tagged triangulation $T$ has no tagged arcs connecting punctures if and only if the associated graph $G_T$ has no edges, Theorem \ref{thm:unique f-vector}(2) always holds if Theorem \ref{thm:d=f}(2) holds. Therefore, if Theorem \ref{thm:d=f}(2) holds, then Theorems \ref{thm:unique f-vector}(1) and \ref{thm:d=f}(1) hold, that is, Lemma \ref{lem:conj non-initial}(2) also holds. Thus the assertion follows from Lemma \ref{lem:conj non-initial}.
\end{proof}

\begin{example}\label{ex:cluster}
For a tagged triangulation $T$ in Example \ref{ex:T Int}, the associated quiver is
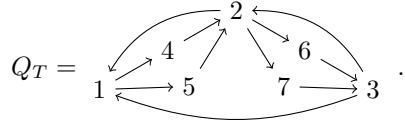
\[
Q_T=
\begin{tikzpicture}[baseline=2mm,scale=0.7]
\node(2)at(90:1.5){$2$};\node(1)at($(2)+(-150:3)$){$1$}; \node(3)at($(2)+(-30:3)$){$3$};
\node(4)at($(1)!0.5!(2)$){$4$};\node(5)at($(4)+(-60:0.8)$){$5$};\node(6)at($(3)!0.5!(2)$){$6$};\node(7)at($(6)+(-120:0.8)$){$7$};
\draw[->](2)to[out=-30,in=-150,relative](1);
\draw[->](1)--(4);\draw[->](4)--(2);\draw[->](1)--(5);\draw[->](5)--(2);
\draw[->](3)to[out=-30,in=-150,relative](2);
\draw[->](2)--(6);\draw[->](6)--(3);\draw[->](2)--(7);\draw[->](7)--(3);
\draw[->](3)to[out=-160,in=-20](1);
\end{tikzpicture}\ .
\]
Then the cluster variable $x_T(\rho(5))$ in $\cA(Q_T)$ is given by
\[
\frac{
\begin{aligned}
&x_1x_2x_3x_6x_7+y_5x_1^2x_3x_6x_7+y_2y_5x_1x_3^2x_4x_5\\
&\hspace{10mm}+y_2y_5y_6x_1x_2x_3x_4x_5+y_2y_5y_7x_1x_2x_3x_4x_5+y_2y_5y_6y_7x_1x_2^2x_4x_5\\
&\hspace{20mm}+y_2y_3y_5y_6y_7x_2x_4x_5x_6x_7+y_1y_2y_3y_5y_6y_7x_2^2x_3x_6x_7+y_1y_2y_3y_5^2y_6y_7x_1x_2x_3x_6x_7
\end{aligned}}
{x_1x_2x_3x_5x_6x_7},
\]
where each initial cluster variable $x_k$ is equal to $x_T(k)$ for $k\in T$. Therefore, we can see that $\Int_T(\rho(5))=f(x_T(\rho(5)))=(1,1,1,0,2,1,1)$. Similarly, we can see that $\Int_T(\rho(4))=f(x_T(\rho(4)))=(1,1,1,2,0,1,1)$, $\Int_T(\rho(7))=f(x_T(\rho(7)))=(1,1,1,1,1,0,2)$, and $\Int_T(\rho(6))=f(x_T(\rho(6)))=(1,1,1,1,1,2,0)$. Thus the cluster monomials $x_T(\rho(4))x_T(\rho(5))$ and $x_T(\rho(6))x_T(\rho(7))$ have the same $f$-vector $(2,2,2,2,2,2,2)$ (cf. Example \ref{ex:graph}).
\end{example}

\section{$\tau$-tilting theory}\label{sec:tau}
\subsection{$\tau$-tilting theory and cluster tilting theory}

First, we recall $\tau$-tilting theory \cite{AIR14}. Let $K$ be an algebraically closed field and $\Lambda$ be a finite dimensional algebra over $K$. We denote by $\mod\Lambda$ (resp., $\proj\Lambda$) the category of finitely generated (resp., finitely generated projective) left $\Lambda$-modules. We denote by $\tau$ the Auslander-Reiten translation in $\mod\Lambda$ and by $|M|$ the number of non-isomorphic indecomposable direct summands of $M\in\mod\Lambda$. We say that $M\in\mod\Lambda$ is
\begin{itemize}
\item \emph{$\tau$-rigid} if $\Hom_{\Lambda}(M,\tau M)=0$;
\item \emph{support $\tau$-tilting} if there is an idempotent $e$ of $\Lambda$ such that $M$ is a $\tau$-rigid $(\Lambda/\langle e \rangle)$-module and $|M|=|\Lambda/\langle e \rangle|$.
\end{itemize}
We denote by $\trigid\Lambda$ (resp., $\itrigid\Lambda$, $\sttilt\Lambda$) the set of all isomorphism classes of $\tau$-rigid (resp., indecomposable $\tau$-rigid, basic support $\tau$-tilting) modules in $\mod\Lambda$. Remark that modules in $\sttilt\Lambda$ are basic, but modules in $\trigid\Lambda$ are not basic.

\begin{thm}[{\cite[Theorem 2.18]{AIR14}\cite[Proposition 5.7]{DF15}}]\label{thm:tau}
Let $M\oplus L\in\sttilt\Lambda$ with indecomposable $L$. Then there is a unique module $N\not\simeq L$ such that it is either indecomposable or zero, and $M\oplus N\in\sttilt\Lambda$. Moreover, we have either $L\in\Fac M$ or $N\in\Fac M$.
\end{thm}

In Theorem \ref{thm:tau}, $M\oplus L$ and $M\oplus N$ are called \emph{mutations} of each other.

Next, we recall cluster tilting theory in $2$-Calabi-Yau triangulated categories \cite{BMRRT06,IY08}. Let $\cC$ be a Hom-finite, Krull-Schmidt, $2$-Calabi-Yau, triangulated category with the suspension functor $\Sigma$. We denote by $\add X$ the category of direct summands of finite direct sums of copies of $X\in\cC$. We say that $X\in\cC$ is
\begin{itemize}
\item \emph{rigid} if $\Hom_{\cC}(X,\Sigma X)=0$;
\item \emph{cluster tilting} if $\add X=\{U\in\cC \mid \Hom_{\cC}(X,\Sigma U)=0\}$.
\end{itemize}
We denote by $\rigid\cC$ (resp., $\irigid\cC$, $\ctilt\cC$) the set of all isomorphism classes of rigid (resp., indecomposable rigid, basic cluster tilting) objects in $\cC$. Remark that objects in $\ctilt\cC$ are basic, but objects in $\rigid\cC$ are not basic. We assume that $\cC$ has cluster tilting objects, that is, $\ctilt\cC\neq\emptyset$. In which case, any maximal rigid object in $\cC$ is cluster tilting \cite[Theorem 2.6]{ZZ11}.

\begin{thm}[{\cite{BIRS09,BMRRT06,IY08}}]\label{thm:ctilt mutation}
Let $R\oplus X\in\ctilt\cC$ with indecomposable $X$. Then there is a unique indecomposable object $Y\not\simeq X$ such that $R\oplus Y\in\ctilt\cC$. Moreover, there are two triangles
\[
X \xrightarrow{f} R_X \xrightarrow{g} Y \rightarrow \Sigma X\ \text{ and }\ 
Y \xrightarrow{f'} R_Y \xrightarrow{g'} X \rightarrow \Sigma Y,
\]
where $f$ (resp., $f'$) is a minimal left $(\add R)$-approximation of $X$ (resp., $Y$) and $g$ (resp., $g'$) is a minimal right $(\add R)$-approximation of $Y$ (resp., $X$). If the quiver of $\End_{\cC}(R\oplus X)$ has neither loops nor $2$-cycles, then its mutation at the vertex corresponding to $X$ is equal to the quiver of $\End_{\cC}(R\oplus Y)$.
\end{thm}

In Theorem \ref{thm:ctilt mutation}, $R\oplus X$ and $R\oplus Y$ are called \emph{mutations} of each other. 

There is a close relationship between cluster tilting theory and $\tau$-tilting theory.

\begin{thm}[{\cite[Theorem 4.1]{AIR14}}]\label{thm:ctilt tau}
Let $C\in\ctilt\cC$ and $\Lambda=\End_{\cC}(C)^{\rm op}$. Then there is a bijection
\[
\overline{(-)}:=\Hom_{\cC}(C,-):\irigid\cC\setminus\{\text{Direct summands of $\Sigma C$}\}\rightarrow\itrigid\Lambda.
\]
Moreover, it induces a bijection
\[
\overline{(-)}:\ctilt\cC\rightarrow\sttilt\Lambda
\]
that sends $C$ (resp., $\Sigma C$) to $\Lambda$ (resp., $0$) and commutes with mutations.
\end{thm}

In the rest of this section, we keep the notations in Theorems \ref{thm:ctilt mutation} and \ref{thm:ctilt tau}. For $M\in\mod\Lambda$, we denote by $[M]$ the corresponding element in the Grothendieck group $K_0(\Lambda)$ of $\mod\Lambda$. All isomorphism classes of simple $\Lambda$-modules form a basis of $K_0(\Lambda)$ and it induces the equivalence $K_0(\Lambda)\simeq\bZ^{|\Lambda|}$. In particular, $[M]\in\bZ^{|\Lambda|}_{\ge 0}$.

\begin{prop}
We keep the notations in Theorems \ref{thm:ctilt mutation} and \ref{thm:ctilt tau}. Then
\begin{align}
[\overline{Y}]&=-[\overline{X}]+\max\left([\overline{R_X}]+[\Cok\overline{\Sigma^{-1}g}],[\overline{R_Y}]+[\Cok\overline{\Sigma^{-1}g'}]\right)\label{eq:YX Sigma}\\
&=-[\overline{X}]+\max\left([\overline{R_X}]+[\Cok\overline{g}],[\overline{R_Y}]+[\Cok\overline{g'}]\right).\label{eq:YX}
\end{align}
\end{prop}

\begin{proof}
Applying $\overline{(-)}$ to the triangles in Theorem \ref{thm:ctilt mutation}, there are exact sequences
\begin{gather*}
0 \rightarrow \Cok\overline{\Sigma^{-1}g} \rightarrow \overline{X} \xrightarrow{\overline{f}} 
\overline{R_X} \xrightarrow{\overline{g}} \overline{Y} \rightarrow \Cok\overline{g} \rightarrow 0,\text{ and}\\
0 \rightarrow \Cok\overline{\Sigma^{-1}g'} \rightarrow \overline{Y} \xrightarrow{\overline{f'}} 
\overline{R_Y} \xrightarrow{\overline{g'}} \overline{X} \rightarrow \Cok\overline{g'} \rightarrow 0,
\end{gather*}
and they induce
\begin{equation}\label{eq:[]}
[\overline{X}]+[\overline{Y}]=[\overline{R_X}]+[\Cok\overline{g}]+[\Cok\overline{\Sigma^{-1}g}]=[\overline{R_Y}]+[\Cok\overline{g'}]+[\Cok\overline{\Sigma^{-1}g'}].
\end{equation}
On the other hand, Theorems \ref{thm:tau} and \ref{thm:ctilt tau} mean that either $\overline{Y}\in\Fac\overline{R}$ or $\overline{X}\in\Fac\overline{R}$ holds, in particular, either $\Cok\overline{g}=0$ or $\Cok\overline{g'}=0$. Similarly, we have either $\Cok\overline{\Sigma^{-1}g}=0$ or $\Cok\overline{\Sigma^{-1}g'}=0$. Therefore, the desired equalities follow from \eqref{eq:[]}.
\end{proof}

\begin{remark}
\begin{enumerate}
\item Assume that the quiver of $\End_{\Lambda}(\overline{R\oplus X})$ has no loops at $X$ and the quiver of $\End_{\Lambda}(\overline{R\oplus Y})$ has no loops at $Y$. Then $\Cok\overline{g}\oplus\Cok\overline{g'}$ gives a labeling in \cite[Definition 2.14]{A20} of the arrow between $\overline{R\oplus X}$ and $\overline{R\oplus Y}$ in the exchange quiver of $\sttilt\Lambda$. The corresponding element in $K_0(\Lambda)$ coincides with the $c$-vector defined in \cite[Subsection 3.4]{A20}.
\item If $\cC$ is a cluster category as in the next subsection and $R\oplus X$ is obtained from $\Sigma T$ by a sequence of mutations, then $X$ corresponds to a cluster variable $x$ and $[\overline{X}]$ coincides with the $f$-vector of $x$. In which case, \eqref{eq:YX} corresponds to the recurrence relation of $f$-vectors (see e.g. \cite[Proposition 2.7]{FGy19}).
\end{enumerate}
\end{remark}

Finally, we consider the following special conditions of $C'\in\ctilt\cC$: Let $\bv\in K_0(\Lambda)$.
{\setlength{\leftmargini}{17mm}
\begin{itemize}
\item[(S1,$\bv$)] $[\overline{\Sigma^{-1}Z}]=[\overline{Z}]+\bv$ for any indecomposable direct summand $Z$ of $C'$.
\item[(S2)] For any indecomposable direct summand $Z$ of $C'$ with $C'=R\oplus Z$, $R_Z$ has exactly two indecomposable direct summands.
\end{itemize}}

\begin{prop}\label{prop:special condition}
We keep the notations in Theorems \ref{thm:ctilt mutation} and \ref{thm:ctilt tau}. Then if $R\oplus X\in\ctilt\cC$ satisfies the conditions (S1,$\bv$) and (S2) for $\bv\in K_0(\Lambda)$, then $R\oplus Y\in\ctilt\cC$ satisfies (S1,$\bv$).
\end{prop}

\begin{proof}
The assumption gives the following equalities:
\begin{align*}
[\overline{\Sigma^{-1}Y}]&=-[\overline{\Sigma^{-1}X}]+\max\left([\overline{\Sigma^{-1}R_X}]+[\Cok\overline{\Sigma^{-1}g}],[\overline{\Sigma^{-1}R_Y}]+[\Cok\overline{\Sigma^{-1}g'}]\right)&&(\text{by }\eqref{eq:YX})\\
&=-[\overline{X}]-\bv+\max\left([\overline{R_X}]+2\bv+[\Cok\overline{\Sigma^{-1}g}],[\overline{R_Y}]+2\bv+[\Cok\overline{\Sigma^{-1}g'}]\right)&&(\text{by (S1,$\bv$), (S2)})\\
&=-[\overline{X}]-\bv+2\bv+\max\left([\overline{R_X}]+[\Cok\overline{\Sigma^{-1}g}],[\overline{R_Y}]+[\Cok\overline{\Sigma^{-1}g'}]\right)\\
&=[\overline{Y}]+\bv.&&(\text{by }\eqref{eq:YX Sigma})
\end{align*}
Thus $R\oplus Y\in\ctilt\cC$ satisfies the condition (S1,$\bv$).
\end{proof}

\subsection{Jacobian algebras and cluster categories}

We recall quivers with potentials and their Jacobian algebras \cite{DWZ08}. Let $Q$ be a quiver without loops. We denote by $\widehat{KQ}$ the complete path algebra of $Q$ with radical-adic topology. A \emph{potential} $W\in\widehat{KQ}$ of $Q$ is a (possibly infinite) linear combination of oriented cycles in $Q$. The pair $(Q,W)$ is called a \emph{quiver with potential} (QP for short). For an oriented cycle $\alpha_1\cdots\alpha_m$ and an arrow $\alpha$ in $Q$, we define
\[
\partial_{\alpha}(\alpha_1\cdots\alpha_m):=\sum_{i : \alpha_i=\alpha}\alpha_{i+1}\cdots\alpha_m\alpha_1\cdots\alpha_{i-1}.
\]
It is extended to the cyclic derivative $\partial_{\alpha}(W)$ of $W$ by linearity and continuously. The \emph{Jacobian ideal} $I(Q,W)$ is the closure, on radical-adic topology, of the ideal of $\widehat{KQ}$ generated by the set $\{\partial_{\alpha}W\mid\alpha\in Q_1\}$. The \emph{Jacobian algebra} $J_{Q,W}$ of $(Q,W)$ is the quotient algebra $\widehat{KQ}/I(Q,W)$. Let $S_i$ be a simple $J_{Q,W}$-module at $i\in Q_0$. Then $\{[S_i]\mid i\in Q_0\}$ forms a basis of $K_0(J_{Q,W})$ and it induces the equivalence $K_0(J_{Q,W})\simeq\bZ^{Q_0}$. For $M\in\mod J_{Q,W}$, the integer vector corresponding to $[M]$ is called its \emph{dimension vector}, denoted by $\udim M$.

One can define the notion of mutations of a QP (see \cite{DWZ08} for the details). We say that a potential $W$ of $Q$ is \emph{non-degenerate} if every quiver obtained from $(Q,W)$ by a sequence of mutations has no $2$-cycles. In which case, $\mu_k(Q,W)=(\mu_kQ,W')$ for some non-degenerate potential $W'$ of $\mu_kQ$. In particular, such $W$ exists if $K$ is uncountable \cite[Corollary 7.4]{DWZ08}. Moreover, a QP $(Q,W)$ gives the Ginzburg differential graded algebra $\Gamma_{Q,W}$ and the generalized cluster category $\cC_{Q,W}$ (see \cite{A09,G06,KY11} for the details). The following means that the observations in the previous subsection can be applied to $\cC_{Q,W}$ when $J_{Q,W}$ is finite dimensional.

\begin{thm}[{\cite[Theorem 3.5]{A09}}]\label{thm:cluster category}
Let $(Q,W)$ be a QP such that $J_{Q,W}$ is finite dimensional. Then $\cC_{Q,W}$ is a Hom-finite, Krull-Schmidt, $2$-Calabi-Yau, triangulated category with a cluster tilting object $\Gamma_{Q,W}$ and $\End_{\cC_{Q,W}}(\Gamma_{Q,W})^{\rm op}\simeq J_{Q,W}$.
\end{thm}

For a QP $(Q,W)$ such that $J_{Q,W}$ is finite dimensional,
\begin{itemize}
\item $\rctilt\cC_{Q,W}$ is the set of all objects in $\ctilt\cC_{Q,W}$ obtained from $\Sigma\Gamma_{Q,W}$ by sequences of mutations;
\item $\ririgid\cC_{Q,W}$ is the subset of $\irigid\cC_{Q,W}$ consisting of indecomposable direct summands of objects in $\rctilt\cC_{Q,W}$;
\item $\rrigid\cC_{Q,W}$ is the subset of $\rigid\cC_{Q,W}$ consisting of finite direct sums of objects in $\ririgid\cC_{Q,W}$.
\end{itemize}
The following is a consequence of a lot of studies about an additive categorification of cluster algebras (e.g. \cite{A09,BIRS09,CK06,BMRRT06,DWZ08,FK10,Pa08}).

\begin{thm}[{\cite[Theorems 6.3 and 6.5]{FK10}}]\label{thm:bijection sC}
Let $(Q,W)$ be a QP such that $W$ is non-degenerate and $J_{Q,W}$ is finite dimensional. Then there is a bijection
\[
\sC_{Q,W}:\ririgid\cC_{Q,W}\rightarrow\{\text{Cluster variables in $\cA(Q)$}\}
\]
such that
\[
\udim\Hom_{\cC_{Q,W}}(\Gamma_{Q,W},X)=f(\sC_{Q,W}(X))
\]
for $X\in\ririgid\cC_{Q,W}$. Moreover, it induces a bijection
\[
\sC_{Q,W}:\rctilt\cC_{Q,W}\rightarrow\{\text{Clusters in $\cA(Q)$}\}
\]
that sends $\Sigma\Gamma_{Q,W}$ to the initial cluster and commutes with mutations.
\end{thm}

Note that when $J_{Q,W}$ is infinite dimensional in Theorem \ref{thm:bijection sC}, a similar result follows from \cite[Corollary 3.5]{IKLP13}, \cite[Subsection 3.3]{Pl11a}, and \cite[Theorem 4.1]{Pl11c}.

\subsection{Jacobian algebras associated with triangulated surfaces}\label{subsec:JT}

Let $T$ be a tagged triangulation of $\cS$. If a Jacobian algebra associated with $Q_T$ is finite dimensional, then the corresponding cluster category has a good property as follows.

\begin{thm}[{\cite[Theorem 1.4]{Y20}}]\label{thm:all rigid}
Let $W$ be a non-degenerate potential of $Q_T$ such that $J_{Q_T,W}$ is finite dimensional.
\begin{itemize}
\item[(1)] If $\cS$ is not a closed surface with exactly one puncture, then $\rigid\cC_{Q_T,W}=\rrigid\cC_{Q_T,W}$.
\item[(2)] If $\cS$ is a closed surface with exactly one puncture, then the suspension functor $\Sigma$ induces a bijection
\[
\Sigma^{-1}:\rrigid\cC_{Q_T,W}\rightarrow\rigid\cC_{Q_T,W}\setminus\rrigid\cC_{Q_T,W}
\]
that sends $\Sigma\Gamma_{Q_T,W}$ to $\Gamma_{Q_T,W}$ and commutes with mutations.
\end{itemize}
\end{thm}

The following theorem is given as a consequence of the classification of non-degenerate potentials of $Q_T$ studied in \cite{GG15,GLS16,GLM22,Lab09,Lab16,Lad12} (see also \cite{Lab16survey}).

\begin{thm}[Finite dimensionality]\label{thm:fin dim}
If $\cS$ is not a closed surface with exactly one puncture, then $J_{Q_T,W}$ is finite dimensional for any non-degenerate potential $W$ of $Q_T$.
\end{thm}

In the rest of this section, we prove Theorems \ref{thm:unique rigid0}, \ref{thm:unique sttilt}, and \ref{thm:unique rigid}.

First, we assume that $\cS$ is not a closed surface with exactly one puncture. Let $W$ be a non-degenerate potential of $Q_T$. Thus $J_{Q_T,W}$ is finite dimensional by Theorem \ref{thm:fin dim}. Theorems \ref{thm:bijection x}(1), \ref{thm:bijection sC}, and \ref{thm:all rigid}(1) induce a bijection
\begin{equation}\label{eq:sX}
\sX_{T,W}:=\sC_{Q_T,W}^{-1}x_{T}:\bM_{\cS}\rightarrow\rigid\cC_{Q_T,W}
\end{equation}
that sends $T$ to $\Sigma\Gamma_{Q_T,W}$ and $\Int_T(U)=\udim\Hom_{\cC_{Q_T,W}}(\Gamma_{Q_T,W},\sX_{T,W}(U))$ for $U\in\bM_{\cS}$. The following theorem was given in \cite{BQ15} under the assumption that $J_{Q_T,W}$ is finite dimensional, but the assumption automatically holds by Theorem \ref{thm:fin dim}.

\begin{thm}[{\cite[Theorem 3.8]{BQ15}}]\label{thm:rotation}
If $\cS$ is not a closed surface with exactly one puncture, then for $U\in\bM_{\cS}$,
\[
\sX_{T,W}(\rho(U))=\Sigma\sX_{T,W}(U).
\]
\end{thm}

 Theorem \ref{thm:ctilt tau} and \eqref{eq:sX} induce a bijection
\[
\sM_{T,W}:=\Hom_{\cC_{Q_T,W}}(\Gamma_{Q_T,W},\sX_{T,W}(-)):\{U\in\bM_{\cS}\mid U\cap T=\emptyset\}\rightarrow\trigid J_{Q_T,W}\setminus\{0\}
\]
that sends $\rho^{-1}(T)$ to $J_{Q_T,W}$ by Theorem \ref{thm:rotation}, and $\Int_T(U)=\udim(\sM_{T,W}(U))$ for $U\in\bM_{\cS}$ with $U\cap T=\emptyset$. Moreover, it induces a bijection
\[
\sM_{T,W}:\{\text{Tagged triangulations of $\cS$}\}\rightarrow\sttilt J_{Q_T,W}
\]
that sends $\rho^{-1}(T)$ (resp., $T$) to $J_{Q_T,W}$ (resp., $0$) and commutes with flips and mutations. Therefore, Theorems \ref{thm:unique rigid0}, \ref{thm:unique sttilt}, and \ref{thm:unique rigid} follow from these observations and Theorems \ref{thm:main}, \ref{thm:unique tag tri}, and \ref{thm:unique Int}, respectively.

Next, we assume that $\cS$ is a closed surface with exactly one puncture. Then $T$ decomposes $\cS$ into only triangle pieces. Each triangle piece $\triangle$ gives rise to an oriented cycle $\cT(\triangle)$ of length three in $Q_T$ up to cyclical equivalence. On the other hand, there is a unique oriented cycle $\cG(T)=\alpha_n\cdots\alpha_1$ in $Q_T$ up to cyclical equivalence such that both $\alpha_i$ and $\alpha_{i+1}$ are not in a common triangle piece for $1\le i\le n$, where $\alpha_{n+1}=\alpha_1$. Then for $\lambda\in K\setminus\{0\}$ and $n\in\bZ_{>0}$, we define a potential of $Q_T$
\begin{equation}\label{eq:potential}
W_T^{\lambda,n}:=\sum_{\triangle}\cT(\triangle)+\lambda\cG(T)^n,
\end{equation}
where $\triangle$ runs over all triangle pieces of $T$. For short, we use the notations
\[
J_T^{\lambda,n}:=J_{Q_T,W_T^{\lambda,n}},\ \Gamma_T^{\lambda,n}:=\Gamma_{Q_T,W_T^{\lambda,n}},\ \text{ and }\ 
\cC_T^{\lambda,n}:=\cC_{Q_T,W_T^{\lambda,n}}.
\]
Note that $W_T^{\lambda,1}$ coincides with a potential introduce in \cite{Lab09}. The following is due to \cite[Theorem 3.1 and Proof of Proposition 3.3]{GLM22}, \cite[Theorem 8.1]{Lab16}, and \cite[Proposition 4.2]{Lad12}.

\begin{thm}[{\cite{GLM22,Lab16,Lad12}}]\label{thm:lambda n}
The potential $W_T^{\lambda,n}$ of $Q_T$ is non-degenerate. Moreover, we assume that the characteristic of $K$ is zero if $n>1$. Then
\begin{itemize}
\item $J_T^{\lambda,n}$ is finite dimensional;
\item $\udim P=(4n,\ldots,4n)$ for any indecomposable projective $J_T^{\lambda,n}$-module $P$.
\end{itemize}
\end{thm}

Note that there is a non-degenerate potential $W$ such that $J_{Q_T,W}$ is infinite dimensional \cite[Proof of Proposition 9.13]{GLS16}.

Now, we assume that the characteristic of $K$ is zero if $n>1$. Thus $J_T^{\lambda,n}$ is finite dimensional by Theorem \ref{thm:lambda n}. Theorems \ref{thm:bijection x}(2), \ref{thm:bijection sC}, and \ref{thm:all rigid}(2) induce bijections
\[
\SelectTips{eu}{}\xymatrix@C=10mm@R=5mm{
\bM_{\cS}^T:=\left\{U\in\bM_{\cS}\mid\text{tags in $U$ are the same as ones in $T$}\right\}\ar[r]^-{\sC_{Q_T,W_T^{\lambda,n}}^{-1}x_{T}}\ar@{<->}[d]_{\rho}&
\rrigid\cC_T^{\lambda,n}\ar[d]^{\Sigma^{-1}}\\
\bM_{\cS}\setminus\bM_{\cS}^T=\left\{U\in\bM_{\cS}\mid\text{tags in $U$ are different from ones in $T$}\right\}&
\rigid\cC_T^{\lambda,n}\setminus\rrigid\cC_T^{\lambda,n},
}
\]
where the left bijection $\rho$ is just an involution changing all tags. These bijections naturally induce a bijection
\begin{equation}\label{eq:sX 1-punc}
\sX_T^{\lambda,n}:\bM_{\cS}\rightarrow\rigid\cC_T^{\lambda,n}
\end{equation}
such that $\sX_T^{\lambda,n}(T)=\Sigma\Gamma_T^{\lambda,n}$ and $\sX_T^{\lambda,n}(\rho(T))=\Gamma_T^{\lambda,n}$.

\begin{lem}\label{lem:S1}
Every object in $\ctilt\cC_T^{\lambda,n}\setminus\rctilt\cC_T^{\lambda,n}$ satisfies the condition (S1,$(4n,\ldots,4n)$).
\end{lem}

\begin{proof}
Let $C\in\ctilt\cC_T^{\lambda,n}$ and $T'$ be a tagged triangulation of $\cS$ such that $\sX_T^{\lambda,n}(T')=C$. Note that $T'$ is obtained from $T$ (resp., $\rho(T)$) by a sequence of flips if and only if $C\in\rctilt\cC_T^{\lambda,n}$ (resp., $C\notin\rctilt\cC_T^{\lambda,n}$).

First, we show that $C$ satisfies (S2). For $C=R\oplus Z$ with indecomposable $Z$, the number of indecomposable direct summands of $R_Z$ is the number of arrows ending at the vertex corresponding to $Z$ in the quiver of $\End_{\cC_T^{\lambda,n}}(C)$. Therefore, we only need to show that the quiver of $\End_{\cC_T^{\lambda,n}}(C)^{\rm op}$ coincides with $Q_{T'}$ since $T'$ decomposes $\cS$ into only triangle pieces.

The quiver of $\End_{\cC_T^{\lambda,n}}(\sX_T^{\lambda,n}(T))^{\rm op}$ coincides with $Q_T$ by Theorem \ref{thm:cluster category}. Then it is follows from Theorem \ref{thm:ctilt mutation} and the compatibility of flips and mutations that the quiver of $\End_{\cC_T^{\lambda,n}}(C)^{\rm op}$ coincides with $Q_{T'}$ if $C\in\rctilt\cC_T^{\lambda,n}$. Moreover, since $Q_{T'}=Q_{\rho(T')}$ and $\End_{\cC_T^{\lambda,n}}(C)\simeq\End_{\cC_T^{\lambda,n}}(\Sigma^{-1}C)$, the quiver of $\End_{\cC_T^{\lambda,n}}(C)^{\rm op}$ coincides with $Q_{T'}$ for any $C$. Therefore, $C$ satisfies (S2).

Finally, $\Gamma_T^{\lambda,n}=\sX_T^{\lambda,n}(\rho(T))$ satisfies (S1,$(4n,\ldots,4n)$) by Theorem \ref{thm:lambda n}. Therefore, if $C\notin\rctilt\cC_T^{\lambda,n}$, then it also satisfies (S1,$(4n,\ldots,4n)$)  by Proposition \ref{prop:special condition}.
\end{proof}

\begin{prop}\label{prop:Intn=dim}
For $U\in\bM_{\cS}$,
\[
\Int^n_T(U)=\udim\Hom_{\cC_T^{\lambda,n}}(\Gamma_T^{\lambda,n},\sX_T^{\lambda,n}(U)).
\]
\end{prop}

\begin{proof}
If $U\in\bM_{\cS}^T$, then
\[
\Int^n_T(U)=\Int_T(U)=f(x_T(U))=\udim\Hom_{\cC_T^{\lambda,n}}(\Gamma_T^{\lambda,n},\sX_T^{\lambda,n}(U)).
\]
If $U\in\bM_{\cS}\setminus\bM_{\cS}^T$, then the desired equality is given as follows:
\begin{align*}
\Int^n_T(U)
&=\Int^n_T(\rho(U))+|U|(4n,\ldots,4n)
&&\text{(by \eqref{eq:Int rho})}\\
&=\udim\Hom_{\cC_T^{\lambda,n}}(\Gamma_T^{\lambda,n},\sX_T^{\lambda,n}(\rho(U)))+|U|(4n,\ldots,4n)
&&\text{(by $\rho(U)\in\bM_{\cS}^T$)}\\
&=\udim\Hom_{\cC_T^{\lambda,n}}(\Gamma_T^{\lambda,n},\Sigma\sX_T^{\lambda,n}(U))+|U|(4n,\ldots,4n)
&&\text{(by the definition of $\sX_T^{\lambda,n}$)}\\
&=\udim\Hom_{\cC_T^{\lambda,n}}(\Gamma_T^{\lambda,n},\sX_T^{\lambda,n}(U))
&&\text{(by Lemma \ref{lem:S1})}.\qedhere
\end{align*}
\end{proof}

By Theorem \ref{thm:ctilt tau}, \eqref{eq:sX 1-punc}, and Proposition \ref{prop:Intn=dim}, there is a bijection
\[
\sM_T^{\lambda,n}:=\Hom_{\cC_T^{\lambda,n}}(\Gamma_T^{\lambda,n},\sX_T^{\lambda,n}(-)):\{U\in\bM_{\cS}\mid U\cap T=\emptyset\}\rightarrow\trigid J_T^{\lambda,n}\setminus\{0\}
\]
such that $\Int^n_T(U)=\udim\sM_T^{\lambda,n}(U)$ for $U\in\bM_{\cS}$ with $U\cap T=\emptyset$. Moreover it induces a bijection
\[
\sM_T^{\lambda,n}:\{\text{Tagged triangulations of $\cS$}\}\rightarrow\sttilt J_T^{\lambda,n}.
\]
Therefore, Theorems \ref{thm:unique rigid0} and \ref{thm:unique sttilt} follow from Theorem \ref{thm:unique Intn} and Corollary \ref{cor:unique tag tri}, respectively. For any $\gamma,\delta\in T$, they form two cycles of length one in $G_T$ and
\[
\udim\sM_T^{\lambda,n}(\rho(\gamma))=\Int^n_T(\rho(\gamma))=\Int^n_T(\rho(\delta))=\udim\sM_T^{\lambda,n}(\rho(\delta)).
\]
Thus Theorem \ref{thm:unique rigid} also holds.

\begin{example}\label{ex:rep}
For a tagged triangulation $T$ in Example \ref{ex:T Int}, we take a non-degenerate potential (see \cite{Lab09})
\[
W_T=-a_3a_2a_1+a_5a_4a_1-b_3b_2b_1+b_5b_4b_1+cb_3b_2a_3a_2
\]
of
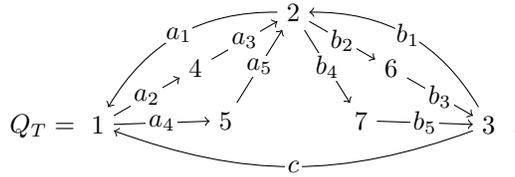
\[
Q_T=
\begin{tikzpicture}[baseline=-1mm]
\node(2)at(90:1.5){$2$};\node(1)at($(2)+(-150:3)$){$1$}; \node(3)at($(2)+(-30:3)$){$3$};
\node(4)at($(1)!0.5!(2)$){$4$};\node(5)at($(4)+(-60:0.8)$){$5$};\node(6)at($(3)!0.5!(2)$){$6$};\node(7)at($(6)+(-120:0.8)$){$7$};
\draw[->](2)to[out=-30,in=-150,relative]node[fill=white,inner sep=0.1]{$a_1$}(1);
\draw[->](1)--node[fill=white,inner sep=1]{$a_2$}(4);\draw[->](4)--node[fill=white,inner sep=1]{$a_3$}(2);
\draw[->](1)--node[fill=white,inner sep=1]{$a_4$}(5);\draw[->](5)--node[fill=white,inner sep=1]{$a_5$}(2);
\draw[->](3)to[out=-30,in=-150,relative]node[fill=white,inner sep=0.1]{$b_1$}(2);
\draw[->](2)--node[fill=white,inner sep=1]{$b_2$}(6);\draw[->](6)--node[fill=white,inner sep=1]{$b_3$}(3);
\draw[->](2)--node[fill=white,inner sep=1]{$b_4$}(7);\draw[->](7)--node[fill=white,inner sep=1]{$b_5$}(3);
\draw[->](3)to[out=-160,in=-20]node[fill=white,inner sep=1]{$c$}(1);
\end{tikzpicture}\ .
\]
Then the associated Jacobian algebra $J_T:=J_{Q_T,W_T}$ is the quotient of the path algebra $KQ_T$ by the ideal generated by
\begin{gather*}
a_1a_5,\ a_4a_1,\ b_1b_5,\ b_4b_1,\ b_3b_2a_3a_2,\ -a_3a_2+a_5a_4,\ -b_3b_2+b_5b_4,\\
-a_1a_3+cb_3b_2a_3,\ -a_2a_1+a_2cb_3b_2,\ -b_1b_3+a_3a_2cb_3,\ -b_2b_1+b_2a_3a_2c.\\
\end{gather*}
The indecomposable projective $J_T$-module $P_4$ is described by the following representation of $Q_T$ (see e.g. \cite{ASS06}):
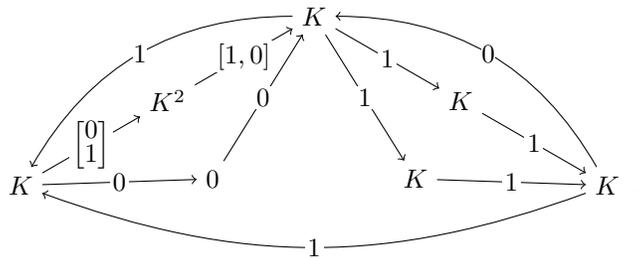
\[
\begin{tikzpicture}[baseline=-1mm,scale=1.5]
\node(2)at(90:1.5){$K$};\node(1)at($(2)+(-150:3)$){$K$}; \node(3)at($(2)+(-30:3)$){$K$};
\node(4)at($(1)!0.5!(2)$){$K^2$};\node(5)at($(4)+(-60:0.8)$){$0$};\node(6)at($(3)!0.5!(2)$){$K$};\node(7)at($(6)+(-120:0.8)$){$K$};
\draw[->](2)to[out=-30,in=-150,relative]node[fill=white,inner sep=0.1]{$1$}(1);
\draw[->](1)--node[fill=white,inner sep=1]{\renewcommand\arraystretch{0.7}$\begin{bmatrix}0\\1\end{bmatrix}$}(4);
\draw[->](4)--node[fill=white,inner sep=1]{$[1,0]$}(2);
\draw[->](1)--node[fill=white,inner sep=1]{$0$}(5);\draw[->](5)--node[fill=white,inner sep=1]{$0$}(2);
\draw[->](3)to[out=-30,in=-150,relative]node[fill=white,inner sep=0.1]{$0$}(2);
\draw[->](2)--node[fill=white,inner sep=1]{$1$}(6);\draw[->](6)--node[fill=white,inner sep=1]{$1$}(3);
\draw[->](2)--node[fill=white,inner sep=1]{$1$}(7);\draw[->](7)--node[fill=white,inner sep=1]{$1$}(3);
\draw[->](3)to[out=-160,in=-20]node[fill=white,inner sep=1]{$1$}(1);
\end{tikzpicture}\ ,
\]
where $1$ denotes the identity. Then $\udim P_4=(1,1,1,2,0,1,1)=\Int_T(\rho^{-1}(4))=\Int_T(\rho(4))$. Similarly, we can see that $\udim P_5=(1,1,1,0,2,1,1)=\Int_T(\rho(5))$, $\udim P_6=(1,1,1,1,1,2,0)=\Int_T(\rho(6))$, and $\udim P_7=(1,1,1,1,1,0,2)=\Int_T(\rho(7))$. Therefore, the basic projective $J_T$-modules $P_4\oplus P_5$ and $P_6\oplus P_7$ have the same dimension vector $(2,2,2,2,2,2,2)$ (cf. Examples \ref{ex:graph} and \ref{ex:cluster}).

Since the Jacobian algebra $J_T=P_1\oplus\cdots\oplus P_7$ is a basic support $\tau$-tilting $J_T$-module with dimension vector $\udim J_T=\Int_T(\rho^{-1}(T))=(9,8,6,7,7,7,7)$, there are no basic support $\tau$-tilting $J_T$-modules with dimension vector $(9,8,6,7,7,7,7)$ by Theorem \ref{thm:unique sttilt}. On the other hand, the $\tau$-rigid $J_T$-module $P_1\oplus P_2\oplus P_3\oplus P_4^{\oplus 2}\oplus P_5^{\oplus 2}$ has the same dimension vector $(9,8,6,7,7,7,7)$.

Finally, we also know that non-basic support $\tau$-tilting $J_T$-modules $J_T\oplus P_4\oplus P_5$ and $J_T\oplus P_6\oplus P_7$ have the same dimension vector $(11,10,8,9,9,9,9)$.
\end{example}

\medskip\noindent{\bf Acknowledgements}.
The author would like to thank Changjian Fu, Yasuaki Gyoda, and Daniel Labardini-Fragoso for useful comments and discussions. He also thank Sota Asai and Jiarui Fei for kindly answering my questions. He was JSPS Overseas Research Fellow and supported by JSPS KAKENHI Grant Numbers JP21K13761.

\bibliographystyle{alpha}
\bibliography{bib}

\newcommand{\etalchar}[1]{$^{#1}$}
\begin{thebibliography}{GLFMO22}

\bibitem[AIR14]{AIR14}
Takahide Adachi, Osamu Iyama, and Idun Reiten.
\newblock $\tau$-tilting theory.
\newblock {\em Compositio Mathematica}, 150(3):415--452, 2014.

\bibitem[Ami09]{A09}
Claire Amiot.
\newblock Cluster categories for algebras of global dimension 2 and quivers
  with potential.
\newblock In {\em Annales de l'institut Fourier}, volume~59, pages 2525--2590,
  2009.

\bibitem[Asa20]{A20}
Sota Asai.
\newblock Semibricks.
\newblock {\em International Mathematics Research Notices},
  2020(16):4993--5054, 2020.

\bibitem[ASS06]{ASS06}
Ibrahim Assem, Daniel Simson, and Andrzej Skowronski.
\newblock {\em Elements of the Representation Theory of Associative Algebras:
  Volume 1: Techniques of Representation Theory}.
\newblock Cambridge University Press, 2006.

\bibitem[BIRS09]{BIRS09}
Aslak~Bakke Buan, Osamu Iyama, Idun Reiten, and Jeanne Scott.
\newblock Cluster structures for 2-calabi--yau categories and unipotent groups.
\newblock {\em Compositio Mathematica}, 145(4):1035--1079, 2009.

\bibitem[BMR{\etalchar{+}}06]{BMRRT06}
Aslak~Bakke Buan, Robert Marsh, Markus Reineke, Idun Reiten, and Gordana
  Todorov.
\newblock Tilting theory and cluster combinatorics.
\newblock {\em Advances in mathematics}, 204(2):572--618, 2006.

\bibitem[BQ15]{BQ15}
Thomas Br{\"u}stle and Yu~Qiu.
\newblock Tagged mapping class groups: Auslander--reiten translation.
\newblock {\em Mathematische Zeitschrift}, 279(3-4):1103--1120, 2015.

\bibitem[CK06]{CK06}
Philippe Caldero and Bernhard Keller.
\newblock From triangulated categories to cluster algebras ii.
\newblock In {\em Annales scientifiques de l'Ecole normale sup{\'e}rieure},
  volume~39, pages 983--1009, 2006.

\bibitem[CK08]{CK08}
Philippe Caldero and Bernhard Keller.
\newblock From triangulated categories to cluster algebras.
\newblock {\em Inventiones mathematicae}, 172(1):169--211, 2008.

\bibitem[CL20]{CL20}
Peigen Cao and Fang Li.
\newblock The enough $g$-pairs property and denominator vectors of cluster
  algebras.
\newblock {\em Mathematische Annalen}, 377(3):1547--1572, 2020.

\bibitem[DF15]{DF15}
Harm Derksen and Jiarui Fei.
\newblock General presentations of algebras.
\newblock {\em Advances in Mathematics}, 278:210--237, 2015.

\bibitem[DWZ08]{DWZ08}
Harm Derksen, Jerzy Weyman, and Andrei Zelevinsky.
\newblock Quivers with potentials and their representations i: Mutations.
\newblock {\em Selecta Mathematica}, 14(1):59--119, 2008.

\bibitem[Fei]{F}
Jiarui Fei.
\newblock Schur rank, compatibility degree, and the canonical decomposition.
\newblock in preparation.

\bibitem[FG06]{FG06}
Vladimir Fock and Alexander Goncharov.
\newblock Moduli spaces of local systems and higher teichm{\"u}ller theory.
\newblock {\em Publications Math{\'e}matiques de l'IH{\'E}S}, 103:1--211, 2006.

\bibitem[FG09]{FG09}
Vladimir Fock and Alexander Goncharov.
\newblock Cluster ensembles, quantization and the dilogarithm.
\newblock {\em Annales scientifiques de l'{\'E}cole Normale Sup{\'e}rieure},
  42(6):865--930, 2009.

\bibitem[FG19a]{FG19}
Changjian Fu and Shengfei Geng.
\newblock On indecomposable $\tau$-rigid modules for cluster-tilted algebras of
  tame type.
\newblock {\em Journal of Algebra}, 531:249--282, 2019.

\bibitem[FG19b]{FGy19}
Shogo Fujiwara and Yasuaki Gyoda.
\newblock Duality between final-seed and initial-seed mutations in cluster
  algebras.
\newblock {\em SIGMA. Symmetry, Integrability and Geometry: Methods and
  Applications}, 15:040, 2019.

\bibitem[FG22]{FG22}
Changjian Fu and Shengfei Geng.
\newblock Intersection vectors over tilings with applications to gentle
  algebras and cluster algebras.
\newblock {\em arXiv preprint arXiv:2212.11497}, 2022.

\bibitem[FG24a]{FG24a}
Changjian Fu and Shengfei Geng.
\newblock Denominator conjecture for some surface cluster algebras.
\newblock {\em arXiv preprint arXiv:2407.11826}, 2024.

\bibitem[FG24b]{FG24b}
Changjian Fu and Shengfei Geng.
\newblock On denominator conjecture for cluster algebras of finite type.
\newblock {\em arXiv preprint arXiv:2409.10914}, 2024.

\bibitem[FK10]{FK10}
Changjian Fu and Bernhard Keller.
\newblock On cluster algebras with coefficients and 2-calabi-yau categories.
\newblock {\em Transactions of the American Mathematical Society},
  362(2):859--895, 2010.

\bibitem[FST08]{FoST08}
Sergey Fomin, Michael Shapiro, and Dylan Thurston.
\newblock Cluster algebras and triangulated surfaces. part i: Cluster
  complexes.
\newblock {\em Acta Mathematica}, 201(1):83--146, 2008.

\bibitem[FT18]{FT18}
Sergey Fomin and Dylan Thurston.
\newblock {\em Cluster algebras and triangulated surfaces Part II: Lambda
  lengths}, volume 255.
\newblock American Mathematical Society, 2018.

\bibitem[FZ02]{FZ02}
Sergey Fomin and Andrei Zelevinsky.
\newblock Cluster algebras i: foundations.
\newblock {\em Journal of the American Mathematical Society}, 15(2):497--529,
  2002.

\bibitem[FZ03]{FZ03}
Sergey Fomin and Andrei Zelevinsky.
\newblock Cluster algebras: Notes for the cdm-03 conference.
\newblock {\em Current developments in mathematics}, 2003(1):1--34, 2003.

\bibitem[FZ07]{FZ07}
Sergey Fomin and Andrei Zelevinsky.
\newblock Cluster algebras iv: coefficients.
\newblock {\em Compositio Mathematica}, 143(1):112--164, 2007.

\bibitem[GGS15]{GG15}
Christof Geiss and Ra{\'u}l Gonz{\'a}lez-Silva.
\newblock Tubular jacobian algebras.
\newblock {\em Algebras and Representation Theory}, 18(1):161--181, 2015.

\bibitem[Gin06]{G06}
Victor Ginzburg.
\newblock Calabi-yau algebras.
\newblock {\em arXiv preprint math/0612139}, 2006.

\bibitem[GLFMO22]{GLM22}
Jan Geuenich, Daniel Labardini-Fragoso, and Jos{\'e}~Luis Miranda-Olvera.
\newblock Quivers with potentials associated to triangulations of closed
  surfaces with at most two punctures.
\newblock {\em S{\'e}minaire Lotharingien de Combinatoire}, B84c:21 pp, 2022.

\bibitem[GLFS16]{GLS16}
Christof Gei{\ss}, Daniel Labardini-Fragoso, and Jan Schr{\"o}er.
\newblock The representation type of jacobian algebras.
\newblock {\em Advances in Mathematics}, 290:364--452, 2016.

\bibitem[GP12]{GP12}
Shengfei Geng and Liangang Peng.
\newblock The dimension vectors of indecomposable modules of cluster-tilted
  algebras and the fomin-zelevinsky denominators conjecture.
\newblock {\em Acta Mathematica Sinica, English Series}, 28(3):581--586, 2012.

\bibitem[GSV05]{GSV05}
Michael Gekhtman, Michael Shapiro, and Alek Vainshtein.
\newblock Cluster algebras and weil-petersson forms.
\newblock {\em Duke Math. J.}, 127(2):291--311, 2005.

\bibitem[GY20]{GY20}
Yasuaki Gyoda and Toshiya Yurikusa.
\newblock $f$-matrices of cluster algebras from triangulated surfaces.
\newblock {\em Annals of Combinatorics}, 24(4):649--695, 2020.

\bibitem[IKLFP13]{IKLP13}
Giovanni~Cerulli Irelli, Bernhard Keller, Daniel Labardini-Fragoso, and
  Pierre-Guy Plamondon.
\newblock Linear independence of cluster monomials for skew-symmetric cluster
  algebras.
\newblock {\em Compositio Mathematica}, 149(10):1753--1764, 2013.

\bibitem[IY08]{IY08}
Osamu Iyama and Yuji Yoshino.
\newblock Mutation in triangulated categories and rigid cohen--macaulay
  modules.
\newblock {\em Inventiones mathematicae}, 172(1):117--168, 2008.

\bibitem[KY11]{KY11}
Bernhard Keller and Dong Yang.
\newblock Derived equivalences from mutations of quivers with potential.
\newblock {\em Advances in Mathematics}, 226(3):2118--2168, 2011.

\bibitem[Lad12]{Lad12}
Sefi Ladkani.
\newblock On jacobian algebras from closed surfaces.
\newblock {\em arXiv preprint arXiv:1207.3778}, 2012.

\bibitem[LF09]{Lab09}
Daniel Labardini-Fragoso.
\newblock Quivers with potentials associated to triangulated surfaces.
\newblock {\em Proceedings of the London Mathematical Society}, 98(3):797--839,
  2009.

\bibitem[LF16a]{Lab16survey}
Daniel Labardini-Fragoso.
\newblock On triangulations, quivers with potentials and mutations.
\newblock {\em Mexican Mathematicians Abroad}, 657:103, 2016.

\bibitem[LF16b]{Lab16}
Daniel Labardini-Fragoso.
\newblock Quivers with potentials associated to triangulated surfaces, part iv:
  removing boundary assumptions.
\newblock {\em Selecta Mathematica}, 22(1):145--189, 2016.

\bibitem[Pal08]{Pa08}
Yann Palu.
\newblock Cluster characters for 2-calabi--yau triangulated categories.
\newblock In {\em Annales de l'institut Fourier}, volume~58, pages 2221--2248,
  2008.

\bibitem[Pla11a]{Pl11a}
Pierre-Guy Plamondon.
\newblock Cluster algebras via cluster categories with infinite-dimensional
  morphism spaces.
\newblock {\em Compositio Mathematica}, 147(6):1921--1954, 2011.

\bibitem[Pla11b]{Pl11c}
Pierre-Guy Plamondon.
\newblock Cluster characters for cluster categories with infinite-dimensional
  morphism spaces.
\newblock {\em Advances in Mathematics}, 227(1):1--39, 2011.

\bibitem[QZ17]{QZ17}
Yu~Qiu and Yu~Zhou.
\newblock Cluster categories for marked surfaces: punctured case.
\newblock {\em Compositio Mathematica}, 153(9):1779--1819, 2017.

\bibitem[Rin11]{R11}
Claus~Michael Ringel.
\newblock Cluster-concealed algebras.
\newblock {\em Advances in Mathematics}, 226(2):1513--1537, 2011.

\bibitem[RS{\etalchar{+}}20]{RS20}
Dylan Rupel, Salvatore Stella, et~al.
\newblock Some consequences of categorification.
\newblock {\em SIGMA. Symmetry, Integrability and Geometry: Methods and
  Applications}, 16:007, 2020.

\bibitem[SZ04]{SZ04}
Paul Sherman and Andrei Zelevinsky.
\newblock Positivity and canonical bases in rank $2$ cluster algebras of finite
  and affine types.
\newblock {\em Moscow Mathematical Journal}, 4(4):947--974, 2004.

\bibitem[Yur19]{Y19}
Toshiya Yurikusa.
\newblock Combinatorial cluster expansion formulas from triangulated surfaces.
\newblock {\em The Electronic Journal of Combinatorics}, 26(2):P2--33, 2019.

\bibitem[Yur20]{Y20}
Toshiya Yurikusa.
\newblock Density of $g$-vector cones from triangulated surfaces.
\newblock {\em International Mathematics Research Notices},
  2020(21):8081--8119, 2020.

\bibitem[Yur24]{Y24}
Toshiya Yurikusa.
\newblock Denominator vectors and dimension vectors from triangulated surfaces.
\newblock {\em Journal of Algebra}, 641:620--647, 2024.

\bibitem[ZZ11]{ZZ11}
Yu~Zhou and Bin Zhu.
\newblock Maximal rigid subcategories in 2-calabi--yau triangulated categories.
\newblock {\em Journal of Algebra}, 348(1):49--60, 2011.

\end{thebibliography}

\end{document}